\documentclass[11pt]{article}
\usepackage[a4paper]{geometry}	

\usepackage{amsmath,amssymb,amsthm}
\usepackage[dvipsnames]{xcolor}
\usepackage{mathtools}
\usepackage{microtype}
\usepackage[hidelinks]{hyperref}
\usepackage{dsfont}
\usepackage{comment}
\usepackage{enumitem}
\usepackage{tikz}
\usetikzlibrary{patterns}
\usepackage{authblk} 
\usepackage{graphicx}
\usepackage{subcaption}
\usepackage{float}
\usepackage{booktabs}

\newcommand{\R}{\mathbb{R}}
\newcommand{\N}{\mathbb{N}}

\newcommand{\ra}{\rightarrow}
\newcommand{\wc}{\rightsquigarrow}
\renewcommand{\d}{\,\mathrm{d}}

\newcommand{\lp}{\left(}
\newcommand{\rp}{\right)}
\newcommand{\lc}{\left[}
\newcommand{\rc}{\right]}
\newcommand{\lacc}{\left\{}
\newcommand{\racc}{\right\}}
\newcommand{\labs}{\left|}
\newcommand{\rabs}{\right|}

\DeclareMathOperator{\oh}{o}
\DeclareMathOperator{\Oh}{O}

\DeclareMathOperator{\EE}{\mathsf{E}}
\DeclareMathOperator{\Var}{\mathsf{Var}}
\renewcommand{\P}{\operatorname{P}}

\newcommand{\PP}{\mathsf{P}}
\newcommand{\1}{\mathbb{I}}

\renewcommand{\mod}{\mathcal{M}}
\newcommand{\param}{\mathcal{R}}
\newcommand{\paramap}{\mathfrak{R}}

\newcommand{\cF}{\mathcal{F}}

\newcommand{\hPhi}{\widehat{\Phi}}
\newcommand{\hQ}{\widehat{Q}}
\newcommand{\tQ}{\widetilde{Q}}
\newcommand{\hF}{\widehat{F}}
\newcommand{\htheta}{\widehat{\theta}}

\definecolor{darkteal}{rgb}{0, 0.35, 0.35}

\definecolor{cerulean}{rgb}{0.0, 0.48, 0.65}

\newcommand{\est}{\widehat{r}_n}
\newcommand{\spacerv}{\mathbb{E}}
\newcommand{\weight}{q}
\newcommand{\dotp}{\boldsymbol{\cdot}}
\newcommand{\ts}{T_n}
\newcommand{\quantile}{\mathcal{Q}}

\newcommand{\egdef}{\stackrel{\mathrm{def}}{=}}
\newcommand{\weightt}{\eta}
\newcommand{\eps}{\varepsilon}
\renewcommand{\epsilon}{\eps}
\newcommand{\tW}{\widetilde{W}}

\theoremstyle{plain}
\newtheorem{theorem}{Theorem}

\newtheorem{corollary}{Corollary}
\newtheorem{lemma}{Lemma}
\newtheorem{assumption}{Assumption}
\theoremstyle{remark}
\newtheorem{remark}{Remark}

\newcommand{\bF}{\overline{F}}
\newcommand{\bhF}{\widehat{\overline{F}}}

\title{Testing parametric models for the angular measure\\
for bivariate extremes}
\author[1]{Stéphane Lhaut}
\author[1,2]{Johan Segers}
\affil[1]{\small UCLouvain, LIDAM/ISBA, Voie du Roman Pays 20, 1348 Louvain-la-Neuve, Belgium.
	E-mail: \href{mailto:stephane.lhaut@uclouvain.be}{stephane.lhaut@uclouvain.be}}
\affil[2]{\small KU Leuven, Department of Mathematics, Celestijnenlaan 200B, 3001 Leuven, Belgium. E-mail: \href{mailto:jjjsegers@kuleuven.be}{jjjsegers@kuleuven.be}}
\date{\today}

\begin{document}
	
\maketitle

\begin{abstract}
The angular measure on the unit sphere characterizes the first-order dependence structure of the components of a random vector in extreme regions and is defined in terms of standardized margins. Its statistical recovery is an important step in learning problems involving observations far away from the center. In this paper, we test the goodness-of-fit of a given parametric model to the extremal dependence structure of a bivariate random sample. The proposed test statistic consists of a weighted $L_1$-Wasserstein distance between a nonparametric, rank-based estimator of the true angular measure obtained by maximizing a Euclidean likelihood on the one hand, and a parametric estimator of the angular measure on the other hand. The asymptotic distribution of the test statistic under the null hypothesis is derived and is used to obtain critical values for the proposed testing procedure via a parametric bootstrap. Consistency of the bootstrap algorithm is proved. A simulation study illustrates the finite-sample performance of the test for the logistic and Hüsler--Reiss models. We apply the method to test for the Hüsler--Reiss model in the context of river discharge data.
\end{abstract}

\section{Introduction}
\label{sec:intro}

In various domains, such as meteorology, finance or engineering, multivariate extreme events may induce important perturbations of the system of interest, so that the occurrence probabilities of such events are of main concern. 
Dangerous combinations of wave heights and still water levels could lead to the collapse of a dike at a certain point at the coast~\cite{dehaan1998}.
The financial portfolio selection problem involves maximizing the return on investment subject to risk constraints; if those constraints are expressed in terms of quantities such as the Value-at-Risk or the Expected Shortfall, tail quantiles of the portfolio need to be estimated accurately~\cite{rocco2014}.
Floating ystems in an offshore environment are submitted to natural forces such as wind, waves and currents, and may be sensible to certain combinations of large values of these forces, the probabilities of which need to be assessed~\cite{morton1996}.

Extreme value theory~\cite{dHF2006, resnick1987} provides a convenient and solid mathematical framework to answer those questions. In the multivariate case, the classical assumption is that the random vector of interest $X = (X_1,\ldots,X_d)$, with values in $\R^d$, belongs to the \emph{maximal domain of attraction} of a multivariate extreme value distribution. This hypothesis comprises two parts:
\begin{enumerate}
	\item the marginal distributions of $X$ belong to the maximal domains of attraction of some univariate extreme value distributions;
	\item after marginal transformation, through the probability integral transform or a variation thereof, the joint distribution of the transformed vector belongs to the maximal domain of attraction of a multivariate extreme value distribution with pre-specified margins.
\end{enumerate}
Under the side assumption that the marginal cumulative distribution functions of $X$ are continuous, which we make in this paper, point~2 above only involves the copula of $X$ in view of Sklar's theorem. Point~2 can be imposed on its own, that is, independently of the assumptions on the margins in point~1, and this is what we will do in this paper.

\paragraph*{Angular measure.}
The multivariate extreme value distribution to which the distribution in point~2 is attracted is fully determined by an \emph{angular measure}, denoted here by $\Phi$, which is a measure on the intersection between the positive orthant $[0,\infty)^d$ with the unit sphere with respect to a given norm. This measure, originally called spectral measure in~\cite{dehaan1977}, describes the first-order dependence structure of joint extremes of $X$ and is rooted in the theory of multivariate regular variation~\cite[Chapter~5]{resnick1987}.

Inference on the angular measure plays a major role in various problems related to extremes. Without being exhaustive, we mention the following examples: estimating the probability of a failure set~\cite{dehaan1998,dehaan1999}, estimating extreme quantile regions~\cite{einmahl2013}, anomaly detection~\cite{goix2017}, approximating conditional densities of an element of a random vector given that the others are large~\cite{cooley2012}, and binary classification in extreme regions~\cite{jalalzai2018}.

For a given dimension~$d$, the collection of all angular measures does not form a parametric family. Hence, a natural way to reduce the complexity of the inference problem is to \emph{assume} a parametric form for $\Phi$. Many models have been proposed, sometimes expressed in terms of other objects, such as the stable tail dependence function: the original bivariate logistic model~\cite{gumbel1960}, the asymmetric bivariate logistic model~\cite{tawn1988}, the tilted Dirichlet model~\cite{coles1991}, the pairwise beta model~\cite{cooley2010}, the Hüsler--Reiss model~\cite{husler1989, engelke2015}, among others.
For conclusions within the proposed models to be reliable, the parametric assumption has to be tested itself. 

\paragraph*{Testing parametric models for extremal dependence.}
Goodness-of-fit tests for parametric models of the tail dependence structure have already been proposed in the literature.
A test based on a weighted $L_2$-distance between the tail empirical copula and the estimated copula under the postulated model with estimation of the parameter performed by censored likelihood maximization and critical values computed on the basis of parametric bootstrap is studied in~\cite{dehaan2008}.
In a similar way, the limiting distribution of a test statistic consisting of a $L_2$-distance between the empirical stable tail dependence function and a parametric estimator with parameter estimation performed by the method of moments is derived in \cite{einmahl2008}; computation of critical values is not considered, however.
Another $L_2$-distance type of statistic based on the tail copula is considered in~\cite{bucher2013} with parameter under the null estimated via a minimum distance method and critical values computed on the basis of a multiplier bootstrap.
A different approach is proposed in~\cite{can2015} where authors develop asymptotically distribution-free tests based on various distances between a semi-parametric estimator of the tail copula and the estimated tail copula under the postulated model. Critical values are computed by simulating directly from limit law of the test statistic, which is distribution-free, that is, does not depend on unknown quantities. The computation of the test statistic itself is more involved.

In this paper, we propose a new test based on the angular measure directly. Our nonparametric estimator is a variation on the empirical angular measure, on which certain moment constraints are enforced by means of a Euclidean empirical likelihood in the spirit of \cite{decarvalho2013}, an idea going back to \cite{owen1991empirical}. The discrepancy between this nonparametric estimator and the one estimated within the parametric model is quantified by a weighted version of the $L_1$-\emph{Wasserstein distance} coming from optimal transport~\cite{Villani2003, Villani2009}. The Wasserstein distance has a long history in hypothesis testing in the non-extreme setting, especially in the one-dimensional situation where it admits an explicit representation, see, e.g., the survey paper~\cite{panaretos2019} for a review of its main statistical applications.
Working with bivariate extremes, the angular measure is defined on a one-dimensional space, so that we may use the explicit formula to compute our test statistic. We derive its limiting distribution under the null hypothesis and propose a consistent method to estimate the associated critical values. 

We only consider the bivariate case, the extension to higher dimensions being far from trivial, because, besides consistency \cite{einmahl2001} and \cite{janssen2020kmeans}, even the limit distribution of the empirical angular measure is not known in higher dimension. In the application in Section~\ref{sec:application}, we apply the test to several pairs of variables at once, correcting the $p$-values in order to account for the multiple hypothesis problem.

\paragraph*{Outline and reading guide.}
In Section~\ref{sec:rv}, we describe the underlying mathematical framework and introduce the maximum empirical Euclidean estimator of the angular measure, including its asymptotic expansion and distribution in Corollary~\ref{cor:asymptoticsMEL}. The goodness-of-fit test is formalized in Section~\ref{sec:GoF} together with our two main results: Theorem~\ref{thm:asymptoticTS} states the 
asymptotic distribution of the test statistic under general conditions on the estimator of the unknown model parameters, and Theorem~\ref{thm:cont_quantiles} ensures that the critical values can be estimated consistently from the asymptotic distribution of the test statistic at the estimated parameter values.
The finite-sample performance of the novel procedure is evaluated in Section~\ref{sec:simulations} for the logistic and the Hüsler--Reiss models, with favorable comparisons to the methods proposed in \cite{dehaan2008} and \cite{can2015}. Finally, the goodness-of-fit of the Hüsler--Reiss model is tested for pairwise extremes of river discharge data in the Danube network in Section~\ref{sec:application}. Section~\ref{sec:conclusion} concludes the paper. All proofs are relegated to the appendices.

As the theoretical and methodological set-up is a bit complex, we offer here some pointers to assist reading Sections~\ref{sec:rv} and~\ref{sec:GoF}:
\begin{itemize}
	\item The normalized angular probability measure $Q_p$ is defined in~\eqref{eq:angular_prob_meas} and its maximum empirical Euclidean estimator $\tQ_p$ in~\eqref{eq:max_emp_lik_est} in Section~\ref{sec:rv:mele}.
	\item The statistic $\ts$ for testing the goodness-of-fit of a parametric extremal dependence model is defined in \eqref{eq:test_statistics} in Section~\ref{sec:GoF:ts}.
	\item The asymptotic null distribution of $\ts$ is stated in Theorem~\ref{thm:asymptoticTS}, while the consistency of critical values computed upon the test statistic's limiting distribution at estimated parameter values is validated in Theorem~\ref{thm:cont_quantiles} in Section~\ref{sec:GoF:crit}.
	\item Sections~\ref{sec:rv:rv} and \ref{sec:rv:ang} provide some necessary background but can be skipped by informed readers. Section~\ref{sec:rv:asy} provides the large-sample theory for $\tQ_p$ and can be skipped by readers who are mainly interested in the goodness-of-fit test itself.
	\item Section~\ref{sec:GoF:param} is a necessary digression on the asymptotics of parameter estimators of extremal dependence models and can be skipped upon first reading too.
\end{itemize}

\section{Regular variation and the angular measure}
\label{sec:rv}

\subsection{Regular variation}
\label{sec:rv:rv}

Let $F_1$ and $F_2$ denote the marginal cumulative distribution functions of $X_1$ and $X_2$, respectively, assumed to be continuous.
As in \cite{einmahl2001,einmahl2009maximum}, our working hypothesis is that the distribution of the standardized random vector
\[
	U = (U_1,U_2) = (1-F_1(X_1), 1-F_2(X_2)),
\]
is \emph{multivariate regularly varying} in the sense that there exists a Radon measure $\Lambda$ on $\spacerv \egdef [0,\infty]^2 \setminus \{(\infty,\infty)\}$ such that
\begin{equation}
\label{eq:Lambda}
	t^{-1} \PP \lc U/t \in \cdot \, \rc \xrightarrow{v} \Lambda(\,\cdot\,), \qquad t \ra 0,
\end{equation}
where $\xrightarrow{v}$ denotes \emph{vague convergence} of measures, see~\cite[Section~3.4]{resnick1987}. 
This is equivalent to assume that $U$ is in the minimal domain of attraction of a min-stable distribution with standardized margins~\cite[Proposition~5.17]{resnick1987}.
In terms of sets, vague convergence is equivalent to the following property: for any Borel set $B \subset \spacerv$ bounded away from the point $(\infty,\infty)$ and with $\Lambda(\partial B) = 0$, we have
\begin{equation}
\label{eq:regular_variation_sets}
	\lim_{t \ra 0} t^{-1} \PP \lc U \in tB \rc = \Lambda(B),
\end{equation}
where $tB = \{(tx,ty): (x,y) \in B\}$.

\begin{remark}[Link with the exponent measure]
Usually, multivariate regular variation is expressed through the standardized vector $(Z_1,Z_2) = (1/U_1,1/U_2)$ with Pareto margins instead of uniform ones, and the vague convergence takes place on $\spacerv^{-1} = [0,\infty]^2 \setminus \{(0,0)\}$ instead. The limiting measure obtained through this standardization is often denoted by $\nu$ and is referred to as the \emph{exponent measure}. Here, we work with $\Lambda$ for convenience, but it is clear that both measures are linked by the one-to-one relation
\begin{equation}
\label{eq:nu}
	\nu = \iota_\# \Lambda = \Lambda \circ \iota^{-1} \qquad \text{and} \qquad 
	\Lambda = \iota^{-1}_\# \nu = \nu \circ \iota,
\end{equation}
where $\iota: (x_1,x_2) \in \spacerv \mapsto  \iota(x_1,x_2) = (x_1^{-1},x_2^{-1}) \in \spacerv^{-1}$ is the inverse map and where the subscript $\#$ indicates the push-forward measure.
\end{remark}

Below, we list some properties of the measure $\Lambda$ which will be useful later:
\begin{itemize}
	\item $\Lambda$ has Lebesgue margins: for any $0 \leq u < \infty$,
	\begin{equation}
	\label{eq:uniform_margins}
		\Lambda([0,u] \times [0,\infty]) = \Lambda([0,\infty] \times [0,u]) = u.
	\end{equation}
	\item $\Lambda$ is homogeneous: for any $c > 0$ and any Borel set $B \subset \spacerv$,
	\begin{equation}
	\label{eq:homogeneity}
		\Lambda(cB) = c \Lambda(B).
	\end{equation}
\end{itemize}

\subsection{Angular measure}
\label{sec:rv:ang}

The angular measure is a finite measure derived from the exponent measure and with support contained in the intersection of the unit sphere and the positive orthant $[0,\infty)^2$. Even though the angular measure can be defined with respect to any norm, we will only consider $L_p$-norms, the ones most used in practice, which are defined, for $p \in [1,\infty]$ and $x = (x_1,x_2) \in \R^2$, by
\begin{equation}
\label{eq:p-norms}
	\|x\|_p \egdef
	\begin{dcases}
		(|x_1|^p+|x_2|^p)^{1/p} & \text{if $1 \leq p < \infty$,} \\
		\max(x_1,x_2) & \text{if $p = \infty$.}
	\end{dcases}
\end{equation}
For $p \in [1,\infty]$, we will let $\Phi_p$ denote the \emph{angular measure} associated to the measure $\Lambda$ in \eqref{eq:Lambda}, defined on any Borel set $A \subset [0,\pi/2]$ by
\begin{equation}
\label{eq:ang_meas_def}
	\Phi_p(A) \egdef \Lambda \lp \{(x_1,x_2) \in (0,\infty]^2: \|(x_1^{-1},x_2^{-1})\|_p \geq 1, \arctan(x_2/x_1) \in A \} \rp.
\end{equation}
The inverses appearing in the radial component come from the fact that $\Phi_p$ is usually defined in terms of $\nu$ in \eqref{eq:nu} rather than $\Lambda$.
If, for $\theta \in [0,\pi/2]$, we introduce the sets
\begin{equation}
\label{eq:Cptheta}
	C_{p,\theta} \egdef
	\begin{dcases}
		([0,\infty] \times \{0\}) \cup (\{\infty\} \times [0,1]) 
		&\text{if } \theta = 0, \\
		\lacc (x,y) : 0 \leq x \leq \infty, 0 \leq y \leq \min\{x \tan\theta, y_p(x)\} \racc &\text{if } 0 < \theta < \pi/2, \\
		\lacc (x,y) : 0 \leq x \leq \infty, 0 \leq y \leq y_p(x) \racc &\text{if } \theta = \pi/2,
	\end{dcases}
\end{equation}
where
\[
	y_p(x) \egdef
	\begin{dcases}
		\infty &\text{if } x \in [0,1), \\
		\lp 1 + \frac{1}{x^p-1} \rp^{1/p} &\text{if } x \in [1,\infty] \text{ and } p \in [1,\infty), \\
		1 &\text{if } x \in [1,\infty] \text{ and } p = \infty,
	\end{dcases}
\]
then it follows that 
\[
	\Phi_p(\theta) \egdef \Phi_p([0,\theta]) = \Lambda(C_{p,\theta}).
\]
For $x \geq 1$, $y_p(x)$ is the smallest value of $y \geq 1$ solving the equation $\|(x^{-1}, y^{-1})\|_p = 1$, while $x\tan\theta < y_p(x)$ if and only if $x < x_p(\theta) \egdef \|(1,\cot\theta)\|_p$.
We illustrate and relate these quantities in Figure~\ref{fig:C_p} in case $p=1$ and $0 < \theta < \pi/2$. 

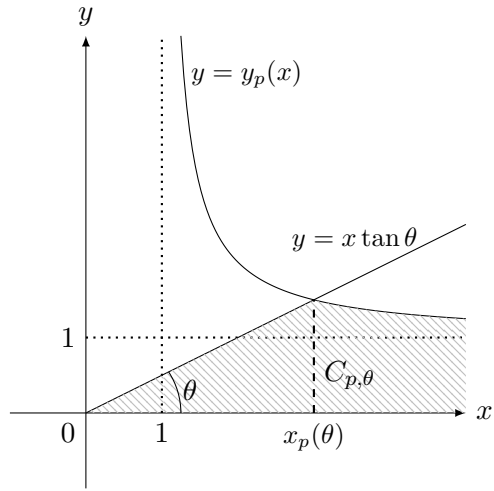
\begin{figure}[h]
	\centering
	\begin{tikzpicture}[scale=1]
		\draw[->,>=latex] (-1,0) -- (5,0);
		\draw (5,0) node[right]{$x$};
		\draw[->,>=latex] (0,-1) -- (0,5);
		\draw (0,5) node[above]{$y$};
		\draw[dotted,thick] (1,0) -- (1,5);
		\draw (1,0) node[below]{$1$};
		\draw[dotted,thick] (0,1) -- (5,1);
		\draw (0,1) node[left]{$1$};
		\draw (0,0) node[below left]{$0$};
		\draw [domain=0:5] plot(\x,{\x/2});
		\draw (4.5,2.3) node[left]{\small $y=x\tan\theta$}; 
		\draw [domain=1.25:5,samples=100] plot(\x,{1+1/(\x - 1)}); 
		\draw (1.25,4.5) node[right]{\small $y=y_p(x)$};
		\fill [pattern=north west lines, pattern color=lightgray] (0,0) -- plot [domain=0:3] (\x, \x/2) -- plot [domain=3:5] (\x,{1+1/(\x - 1)}) -- (5,0) -- cycle;
		\draw (3,0.5) node[right]{$C_{p,\theta}$};
		\draw[dashed, thick] (3,0) -- (3,1.5);
		\draw (3,0) node[below]{\small $x_p(\theta)$};
		\draw (1.25,0) arc (0:33:1);
		\draw (1.6,0.3) node[left]{$\theta$};
	\end{tikzpicture}
	\caption{The set $C_{p,\theta}$ in \eqref{eq:Cptheta} and related quantities for $p=1$ and $0 < \theta < \pi/2$.}
	\label{fig:C_p}
\end{figure}

The Lebesgue margins of $\Lambda$ in~\eqref{eq:uniform_margins} imply constraints on $\Phi_p$: necessarily
\begin{equation}
\label{eq:margins_phi}
	\int_0^{\pi/2} \frac{\sin\theta}{\|(\sin\theta,\cos\theta)\|_p} \d\Phi_p(\theta)
	= 1
	= \int_0^{\pi/2} \frac{\cos\theta}{\|(\sin\theta,\cos\theta)\|_p} \d\Phi_p(\theta).
\end{equation}
It will be convenient to work with the \emph{angular probability measure} $Q_p$ defined for Borel sets $A \subset [0,\pi/2]$ by
\begin{equation}
\label{eq:angular_prob_meas}
	Q_p(A) \egdef \frac{\Phi_p(A)}{\Phi_p([0,\pi/2])},
\end{equation}
for which the marginal constraints~\eqref{eq:margins_phi} become
\begin{equation}
\label{eq:margins_Q}
	\int_0^{\pi/2} \frac{\sin\theta}{\|(\sin\theta,\cos\theta)\|_p} \d Q_p(\theta)
	= m(Q_p)
	= \int_0^{\pi/2} \frac{\cos\theta}{\|(\sin\theta,\cos\theta)\|_p} \d Q_p(\theta),
\end{equation}
where $m(Q_p) \egdef \Phi_p([0,\pi/2])^{-1}$.

\subsection{Maximum empirical Euclidean likelihood estimator}
\label{sec:rv:mele}

Given is an independent random sample $\{(X_{i1}, X_{i2}): i = 1,\ldots, n\}$ with the same distribution as $(X_1,X_2)$. We would like to obtain a nonparametric estimator of $\Phi_p$ or, equivalently, $Q_p$. One way to proceed is to observe that, by~\eqref{eq:regular_variation_sets} and~\eqref{eq:ang_meas_def}, for any Borel set $A \subset [0,\pi/2]$ whose topological boundary is a $\Phi_p$-null set, we have, writing $\bF_j = 1 - F_j$,
\[
	\Phi_p(A) = \lim_{t \ra 0} t^{-1} \PP \lc \left\| \lp \frac{1}{\bF_1(X_1)}, \frac{1}{\bF_2(X_2)} \rp \right\|_p \geq t^{-1}, \; \arctan \left\{ \frac{\bF_2(X_2)}{\bF_1(X_1)} \right\} \in A \rc.
\]
Let $k = k(n)$ be such that $k \ra \infty$ and $k/n \ra 0$ as $n \ra \infty$. Replacing $t$ by $k/n$, the law of $(X_1,X_2)$ by its empirical counterpart on the sample and estimating the marginal distribution functions by $\hF_j(x_j) = n^{-1} \sum_{i=1}^n \1 \{X_{ij} < x_j\}$ with survival functions $\bhF_j = 1 - \hF_j$ for $j=1,2$,  leads to the \emph{empirical angular measure} with distribution function
\begin{align}
	\widehat{\Phi}_p(\theta) 
	&\egdef \frac{n}{k} \frac{1}{n} \sum_{i=1}^n \1 \left\{ \left\| \lp \frac{1}{\bhF_1(X_{i1})}, \frac{1}{\bhF_2(X_{i2})} \rp \right\|_p \geq \frac{n}{k}, \; \arctan \left\{ \frac{\bhF_2(X_{i2})}{\bhF_1(X_{i1})} \right\} \leq \theta \right\} \nonumber \\
	&= \frac{1}{k} \sum_{i=1}^n \1 \Big\{ (n+1-R_{i1})^{-p} + (n+1-R_{i2})^{-p} \geq k^{-p}, \; \frac{n+1-R_{i2}}{n+1-R_{i1}} \leq \tan\theta \Big\}, \label{eq:empirical_ang_meas}
\end{align}
for $\theta \in [0,\pi/2]$; here, $R_{ij}$ denotes the rank of $X_{ij}$ in the marginal sample $X_{1j}, \ldots, X_{nj}$.
This estimator has been studied from the asymptotic point of view in~\cite{einmahl2001,einmahl2009maximum} and from a finite-sample perspective (in general dimension) in~\cite{clemencon2023}.

The estimator $\widehat{\Phi}_p$ is not guaranteed to satisfy the marginal constraints~\eqref{eq:margins_phi} and is therefore not a \emph{true} angular measure itself. One solution is to enforce the constraints by maximizing an empirical Euclidean likelihood as done in~\cite{decarvalho2013} for the case $p=1$. In particular, they show that this procedure improves finite-sample performance of the estimator. 
Here, we extend their approach to the general case $p \in [1,\infty]$ to make use of the resulting estimator in our test.

Let
\[
	K \egdef \sum_{i=1}^n \1 \lacc \left\| \lp \frac{1}{\bhF_1(X_{i1})}, \frac{1}{\bhF_2(X_{i2})} \rp \right\|_p \geq \frac{n}{k} \racc
\]
be the (random) number of data points exceeding $n/k$ on the Pareto scale and denote these points by $(X_{i_j,1}, X_{i_j,2})$ for $j \in \{1,\ldots,K\}$. 
The collection of angles associated to those points is
\begin{equation}
\label{eq:hthetaj}
	\htheta_j \egdef \arctan \lacc \frac{\bhF_2(X_{i_j,2})}{\bhF_1(X_{i_j,1})} \racc, \qquad j \in \{1,\ldots,K\}.
\end{equation}
It is then easily seen from~\eqref{eq:empirical_ang_meas} that for any $\theta \in [0,\pi/2]$,
\[
	\hQ_p(\theta) \egdef \frac{\hPhi_p(\theta)}{\hPhi_p(\pi/2)} = \sum_{j=1}^K K^{-1} \1\{\htheta_j \leq \theta\}.
\]
Following in the footsteps of~\cite{decarvalho2013}, we suggest to modify this estimator by introducing the \emph{maximum (empirical) Euclidean likelihood estimator}
\begin{equation}
\label{eq:max_emp_lik_est}
    \tQ_p(\theta) \egdef \sum_{j=1}^{K} \widehat{p}_j \1 \lacc \htheta_j \leq \theta \racc, 
    \qquad \theta \in [0,\pi/2],
\end{equation}
with $\htheta_j$ as in~\eqref{eq:hthetaj} and where $\widehat{p} = (\widehat{p}_1,\ldots,\widehat{p}_K)$ is the solution the optimization problem
\begin{equation}
\label{eq:EuclLik}
	\min_{p \in \R^K} \frac{1}{2} \sum_{j=1}^K (Kp_j -1)^2
\end{equation}
under the constraints
\[
\left\{
\begin{array}{@{}l@{\;=\;}l}
	\sum_{j=1}^K p_j & 1, \\[1ex]
	\sum_{j=1}^K p_j f(\htheta_j) & 0,
\end{array}
\right.
\]
where we wrote
\begin{equation}
\label{eq:def_f}
	f(\theta) \egdef \frac{\sin\theta - \cos\theta}{\|(\sin\theta, \cos\theta)\|_p},
	\qquad \theta \in [0,\pi/2].
\end{equation}
The constraints ensure that the estimated measure has total mass equal to one and that~\eqref{eq:margins_Q} is satisfied. The objective function in \eqref{eq:EuclLik} is a second-order approximation to a negative empirical log-likelihood known as an empirical Euclidean log-likelihood \cite{owen1991empirical}, in view of the use of the Euclidean distance between the probability mass vector $(p_1,\ldots,p_K)$ and the uniform distribution $(1/K,\ldots,1/K)$.
The above optimization problem can be solved by the method of Lagrange multipliers, with solution
\[
	\widehat{p}_j = \frac{1}{K} \lacc 1 - \frac{\overline{f_{\htheta}}}{\sigma_{f_{\htheta}}^2} \lp f(\htheta_j) - \overline{f_{\htheta}} \rp \racc,
	\qquad j \in \{1,\ldots,K\},
\]
where we used the notation
\[
	\overline{f_{\htheta}} \egdef \frac{1}{K} \sum_{j=1}^K f(\htheta_j) \quad \text{and} \quad
	\sigma_{f_{\htheta}}^2 \egdef  \frac{1}{K} \sum_{j=1}^K \lp f(\htheta_j) - \overline{f_{\htheta}} \rp^2.
\]
Note that the weights $\widehat{p}_j$ may be negative, but that the set on which this happens has probability tending to zero, see~\cite[page~5]{decarvalho2013}. This follows from the facts that
\[
	\overline{f_{\htheta}} \xrightarrow{\PP} \mu_{Q_p}(f) = 0 \quad \text{and} \quad
	\sigma_{f_{\htheta}}^2 \xrightarrow{\PP} \sigma_{Q_p}^2(f) > 0
\]
as $n \ra \infty$, where we used the notation in~\eqref{eq:mu_var}, together with $\labs f(\htheta_j) - \overline{f_{\htheta}} \, \rabs \leq 2$.

\subsection{Asymptotic expansion and distribution}
\label{sec:rv:asy}

The asymptotic distribution of the empirical process
\begin{equation}
\label{eq:mel_process}
	\lacc \sqrt{k} \lp \tQ_p(\theta) - Q_p(\theta) \rp: \theta \in [0,\pi/2] \racc
\end{equation}
associated to the maximum Euclidean likelihood estimator derived in the previous section can be obtained on the basis of the one of the process $\sqrt{k}(\hPhi_p-\Phi_p)$ associated to the empirical angular measure. The latter was analyzed in~\cite{einmahl2001} for the case $p=\infty$ and in~\cite{einmahl2009maximum} for general $p \in [1,\infty]$. Based on those results, an asymptotic expansion was also derived in~\cite{lhaut2023asymptotic} which  will be needed to derive the limiting distribution of our test statistic as the sample size grows to infinity. First we recall the assumptions and the precise statement and then we derive an asymptotic expansion for the process~\eqref{eq:mel_process}.

\begin{assumption}
\label{ass:smoothness}
The measure $\Lambda$ is absolutely continuous with respect to the Lebesgue measure with a density $\lambda$ that is continuous on $[0,\infty)^2 \setminus \{(0,0)\}$.
Furthermore, $\Lambda(\{\infty\} \times [0,1]) = \Lambda([0,1] \times \{\infty\}) = 0$.
\end{assumption}

Assumption~\ref{ass:smoothness} implies that $\Phi_p$ is concentrated on $(0,\pi/2)$, thereby excluding asymptotic independence. Indeed, a calculation shows that $\Phi_p(\{0\}) = \Lambda(\{\infty\} \times [0,1])$ and $\Phi_p(\{\pi/2\}) = \Lambda([0,1] \times \{\infty\})$, both of which are zero by Assumption~\ref{ass:smoothness}. 
It also implies that the map $\theta \in (0,\pi/2) \mapsto \Phi_p(\theta)$ has a continuous derivative $\varphi_p$ on $(0,\pi/2)$. In particular, $\lambda$ and $\varphi_p$ are related through the following expression: for every $x,y>0$,
\begin{equation}
	\label{eq:densities_relation}
	\lambda(x,y) = \frac{xy}{x^2+y^2} \|(x,y)\|_p^{-1} \varphi_p(\arctan(y/x)).
\end{equation}
The proof of this identity can be found in~\cite[Appendix~A]{lhaut2023asymptotic}. Note that relation~\eqref{eq:densities_relation} implies
\begin{equation}
	\label{eq:lambda_homogen}
	\lambda(ax,ay) = a^{-1} \lambda(x,y), \qquad a > 0, \ (x,y) \in [0,\infty)^2 \setminus \{(0,0)\}.
\end{equation} 
Evaluating expression~\eqref{eq:densities_relation} at $(x,y) = (\cos\theta,\sin\theta)$ for $\theta \in (0/\pi/2)$ leads to the useful formula
\begin{equation}
\label{eq:phi_from_lambda}
	\varphi_p(\theta) 
	= \frac{\|(\cos\theta,\sin\theta)\|_p}{\cos\theta\sin\theta} \, \lambda(\cos\theta,\sin\theta).
\end{equation}

\begin{assumption}
\label{ass:bias}
If $c$ denotes the density of $U = (1-F_1(X_1),1-F_2(X_2))$, the quantity
\[
	\mathcal{D}_T(t) \egdef \iint_{\mathcal{L}_T} |tc(tu_1,tu_2) - \lambda(u_1,u_2)| \, \d u_1 \d u_2, \quad 1 \leq T < \infty, \, t>0,
\]
where $\mathcal{L}_T \egdef \{(u_1,u_2) \in [0,T]^2 : u_1 \wedge u_2 \leq 1 \}$, satisfies $\mathcal{D}_{1/t}(t) = \Oh(t^{\alpha_1})$ and for any $p \in [1,\infty]$,
\[
	\Phi_p(t) = \Oh(t^{\alpha_2}) \qquad \text{and} \qquad 
	\Phi_p \lp \tfrac{\pi}{2} \rp  - \Phi_p \lp \tfrac{\pi}{2} - t \rp = \Oh(t^{\alpha_3})
\]
for some $\alpha_1,\alpha_2, \alpha_3 > 0$ as $t \ra 0$, which can be taken independently of $p$. Furthermore, $$k = \oh \lp n^{\tfrac{2\alpha}{2\alpha+1}} \rp, \qquad \text{where } \alpha \egdef \min \{\alpha_1,\alpha_2,\alpha_3\}.$$
\end{assumption}

Note in particular that under Assumption~\ref{ass:bias},
\[
	\lim_{n \ra \infty} \sqrt{k} \, \mathcal{D}_{n/k}(k/n)
	= \lim_{n \ra \infty} \sqrt{k} \, \Phi_p \lp \tfrac{k}{n} \rp
	= \lim_{n \ra \infty} \sqrt{k} \, \lacc \Phi_p \lp \tfrac{\pi}{2} \rp  - \Phi_p \lp \tfrac{\pi}{2} - \tfrac{k}{n} \rp \racc = 0.
\]
Assumption~\ref{ass:bias} is key to control the bias between $\Phi_p(\theta) = \Lambda(C_{p,\theta})$ and its pre-asymptotic version $s \PP(sU \in C_{p,\theta})$ for $s>0$ large. Details are to be found in~\cite[Section~3]{lhaut2023asymptotic}. 

Let $W_\Lambda$ be a tight centered Wiener process indexed by Borel sets on $\spacerv$ with ``time'' $\Lambda$ and covariance function $\EE[W_\Lambda(C)W_\Lambda(C')] = \Lambda(C \cap C')$.

We are now ready to formulate Theorem~2 in~\cite{lhaut2023asymptotic}, and, consequently, Theorem~3.1 in~\cite{einmahl2009maximum}. Here and below, the symbol $\wc$ denotes weak convergence in a metric space as in~\cite{VVV1996}. Let $\ell^\infty(\mathcal{F})$ denote the Banach space of bounded functions from a set $\mathcal{F}$ to $\mathbb{R}$, equipped with the supremum norm.

\begin{theorem}
\label{thm:asymptoticsEAM}
Let $P$ be the law of $U = (1-F_1(X_1),1-F_2(X_2))$ and let $P_n$ denote the empirical measure of an independent sample $U_1,\ldots,U_n$ from the law of $U$. For $j \in \{1,2\}$ and $u \in [0,1]$, let the associated marginal empirical cumulative distribution functions be denoted by $\Gamma_{jn}(u) = n^{-1} \sum_{i=1}^n \1\{U_{ij} \leq u\}$. We set $\Gamma_{jn}(u) = u$ for $u > 1$. Let $Q_{jn} = \inf\{x \geq 0: \Gamma_{jn}(x) \geq u\}$ denote the quantile function associated to $\Gamma_{jn}$ and set $Q_{jn}(y) = 0$ for $0 \leq y \leq (2n)^{-1}$. Define the tail marginal empirical processes
\[
	w_{jn}(x) = \sqrt{k} \left\{ \tfrac{n}{k} \Gamma_{jn} (\tfrac{k}{n} x) - x \right\}
	\quad \text{and} \quad
	v_{jn}(x) = \sqrt{k} \left\{ \tfrac{n}{k} Q_{jn} (\tfrac{k}{n} x) - x \right\},
\]
for $j \in \{1,2\}$ and $x\geq 0$.
Under Assumptions~\ref{ass:smoothness} and~\ref{ass:bias}, we have for any $p \in [1,\infty]$
\[
	\sup_{\theta \in [0,\pi/2]} \labs \sqrt{k}\lp \hPhi_p(\theta) - \Phi_p(\theta) \rp - E_{n,p}(\theta) \rabs = \oh_{\PP}(1),
\]
as $n \to \infty$, where
\begin{align*}
	\lefteqn{
		E_{n,p}(\theta) \egdef \sqrt{k} \lacc \tfrac{n}{k} P_n \lp \tfrac{k}{n} C_{p,\theta} \rp - \tfrac{n}{k} P\lp \tfrac{k}{n} C_{p,\theta} \rp \racc 
	} \\
	&\quad + \int_0^{x_p(\theta)} \lambda(x,x\tan\theta) \lacc w_{1n}(x) \tan\theta - w_{2n} \lp x \tan\theta \rp \racc \d x \\
	&\quad + 
	\begin{dcases}
		\int_{x_p(\theta)}^\infty \lambda(x,y_p(x)) \lacc y_p'(x) w_{1n}(x) - w_{2n} \lp y_p(x) \rp \racc \d x & \text{ if } p < \infty, \\
		- w_{1n}(1) \int_1^{1 \vee \tan\theta} \lambda(1,y) \d y - w_{2n}(1) \int_{1 \vee \cot\theta}^\infty \lambda(x,1) \d x & \text{ if } p = \infty.
	\end{dcases}
\end{align*}

Consequently, for any $p \in [1,\infty]$, under the same assumptions, we have in $\ell^\infty([0,\pi/2])$ the weak convergence
\[
	\lacc \sqrt{k} \lp \widehat{\Phi}_p(\theta) - \Phi_p(\theta) \rp \racc_{\theta \in [0,\pi/2]} \wc \lacc \alpha_p(\theta) \egdef W_\Lambda(C_{p,\theta}) + Z_p(\theta)  \racc_{\theta \in [0,\pi/2]}
\]
where the process $Z_p$ is defined by
\begin{align*}
	&Z_p(\theta) 
	\egdef \int_0^{x_p(\theta)} \lambda(x,x\tan \theta) \left[W_1(x) \tan \theta - W_2(x \tan \theta)\right] \, \d x \\ 
	&\, +
	\begin{dcases}
		\int_{x_p(\theta)}^\infty \lambda(x,y_p(x))  \left[W_1(x) y_p'(x) - W_2(y_p(x))\right] \, \d x  &\text{if } p < \infty, \\
		- W_1(1) \int_1^{1 \vee \tan \theta} \lambda(1,y) \, \d y - W_2(1) \int_{1 \vee \cot \theta}^\infty \lambda(x,1) \, \d x   &\text{if } p = \infty,
	\end{dcases}
\end{align*}
with $y_p'$ the derivative of $y_p$ and with $W_1(x) \egdef W_\Lambda((0,x] \times (0,\infty])$ and $W_2(y) \egdef W_\Lambda((0,\infty] \times (0,y])$ the marginal processes.
\end{theorem}

Based on Theorem~\ref{thm:asymptoticsEAM}, a simple application of Slutsky's Lemma permits to obtain an expansion of the process associated to the empirical angular probability measure. Some computations lead to
\begin{equation}
\label{eq:expansion_EAPM}
    \sup_{\theta \in [0,\pi/2]} \labs \sqrt{k}\lp \hQ_p(\theta) - Q_p(\theta) \rp - 
    E_{n,p}^Q(\theta)
    \rabs = \oh_{\PP}(1),
\end{equation}
where
\begin{equation}
\label{eq:expansion_EAPM_name}
    E_{n,p}^Q(\theta) \egdef \frac{E_{n,p}(\theta)\Phi_p(\tfrac{\pi}{2}) - \Phi_p(\theta) E_{n,p}(\tfrac{\pi}{2})}{\Phi_p(\tfrac{\pi}{2})^2}, \qquad
    \theta \in [0,\pi/2].
\end{equation}
By the second part of Theorem~\ref{thm:asymptoticsEAM}, we get weak convergence in $\ell^\infty([0,\pi/2])$,
\begin{equation}
\label{eq:asymptotic_EAPM}
    \lacc \sqrt{k} \lp \hQ_p(\theta) - Q_p(\theta) \rp \racc_{\theta \in [0,\pi/2]} \wc 
    \lacc \beta_p(\theta) \egdef \frac{\alpha_p(\theta)\Phi_p(\tfrac{\pi}{2}) - \Phi_p(\theta) \alpha_p(\tfrac{\pi}{2})}{\Phi_p(\tfrac{\pi}{2})^2} \racc_{\theta \in [0,\pi/2]}.
\end{equation}

Let $\Psi : D_\Psi \subset \ell^\infty([0,\pi/2]) \ra \ell^\infty([0,\pi/2])$ be the map defined for each element $F \in D_\Psi$, the set of cumulative distribution functions of non-degenerate probability measures on $[0,\pi/2]$, by
\begin{equation}
\label{eq:psi_map}
	(\Psi(F))(\theta) \egdef F(\theta) - \frac{\mu_F(f)}{\sigma^2_F(f)} \int_0^\theta \lp f(\psi) - \mu_F(f) \rp \d F(\psi),
\end{equation}
where
\begin{equation}
\label{eq:mu_var}
	\mu_F(f) \egdef \int_0^{\pi/2} f(\psi) \d F(\psi) \quad \text{and} \quad
	\sigma_F^2(f) \egdef \int_0^{\pi/2} \lp f(\psi) - \mu_F(f) \rp^2  \d F(\psi),
\end{equation}
and $f$ is defined as in~\eqref{eq:def_f}. The Euclidean likelihood estimator $\tQ_p$ satisfies
\[
	\tQ_p(\theta) = (\Psi(\hQ_p))(\theta), \qquad \theta \in [0,\pi/2].
\]
Hadamard differentiability of the map $\Psi$ combined with the functional delta method~\cite[Theorem~20.8]{VVV1998} permit to derive an asymptotic expansion for the maximum Euclidean likelihood estimator. This is formalized in the following result, the proof of which is to be found in Appendix~\ref{sec:app:expansionMEL}.

\begin{corollary}
\label{cor:asymptoticsMEL}
The map $\Psi : D_\Psi \subset \ell^\infty([0,\pi/2]) \ra \ell^\infty([0,\pi/2])$ defined in~\eqref{eq:psi_map} is Hadamard differentiable at $Q_p \in D_\Psi$ tangentially to $C_0([0,\pi/2])$, the space of continuous functions on $[0,\pi/2]$ vanishing at the boundaries, with derivative given by
\[
	(\Psi'_{Q_p}[h])(\theta) = h(\theta) +  \frac{\int_0^{\pi/2} h(\psi) f'(\psi) \d\psi }{\sigma^2_{Q_p}(f)} \int_0^\theta f(\psi) \d Q_p(\psi)
\]
for any $h \in C_0([0,\pi/2])$ and $\theta \in [0,\pi/2]$. Furthermore, $\Psi'_{Q_p}$ is defined and continuous on the whole space $\ell^\infty([0,\pi/2])$ so that in $\ell^\infty([0,\pi/2])$,
\[
	\sqrt{k} \lp \tQ_p - Q_p \rp 
	= \sqrt{k} \lp \Psi(\hQ_p) - \Psi(Q_p) \rp
	= \Psi'_{Q_p} \lc \sqrt{k} \lp \hQ_p - Q_p \rp \rc + \oh_{\PP}(1).
\]
Consequently, under the same assumptions as Theorem~\ref{thm:asymptoticsEAM}, we have, as $n \ra \infty$,
\begin{multline*}
    \sup_{\theta \in [0,\pi/2]} \labs \sqrt{k} \lp \tQ_p(\theta) - Q_p(\theta) \rp - \lp E_{n,p}^Q(\theta) + \frac{\int_0^{\pi/2} E_{n,p}^Q(\psi) f'(\psi) \d\psi }{\sigma_{Q_p}^2(f)} \int_0^\theta f(\psi) \d Q_p(\psi) \rp  \rabs \\
    = \oh_{\PP}(1),
\end{multline*}
and, in $\ell^\infty([0,\pi/2])$,
\begin{multline*}
    \lacc \sqrt{k} \lp \tQ_p(\theta) - Q_p(\theta) \rp \racc_{\theta \in [0,\pi/2]} \\ \wc 
    \lacc \gamma_p(\theta) \egdef \beta_p(\theta) + \frac{\int_0^{\pi/2} \beta_p(\psi) f'(\psi) \d\psi }{\sigma_{Q_p}^2(f)} \int_0^\theta f(\psi) \d Q_p(\psi) \racc_{\theta \in [0,\pi/2]}.
\end{multline*}
\end{corollary}

The asymptotic process in Corollary~\ref{cor:asymptoticsMEL} equals the one in Theorem~4.1 of~\cite{einmahl2009maximum}, showing that, at least asymptotically, enforcing the marginal constraints~\eqref{eq:margins_Q} via maximization of the empirical likelihood or its Euclidean version makes no difference, as already discussed and illustrated in~\cite{decarvalho2013} for the case $p=1$. 

\section{Goodness-of-fit tests for parametric models}
\label{sec:GoF}

\subsection{Test statistic}
\label{sec:GoF:ts}

Given a parametric model $\mod = \{\Phi_{p,r}: r \in \param \subseteq \R^m\}$ for the angular measure $\Phi_p$, our main goal is to test
\begin{equation}
\label{eq:THEtest}
	H_0: \Phi_p \in \mod \qquad \text{vs} \qquad H_1: \Phi_p \notin \mod.
\end{equation}
As a test statistic, we propose to use a weighted version of the \emph{$L_1$-Wasserstein distance} from optimal transport theory~\cite{Villani2003} between $\tQ_p$ in \eqref{eq:max_emp_lik_est} and $Q_{p,\widehat{r}_n}$, where $\est$ is some estimator of the true parameter $r_0$ under $H_0$, which can be seen as a nuisance parameter. More explicitly, our test statistic will be of the form
\begin{equation}
\label{eq:test_statistics}
	\ts \egdef \sqrt{k} \int_0^{\pi/2} \labs \tQ_p(\theta) - Q_{p,\est}(\theta) \rabs \weight(\theta) \d\theta,
\end{equation}
where $\weight : (0,\pi/2) \ra (0,\infty)$ is some positive, integrable weight function. The special case $\weight \equiv 1$ reduces to the classical $L_1$-Wasserstein distance. Under differentiability assumptions on the model $\mod$ and general conditions on the estimator $\est$, we can derive the asymptotic distribution of $\ts$ in Section~\ref{sec:GoF:crit}. Before we do so, we develop an asymptotic expansion of $\est$ in Section~\ref{sec:GoF:param}.

\subsection{Parameter estimator and empirical stable tail dependence function}
\label{sec:GoF:param}

The condition on the parameter estimator will be formulated in terms of the \emph{stable tail dependence function} $\ell : [0,\infty)^2 \ra [0,\infty)$ obtained by taking the limit
\begin{equation}
\label{eq:stdf}
	\ell(x_1,x_2) \egdef \lim_{t \ra 0} t^{-1} \PP \lc U_1 \leq tx_1 \text{ or } U_2 \leq tx_2 \rc,
\end{equation}
which is known to exist due to~\eqref{eq:regular_variation_sets}. 
This function describes the extremal dependence structure of a random vector and has a close link with the angular measure, see~\cite[Section~8.6.2]{beirlant05}.
As for the angular measure, we may estimate $\ell$ nonparametrically by replacing the distribution functions in the expression by their empirical counterparts and by replacing $t$ by $k/n$. This leads to the rank-based estimator
\begin{equation}
\label{eq:stdf_estimation}
	\widehat{\ell}_n(x_1,x_2) \egdef \frac{1}{k} \sum_{i=1}^n \1 \lacc R_{i1} > n + \tfrac{1}{2} - kx_1 \text{ or } R_{i2} > n + \tfrac{1}{2} - kx_2 \racc,
\end{equation}
where $R_{ij}$ denotes the rank of $X_{ij}$ in the marginal sample $X_{1j}, \ldots, X_{nj}$. The constant $1/2$ is added in the indicators to improve finite-sample properties but will not play any role in our asymptotic considerations~\cite[Equation~(7.1)]{einmahl2012}.

Let $x \dotp y$ denote the scalar product of $x, y \in \R^m$ and let $|x|$ denote the Euclidean norm of $x \in \R^m$.

\begin{assumption}
\label{ass:modelAndEstimator}
Let $r_0 \in \param$ denote the parameter value such that $\Phi_p = \Phi_{p,r_0}$ under $H_0$.
\begin{enumerate}
\item The point $r_0$ lies in the interior of $\param$ and the map $\paramap: r \in \param \subseteq \R^m \mapsto Q_{p,r} \in \ell^\infty([0,\pi/2])$ is Hadamard differentiable at $r_0$ with derivative $\paramap'_{r_0} : \R^m \ra \ell^\infty([0,\pi/2])$. 
For every $\theta \in [0,\pi/2]$, the map $r \in \param \mapsto \varrho_{p,r}(\theta)$ is differentiable and the derivative satisfies
\[
	\int_0^{\pi/2} \labs \nabla_r \varrho_{p,r_0}(\theta) \rabs \d\theta < \infty.
\]
Differentiation with respect to $r$ and integration with respect to $\theta$ can be exchanged while considering the map $(r,\theta) \in \param \times [0,\pi/2] \mapsto \varrho_{p,r}(\theta)$, so that
\[
	(\paramap_r'[h])(\theta) = \int_0^\theta \nabla_r \varrho_{p,r}(x) \d x \dotp h.
\]
\item The estimator $\est$ satisfies the following expansion~\cite[Assumption~4.6]{hu2024}: there exists a finite Borel measure $\sigma$ on $[0,1]^2$ and a function $g : [0,1]^2 \ra \R^m$ in $\lp L_2([0,1]^2,\sigma) \rp^m$ such that, as $n \ra \infty$,
\[
	\sqrt{k}(\est - r_0) = \int_{[0,1]^2} \sqrt{k} \lp \widehat{\ell}_n(x_1,x_2) - \ell(x_1,x_2) \rp g(x_1,x_2) \, \d\sigma(x_1,x_2) + \oh_{\PP}(1).
\]
\end{enumerate}
\end{assumption}

The second part of Assumption~\ref{ass:modelAndEstimator} will permit to derive an asymptotic expansion of $\sqrt{k} \, ( \est - r_0 )$ on the basis of an expansion of the process $\sqrt{k} \, ( \widehat{\ell}_n - \ell )$ and hence to derive the limit distribution of $\ts$ under $H_0$ using the first part of the assumption. It is satisfied by many common estimators, e.g., M-estimators (and hence moment estimators) or weighted least-square estimators, see Examples~4.8--10 in \cite{hu2024}.

Asymptotic theory for the process $\sqrt{k} \, (\widehat{\ell}_n - \ell)$ is studied in~\cite{einmahl2012}, in general dimension $d \geq 2$ (the definitions are analogous to the bivariate case). For any $T > 0$, the process is asymptotically Gaussian in the space $\ell^\infty([0,T]^d)$, and this under general conditions which are met in our bivariate setting. Under the same assumptions as in~\cite{einmahl2012}, we provide an asymptotic expansion of the process in general dimension too.

\begin{assumption}
\label{ass:stdf_expansion}
Let $(X_1,\ldots,X_d)$ be a random vector in $\R^d$ with continuous distribution function $F$ and marginal distribution functions $F_1,\ldots,F_d$. Let $U = (1-F_1(X_1),\ldots,1-F_d(X_d))$ be its standardization to uniform margins, whose law will be denoted by $P$, for which the stable tail dependence function is assumed to exist, that is, for $x \in [0,\infty)^d$, the limit
\[
\lim_{t \ra 0} t^{-1} P \lc tA_x \rc = \lim_{t \ra 0} t^{-1} \PP \lc U_1 \leq tx_1 \text{ or } \ldots \text{ or } U_d \leq tx_d \rc \egdef \ell(x_1,\ldots,x_d)
\]
exists, with $A_x \egdef \{u \in [0,\infty)^d: u_1 \leq x_1 \text{ or } \ldots \text{ or } u_d \leq x_d \}$. Moreover, for some $\alpha > 0$,
\begin{enumerate}
	\item $|t^{-1} \PP \lc U_1 \leq tx_1 \text{ or } \ldots \text{ or } U_d \leq tx_d \rc - \ell(x_1,\ldots,x_d)| = \Oh(t^\alpha)$ uniformly in $x \in [0,\infty)^d$ such that $x_1 + \cdots + x_d = 1$ as $t \ra 0$;
	\item $k = \oh(n^{2\alpha/(1+2\alpha)})$ and $k \ra \infty$ as $n \ra \infty$;
	\item for every $j \in \{1,\ldots,d\}$, the first-order partial derivative of $\ell$ with respect to $x_j$, denoted by $\dot{\ell}_j$, exists and is continuous on the set of points $x$ such that $x_j > 0$.
\end{enumerate}
\end{assumption}

\begin{theorem}[Asymptotic expansion of the empirical stable tail dependence function]
\label{thm:expansionSTDF}
Under Assumption~\ref{ass:stdf_expansion}, let $P_n$ denote the empirical distribution of an independent random sample from $P$ of size $n$.
Then for every $T > 0$, as $n \ra \infty$,
\[
	\sup_{x \in [0,T]^d} \labs \sqrt{k} \lp \ell_n(x) - \ell(x) \rp - \lp v_n(x) - \sum_{j=1}^{d} \dot{\ell}_j(x) v_{nj}(x_j) \rp \rabs = \oh_{\PP}(1),
\]
where, for every $x \in [0,\infty)^d$ and $j=1,\ldots,d$,
\[
	v_n(x) = \sqrt{k} \lacc \tfrac{n}{k} P_n \lp \tfrac{k}{n} A_x \rp - \tfrac{n}{k} P \lp \tfrac{k}{n} A_x \rp \racc, \quad
	v_{nj}(x_j) = v_n(0,\ldots,0,x_j,0,\ldots,0),
\]
and where $\dot{\ell}_j$ is taken to be the right-hand partial derivative if $x_l = 0$ for some $l = 1,\ldots,d$ (the latter always exists, see~\cite[Equation~(2.7)]{einmahl2012}).

Furthermore, if $d=2$, Assumption~\ref{ass:stdf_expansion} is satisfied with the same $\alpha > 0$ as in Assumption~\ref{ass:bias} provided Assumptions~\ref{ass:smoothness} and~\ref{ass:bias} are verified.
\end{theorem}

The proof of Theorem~\ref{thm:expansionSTDF} can be found in Appendix~\ref{sec:app:expansionSTDF}.

\subsection{Limit distribution of test statistic and estimation of critical values}
\label{sec:GoF:crit}

We are now ready to state the limit distribution of the test statistic $\ts$ under $H_0$. The proof of Theorem~\ref{thm:asymptoticTS} is provided in Appendix~\ref{sec:app:asymptoticTS}.

\begin{theorem}
\label{thm:asymptoticTS}
If Assumptions~\ref{ass:smoothness},~\ref{ass:bias} and~\ref{ass:modelAndEstimator} are satisfied and if $r_0 \in \param$ denotes the true value of the parameter under $H_0$, then under $H_0$, we have
\[
	\ts \wc L_{r_0} =
	 \int_0^{\pi/2} \labs \gamma_{p,r_0}(\theta) - \int_0^\theta \nabla_r \varrho_{p,r_0}(x) \d x \dotp I_{g,\sigma,r_0} \rabs \weight(\theta) \d\theta,
\]
as $n \ra \infty$, where we recall $\gamma_p$ in Corollary~\ref{cor:asymptoticsMEL} and where
\[
	I_{g,\sigma,r_0} = \int_{[0,1]^2} \Big( W_{\Lambda_{r_0}}(A_{(x_1,x_2)}) - \sum_{j=1}^2 \dot{\ell}_{r_0,j}(x_1,x_2) W_j(x_j) \Big) g(x_1,x_2) \d\sigma(x_1,x_2).
\]
\end{theorem}

Under an additional regularity condition on the model, which will be easily verified in practice, one may show that the critical values for the asymptotic distribution can be computed from an estimated version of the true unknown parameter under $H_0$. The critical values can thus be computed by Monte Carlo simulation from the limit distribution of $T_n$ in Theorem~\ref{thm:asymptoticTS} but at the estimated parameter value, as explained in Section~\ref{sec:simulations}.

\begin{assumption}
\label{ass:density_reg}
Let $r_0 \in \param$ be an interior point and introduce the function
\[
	\weightt : \theta \in [0,\pi/2] \mapsto \weightt(\theta) \egdef  \min \lacc \tan\theta, \tan(\tfrac{\pi}{2} - \theta) \racc \in [0,1].
\]
\begin{enumerate}
	\item We have
	\[
		\sup_{\theta \in [0,\pi/2]} \lacc \weightt(\theta) \varphi_{p,r_0}(\theta) \racc < \infty.
	\]
	\item There exists $\nu \in (0,1)$ such that
	\[
		\lim_{n \ra \infty} \sup_{\theta \in [0,\pi/2]} \lacc \weightt(\theta)^\nu \labs \varphi_{p,r_n}(\theta) - \varphi_{p,r_0}(\theta) \rabs \racc = 0,
	\]
	and 
	\[
		\lim_{n \ra \infty} \sup_{\theta \in [0,\pi/2]} \lacc \weightt(\theta)^\nu \labs \nabla_r \varphi_{p,r_n}(\theta) - \nabla_r \varphi_{p,r_0}(\theta) \rabs \racc = 0,
	\]
	for any sequence $(r_n)_{n \in \N} \subset \param$ such that $r_n \ra r_0$ as $n \ra \infty$.
\end{enumerate}
\end{assumption}

\begin{theorem}
\label{thm:cont_quantiles}
For any $\alpha \in (0,1)$ and $r \in \param$, let $\quantile_{r}(\alpha)$ be the $\alpha$-quantile of $L_{r}$. Under the same assumptions as in Theorem~\ref{thm:asymptoticTS} in combination with Assumption~\ref{ass:density_reg}, we have
\[
	\quantile_{\est}(1-\alpha) \xrightarrow{\PP} \quantile_{r_0}(1-\alpha), \qquad \alpha \in (0,1),
\]
for any consistent estimator $\est$ of $r_0$, the true parameter under $H_0$. The convergence also holds locally uniformly: for any $[\alpha_0, \alpha_1] \subset [0,1]$,
\[
	\sup_{\alpha \in [\alpha_0, \alpha_1]} \labs \quantile_{\est}(1-\alpha) - \quantile_{r_0}(1-\alpha) \rabs \xrightarrow{\PP} 0.
\]
\end{theorem}

The proof of Theorem~\ref{thm:cont_quantiles} is involved and can be found in Appendix~\ref{sec:app:cont_quantiles}.

\begin{remark}
Point 1 of Assumption~\ref{ass:modelAndEstimator} is formulated in terms of the map $r \in \param \mapsto Q_{p,r}$ but is fully equivalent to its counterpart for the map $r \in \param \mapsto \Phi_{p,r}$. The same remark holds true for Assumption~\ref{ass:density_reg} which is equivalent to the present formulation where $\varphi_{p,r}$ would be replaced by $\varrho_{p,r}$. The normalizing constant permitting to move from one definition to another, even though it depends on $r$, can be shown to have no influence on our assumptions.
\end{remark}

\section{Simulations and finite-sample performance}
\label{sec:simulations}

Theorems~\ref{thm:asymptoticTS} and~\ref{thm:cont_quantiles} provide a natural way to perform the goodness-of-fit test proposed in Eq.~\eqref{eq:THEtest}, for any size $\alpha \in (0,1)$. Given a model to be tested and an estimator for the unknown parameter of the model, one can approximate the quantiles $\quantile_{r}(1-\alpha)$ for any $r \in \param$ by directly simulating the random variable $L_r$, the distribution of which only depends on the given model and the asymptotic expansion of the estimator as in point~2 of Assumption~\ref{ass:modelAndEstimator}, and then taking empirical quantiles. This is typically done for a finite grid of values for $r \in \param$ and, by interpolation, leads to critical values for any value of the parameter. Once those values have been obtained, they can be used for any dataset for which the model needs to be tested, provided that the estimator of the model parameter satisfies the same expansion as the one used to obtain the critical values. The procedure is illustrated for the logistic (Section~\ref{sec:simu-logistic}) and the Hüsler--Reiss (Section~\ref{sec:simu-HR}) models. Details underlying the simulation of $L_r$ are to be found in Appendix~\ref{sec:app:simus}. An \texttt{R}-package is under construction, and in the meantime, software will be available from the first author on GitHub \footnote{E-mail \href{mailto:stephane.lhaut@uclouvain.be}{stephane.lhaut@uclouvain.be} if the repository is not yet available at the time of reading.}.

\subsection{Logistic model}
\label{sec:simu-logistic}

\paragraph*{Model.} 
One of the most famous and oldest parametric families of bivariate extreme value dependence structures is the \emph{logistic model} with the stable tail dependence function~\eqref{eq:stdf}
\[
	\ell_r(x,y) = \lp x^{1/r} + y^{1/r} \rp^r, \qquad r \in \param \egdef (0,1].
\]
The model ranges from asymptotic full dependence ($r \ra 0$) to asymptotic independence ($r=1$). Since $\ell_r(x,y) = x + y - \Lambda_r([0,x] \times [0,y])$, one easily obtains
\[
	\lambda_r(x,y) = \lp \frac{1}{r}-1 \rp \frac{(xy)^{\frac{1}{r}-1}}{(x^{1/r}+y^{1/r})^{2-r}}, \qquad x,y > 0.
\]
For $r \in (0,1)$, the angular measure $\Phi_{p,r}$ on the $L_p$-sphere, for $1 \leq p < \infty$, satisfies $\Phi_p(\{0,\pi/2\}) = 0$, while on the interior of the sphere $(0,\pi/2)$, it is absolutely continuous with a density which can be computed using relation~\eqref{eq:phi_from_lambda} to be
\[
	\varphi_{p,r}(\theta) = \lp \frac{1}{r}-1 \rp \|(\sin\theta,\cos\theta)\|_p \frac{(\sin\theta \cos\theta)^{\frac{1}{r}-2}}{\lp (\sin\theta)^{1/r} + (\cos\theta)^{1/r} \rp^{2-r}}, \qquad \theta \in (0,\pi/2).
\]
In particular, Assumption~\ref{ass:smoothness} is satisfied for this model. Point~1 of Assumption~\ref{ass:modelAndEstimator} and Assumption~\ref{ass:density_reg} are also satisfied for $r_0 = 0.5$.

\paragraph*{Parameter estimation.}
Given a random sample $X_1, \ldots, X_n$, the unknown parameter $r \in \param$ of the model will be estimated by inverting the \emph{extremal coefficient}~\cite[Chapter~8]{beirlant05} $$
	\chi \egdef \ell(1,1) \in [1,2],
$$
which equals $\chi = 2^r$ for the logistic model. In the limiting case $\chi \to 1$, we approach asymptotic full dependence while $\chi \to 2$ means we reach asymptotic independence. Our estimator simply consists of writing $r = \log_2(\chi)$ and replacing $\chi$ by its empirical version leading to
\begin{equation}
\label{eq:estimation_logistic}
	\widehat{r}_n = \log_2(\widehat{\ell}_n(1,1)),
\end{equation}
where $\widehat{\ell}_n$ is defined in~\eqref{eq:stdf_estimation}. It follows from the delta method that $\widehat{r}_n$ satisfies the asymptotic expansion
\[
	\sqrt{k} \lp \widehat{r}_n - r \rp
	= \sqrt{k} \lp \widehat{\ell}_n(1,1) - \ell(1,1) \rp \cdot \frac{1}{2^r \log(2)} + \oh_{\PP}(1),
\]
as $n \ra \infty$, meaning that point~2 of Assumption~\ref{ass:modelAndEstimator} is satisfied with
\[
	g \equiv \frac{1}{2^r \log(2)} \qquad \text{and} \qquad
	\sigma = \delta_{(1,1)}.
\]

\paragraph*{Data generating processes.}
We generate data from the copula mixture model
\begin{equation}
\label{eq:copmixmod}
	C_\lambda(u,v) 
	= (1 - \lambda) C_0(u,v)  + \lambda C_a(u,v), 
	\qquad (u,v) \in [0,1]^2,
\end{equation}
for $\lambda \in \{0,0.1,0.2,\ldots,0.8\}$ where
\begin{equation}
\label{eq:simuC0logistic}
	C_0(u,v) \egdef 
	\exp \left\{ - \lp (-\log u)^2 + (-\log v)^2 \rp^{0.5} \right\}
\end{equation}
is the \emph{Gumbel} copula with parameter equal to $2$ and with two choices for the second component $C_a \in \{C_1,C_2\}$:
\begin{itemize}
\item In scenario~1, 
$
	C_1(u,v) \egdef \min(u,v)
$
is the copula of the random vector $(U,U)$ where $U$ is uniformly distributed on $[0,1]$.
\item In scenario~2,
\[
	C_2(u,v) \egdef \exp \lacc - \max(-0.7 \log u, -0.1 \log v) - \max(-0.3 \log u, -0.9 \log v) \racc
\]
is the copula of the max-linear factor model
\[
	\begin{pmatrix}
		X_1 \\ X_2
	\end{pmatrix}
	=
	\begin{pmatrix}
		0.7 Z_1 \vee 0.3 Z_2 \\ 0.1 Z_1 \vee 0.9 Z_2 
	\end{pmatrix}
\]
where $Z_1$ and $Z_2$ are independent Fréchet(1) random factors.
\end{itemize}

\begin{figure}[H]
\begin{subfigure}{0.49\textwidth}
    \includegraphics[width=\textwidth]{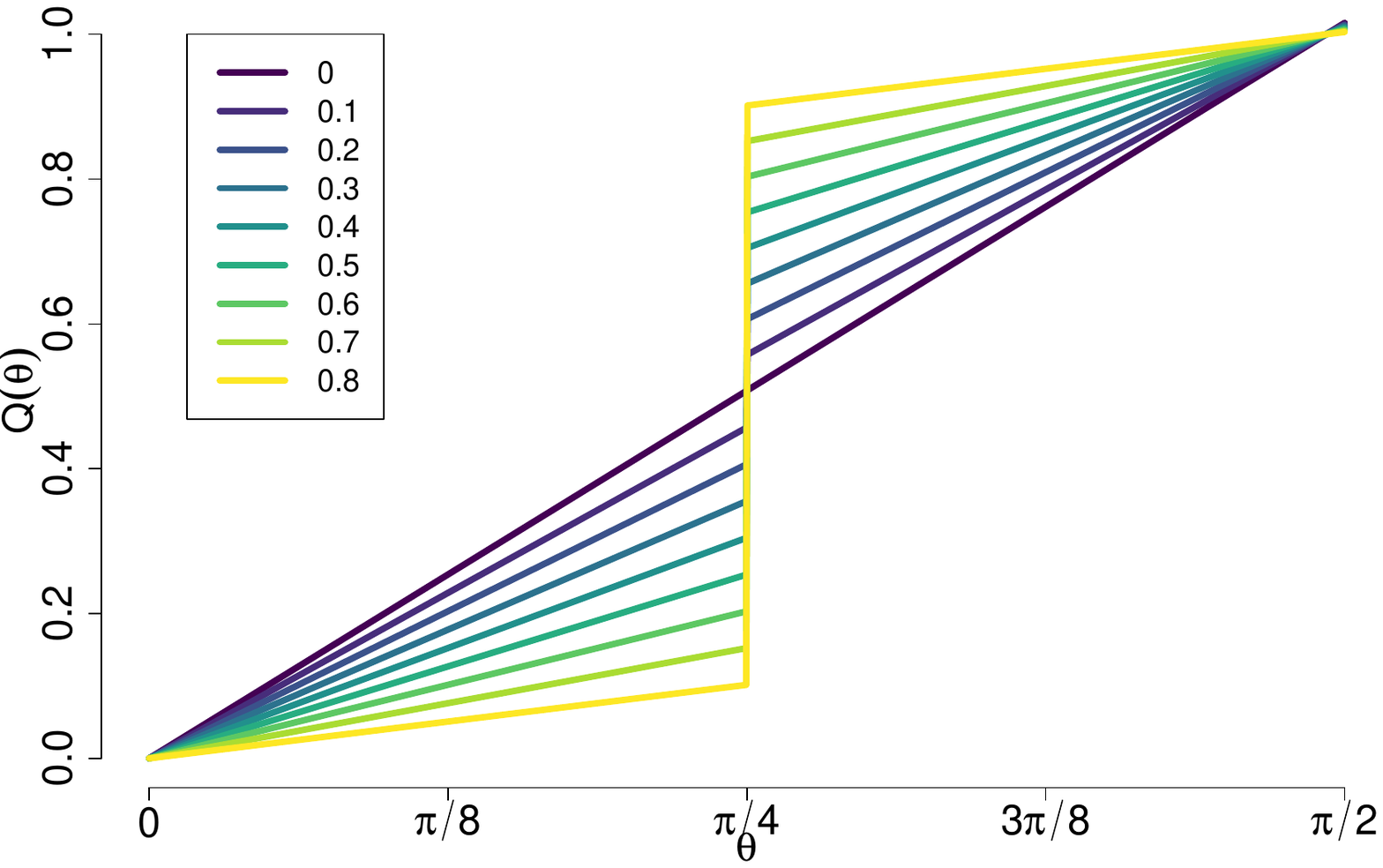}
    \caption{Scenario 1, $C_a = C_1$}
    \label{fig:sc1}
\end{subfigure}
\hfill
\begin{subfigure}{0.49\textwidth}
    \includegraphics[width=\textwidth]{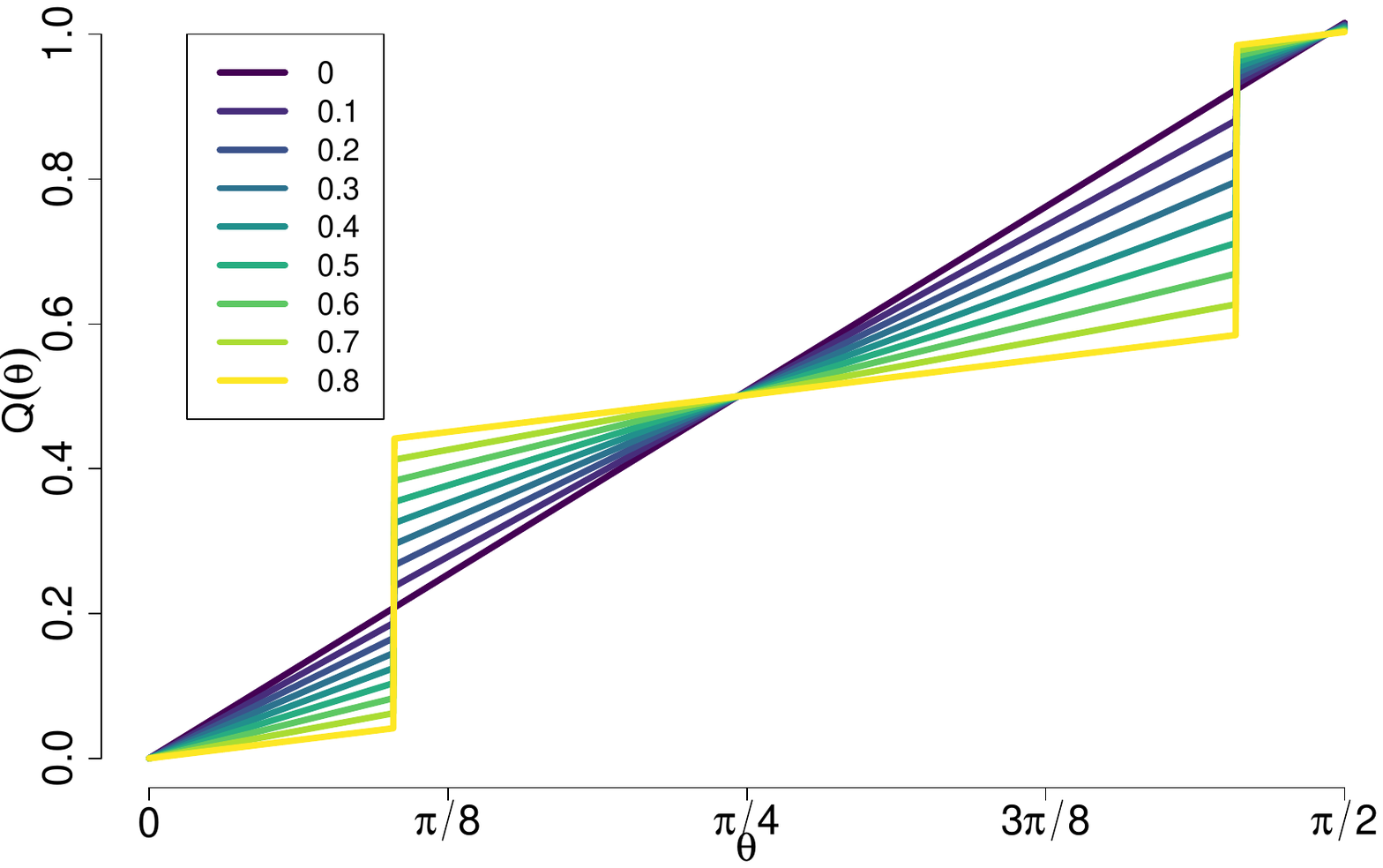}
    \caption{Scenario 2, $C_a = C_2$}
    \label{fig:sc2}
\end{subfigure}
\caption{Angular distribution functions $\theta \mapsto Q_p(\theta)$ for $p=2$ of the copula mixture model~\eqref{eq:copmixmod} with Gumbel copula $C_0$ in \eqref{eq:simuC0logistic} and under two scenarios for $C_a$.}
\label{fig:sc}
\end{figure}

In both cases, when $\lambda = 0$, the null hypothesis of the test~\eqref{eq:THEtest} is satisfied for the logistic model with parameter $r_0 = 0.5$ and Assumption~\ref{ass:bias} holds true if we choose $\alpha = 1$, while for $0 < \lambda < 1$, the null hypothesis is no longer satisfied. The asymptotic dependence structures corresponding to those mixtures are represented on Figure~\ref{fig:sc} by their angular distribution function $Q_p$ for $p = 2$. Scenario~1 corresponds exactly to the simulation setup in~\cite{dehaan2008}, while scenario~2 is similar to the one in~\cite{can2015}, to facilitate comparison.

\paragraph*{Results.}

The results for scenario~1 are presented in Figure~\ref{fig:results_sc1}. On the top-left panel, estimates of the logistic model parameter are shown as a function of the mixture weight $\lambda$. Results are consistent with the choice $r_0 = 0.5$ and the fact that stronger dependence corresponds to smaller logistic parameter values. The three other panels show the empirical power (green line) of the test for different values of the sample size $n$ and effective sample size $k$. The weight function is $q(\theta) = 1/\sqrt{|\theta - \pi/4|}$, emphasizing the points close to $\theta = \pi/4$. Regarding Figure~\ref{fig:sc1}, this choice is expected to help in scenario~1, but, perhaps more surprisingly, it is also helpful in scenario~2 where the largest deviation from the null hypothesis is no longer located at $\theta = \pi/4$. Overall, such a choice produces better results than the classical $L_1$-Wasserstein distance with $q \equiv 1$.
Both the power and the asymptotic quantile are estimated on the basis of $2000$ replications of the corresponding quantities. The size $\alpha$ of the test is $5\%$ in all cases and is shown as a horizontal brown line. The graphs confirm the good performance of the test. On the bottom-right panel, the additional orange dots visualize the empirical power of the method in~\cite{dehaan2008} for the same scenario and sample size $n$. Their test is a bit more conservative under the null, while its power under the alternative is less than the one of our method. The slight power decrease as $\lambda \ra 1$ is unsurprising since $H_0$ holds true for the copula $C_1$ with $r = 0$. 

The corresponding results for scenario~2 are presented in Figure~\ref{fig:results_sc2}. The empirical power curves are obtained in the same way as for scenario~1. The weighted $L_1$-Wasserstein test easily detects the asymmetric alternative. The same scenario is considered in~\cite{can2015} (with a slightly different factor model for the alternative, yielding different locations of the atoms of the angular measure), where an empirical power of $90\%$ to $95\%$ is attained depending on the method. In the bottom-right panel, the empirical power almost attains $1$, even for $\lambda = 0.8$ and $n = 500$. The method in \cite{can2015} has the advantage that the critical values do not depend on the model and thus have to be computed only once; the computation of the test statistic is more involved, however.

\begin{figure}[h]
    \centering
    \begin{subfigure}{0.45\textwidth}
        \centering
   		\includegraphics[scale=0.2]{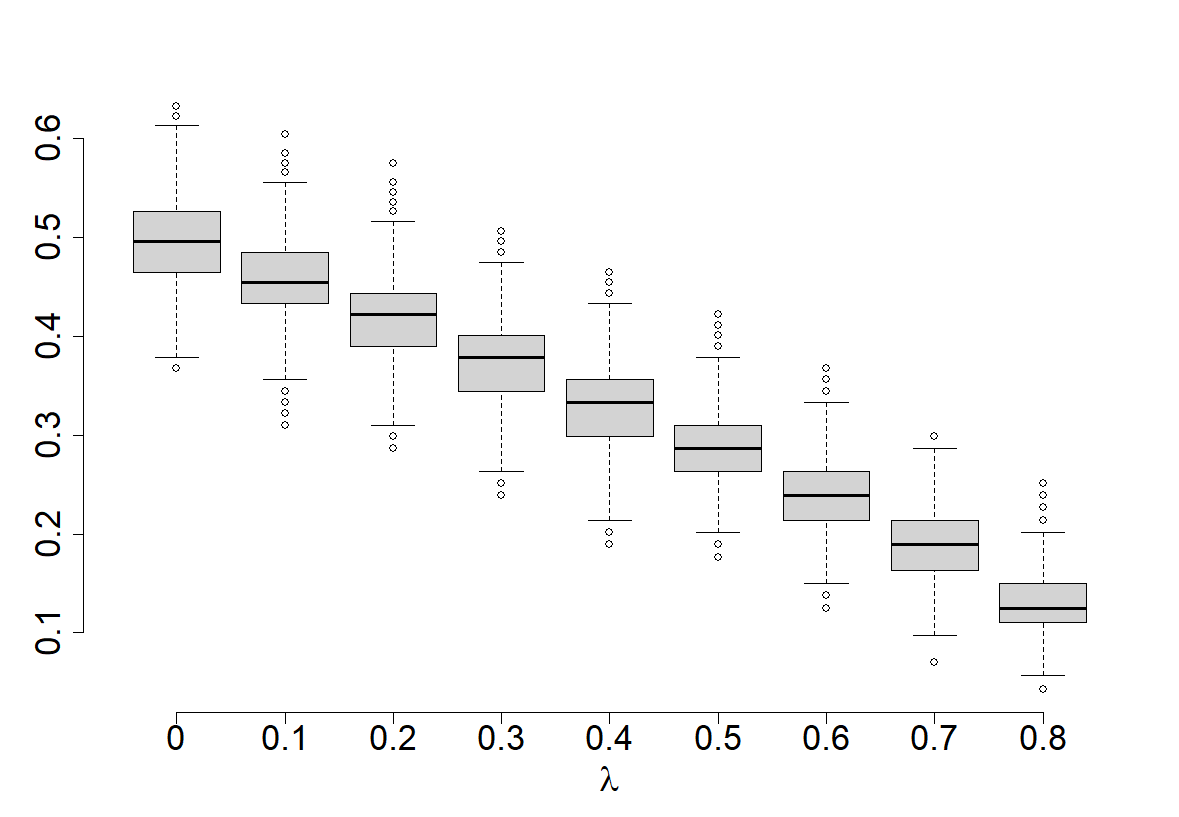}
   		\caption{Parameter estimates, $n=10\,000$ and $k=100$}
    \end{subfigure}
    \hspace{0.05\textwidth}
    \begin{subfigure}{0.45\textwidth}
        \centering
        \includegraphics[scale=0.2]{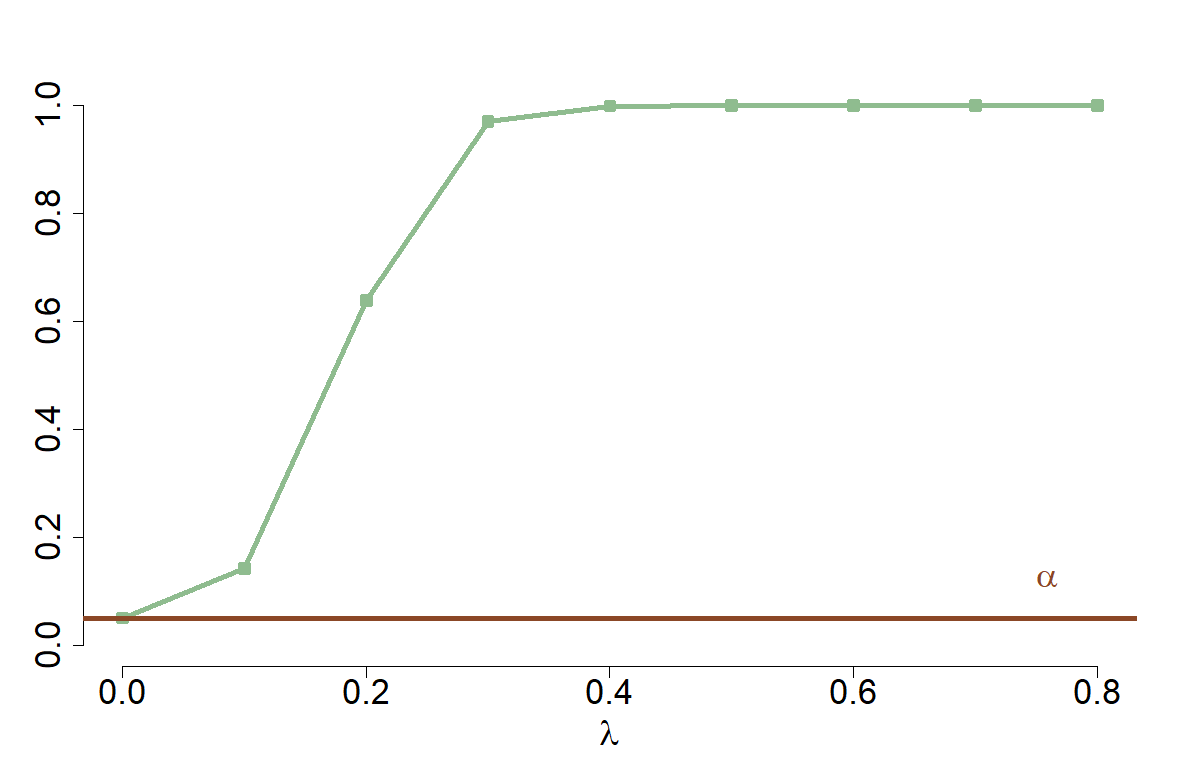}
   		\caption{Power curve, $n=10\,000$ and $k=100$}
    \end{subfigure}
    \vspace{0.5cm}
    \begin{subfigure}{0.45\textwidth}
        \centering
  			\includegraphics[scale=0.2]{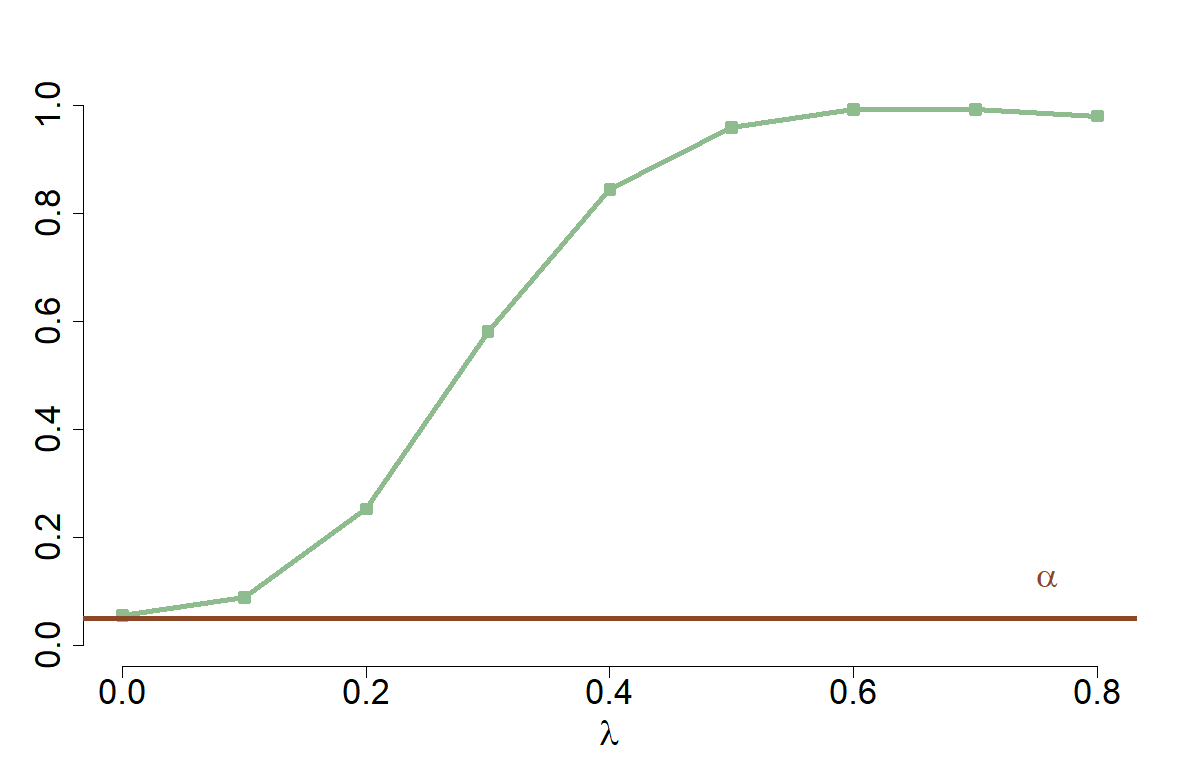}
   		\caption{Power curve, $n=3\,000$ and $k=50$}
    \end{subfigure}
    \hspace{0.05\textwidth}
    \begin{subfigure}{0.45\textwidth}
        \centering
   		\includegraphics[scale=0.2]{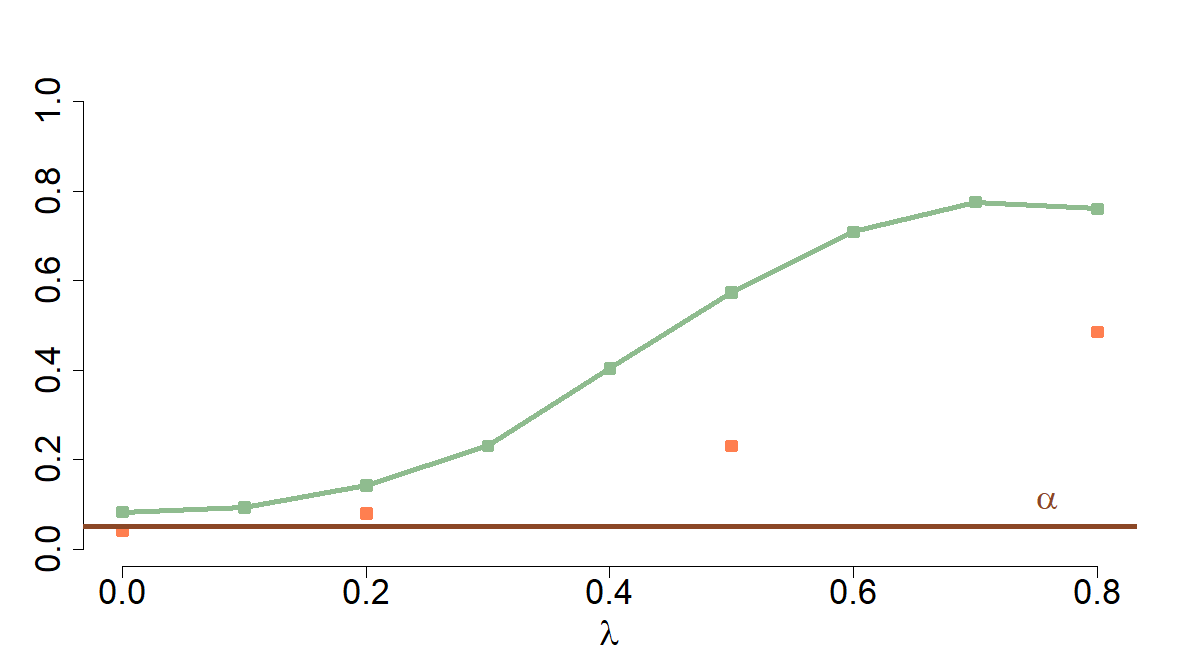}
   		\caption{Power curve, $n=500$ and $k=25$; orange dots show power of test in \cite{dehaan2008}}
    \end{subfigure}
\caption{Logistic model: parameter estimates (a) and empirical power curves (b--d) of weighted $L_1$-Wasserstein goodness-of-fit test in scenario~1 of copula mixture model~\eqref{eq:copmixmod}}
\label{fig:results_sc1}
\end{figure}
\begin{figure}[h]
    \centering
    \begin{subfigure}{0.45\textwidth}
        \centering
   		\includegraphics[scale=0.2]{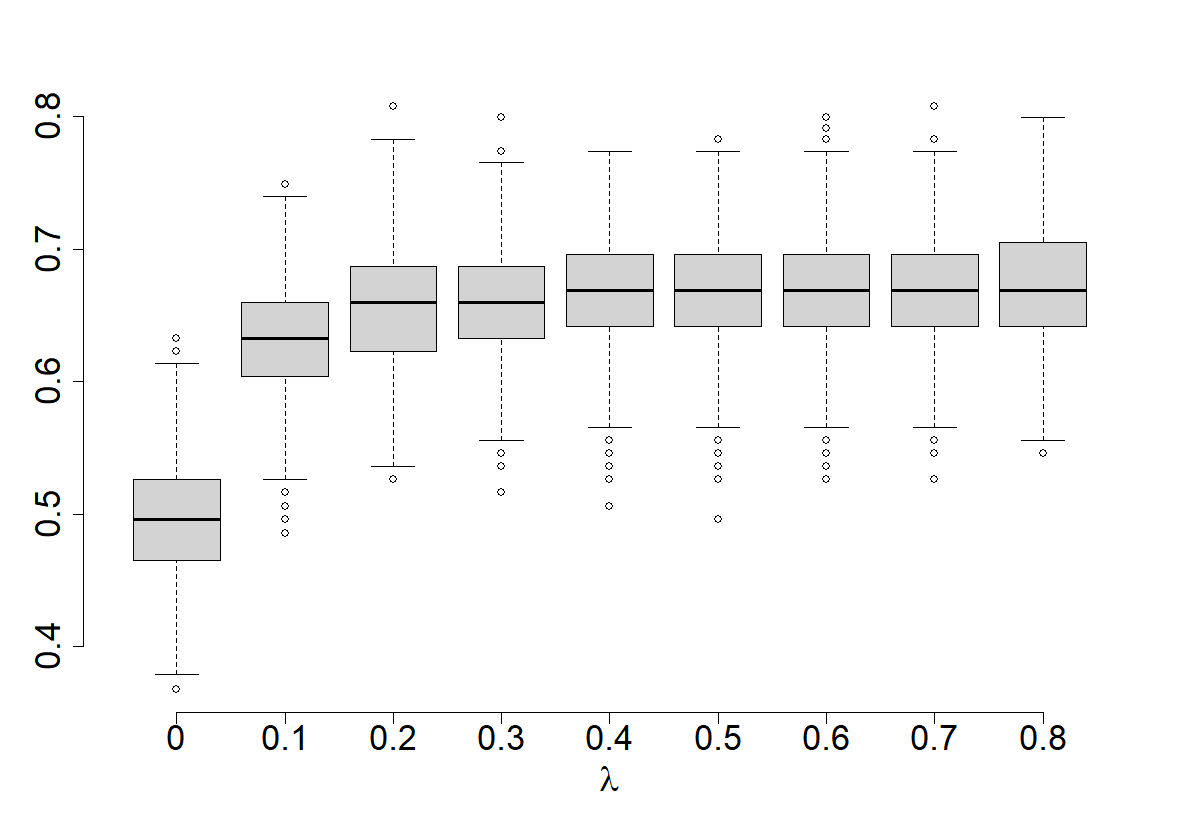}
   		\caption{Parameter estimates, $n=10\,000$ and $k=100$}
    \end{subfigure}
    \hspace{0.05\textwidth}
    \begin{subfigure}{0.45\textwidth}
        \centering
        \includegraphics[scale=0.2]{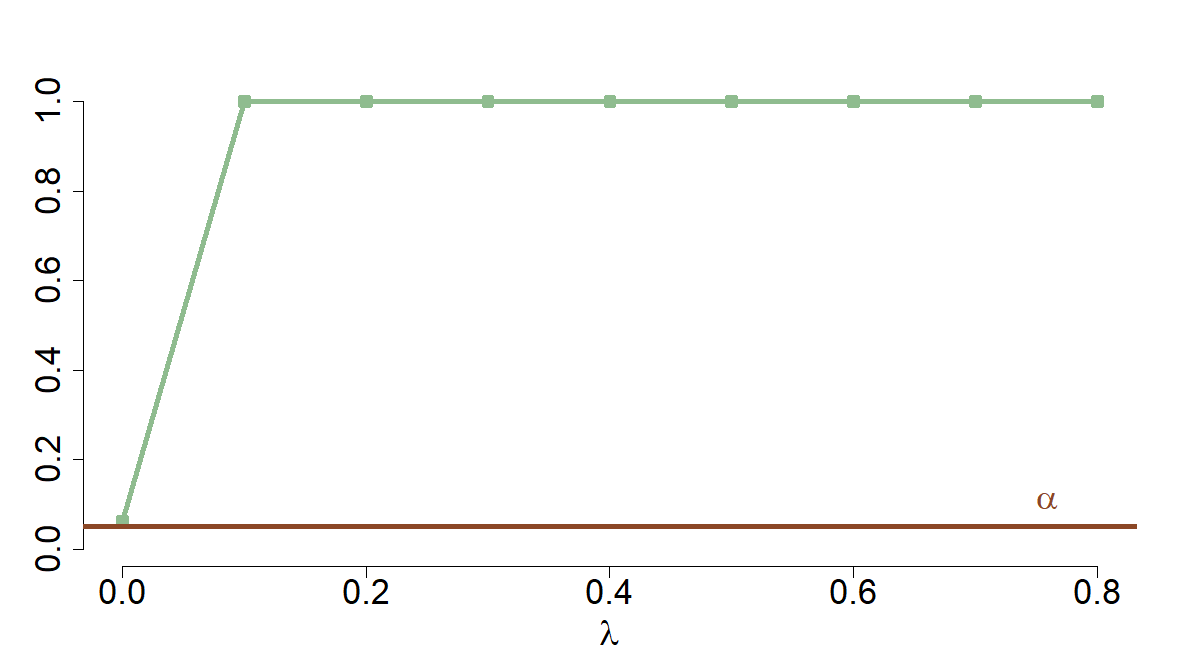}
   		\caption{Power curve, $n=10\,000$ and $k=100$}
    \end{subfigure}
    \vspace{0.5cm}
    \begin{subfigure}{0.45\textwidth}
        \centering
  			\includegraphics[scale=0.2]{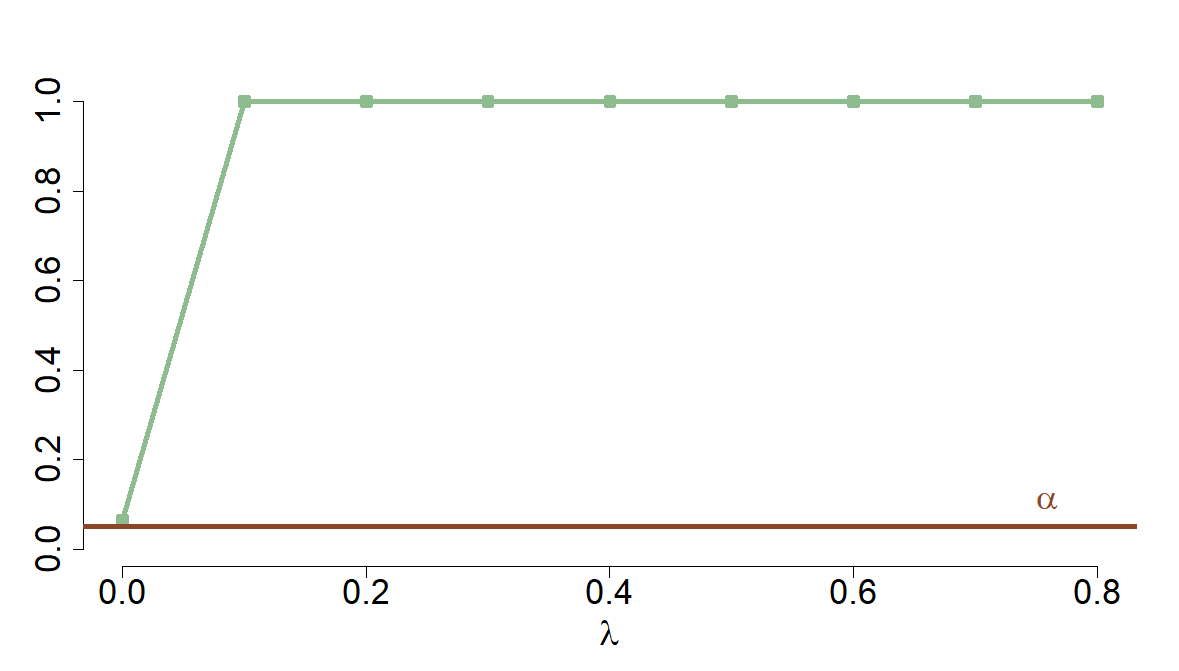}
   		\caption{Power curve, $n=3\,000$ and $k=50$}
    \end{subfigure}
    \hspace{0.05\textwidth}
    \begin{subfigure}{0.45\textwidth}
        \centering
   		\includegraphics[scale=0.2]{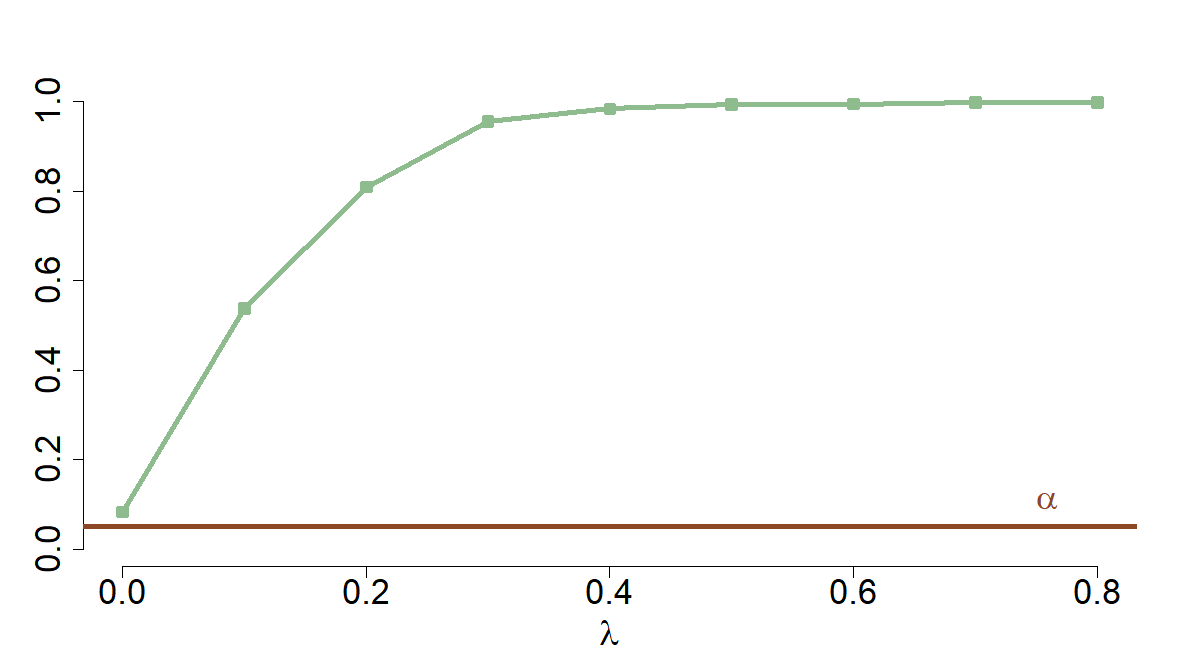}
   		\caption{Power curve, $n=500$ and $k=25$}
    \end{subfigure}
\caption{Logistic model: parameter estimates (a) and empirical power curves (b--d) of weighted $L_1$-Wasserstein goodness-of-fit test in scenario~2 of copula mixture model~\eqref{eq:copmixmod}}
\label{fig:results_sc2}
\end{figure}

\subsection{Hüsler--Reiss model}
\label{sec:simu-HR}

\paragraph*{Model.}
Also highly popular in extreme value analysis is the \emph{Hüsler--Reiss} model \cite{husler1989}, with stable tail dependence function~\eqref{eq:stdf}
\[
	\ell_r(x,y) = x \, \varPhi \lp r + (2r)^{-1} \log(x/y) \rp + y \, \varPhi \lp r + (2r)^{-1} \log(y/x) \rp, \qquad r \in (0,\infty),
\]
where $\varPhi$ is the standard normal distribution function, not to be confounded with the angular measure $\Phi$. Again, the model covers the full range from asymptotic full dependence ($r \ra 0$) to asymptotic independence ($r\ra\infty$). Straightforward computations give
\begin{multline*}
	\lambda_r(x,y) = \frac{1}{2ry} \phi \lp r + (2r)^{-1} \log(x/y) \rp \lp 0.5 - (4r^2)^{-1} \log(x/y) \rp \\ + \frac{1}{2rx} \phi \lp r + (2r)^{-1} \log(y/x) \rp \lp 0.5 - (4r^2)^{-1} \log(y/x) \rp, \qquad x,y > 0,
\end{multline*}
where $\phi$ is the standard normal density. The $L_p$-sphere angular measure $\Phi_{p,r}$, for $1 \leq p < \infty$, again satisfies $\Phi_p(\{0,\pi/2\}) = 0$, while on the interior of the sphere, its density can be computed via relation~\eqref{eq:phi_from_lambda}. In particular, Assumption~\ref{ass:smoothness} is satisfied for this model. Point~1 of Assumption~\ref{ass:modelAndEstimator} and Assumption~\ref{ass:density_reg} are also satisfied for $r_0 = 1$.

\paragraph*{Parameter estimation.}

Given a random sample $X_1, \ldots, X_n$, the unknown model parameter $r \in \param$ is again estimated by inverting the extremal coefficient $\chi = \ell(1,1)$ which equals $\chi = 2\varPhi(r)$ in this case, yielding
\begin{equation}
\label{eq:estimation_HR}
	\widehat{r}_n = \varPhi^{-1} \lp \frac{\widehat{\ell}_n(1,1)}{2} \rp.
\end{equation}
By the delta method, we deduce
\[
	\sqrt{k} \lp \widehat{r}_n - r \rp
	= \sqrt{k} \lp \widehat{\ell}_n(1,1) - \ell(1,1) \rp \cdot \frac{1}{2 \phi(r)} + \oh_{\PP}(1),
\]
as $n \ra \infty$, implying that point~2 of Assumption~\ref{ass:modelAndEstimator} is satisfied with
\[
	g \equiv \frac{1}{2 \phi(r)} \qquad \text{and} \qquad
	\sigma = \delta_{(1,1)}.
\]

\paragraph*{Data generating process.}
As for the logistic model, we generate data from the copula mixture model~\eqref{eq:copmixmod} with mixture weight $\lambda \in \{0, 0.1, 0.2, \ldots, 0.8\}$ and with alternative copula component $C_a \in \{C_1, C_2\}$ as in the same two scenarios, while 
\begin{equation}
\label{eq:simuC0HR}
	C_0(u,v) \egdef \exp \lc \varPhi \lacc 0.5 + \log \lp \frac{\log v}{\log u} \rp \racc \log u + \varPhi \lacc 0.5 + \log \lp \frac{\log u}{\log v} \rp \racc \log v \rc
\end{equation}
is the \emph{Hüsler--Reiss} copula with parameter $r_0 = 1$.

\begin{figure}[h]
\begin{subfigure}{0.49\textwidth}
    \includegraphics[width=\textwidth]{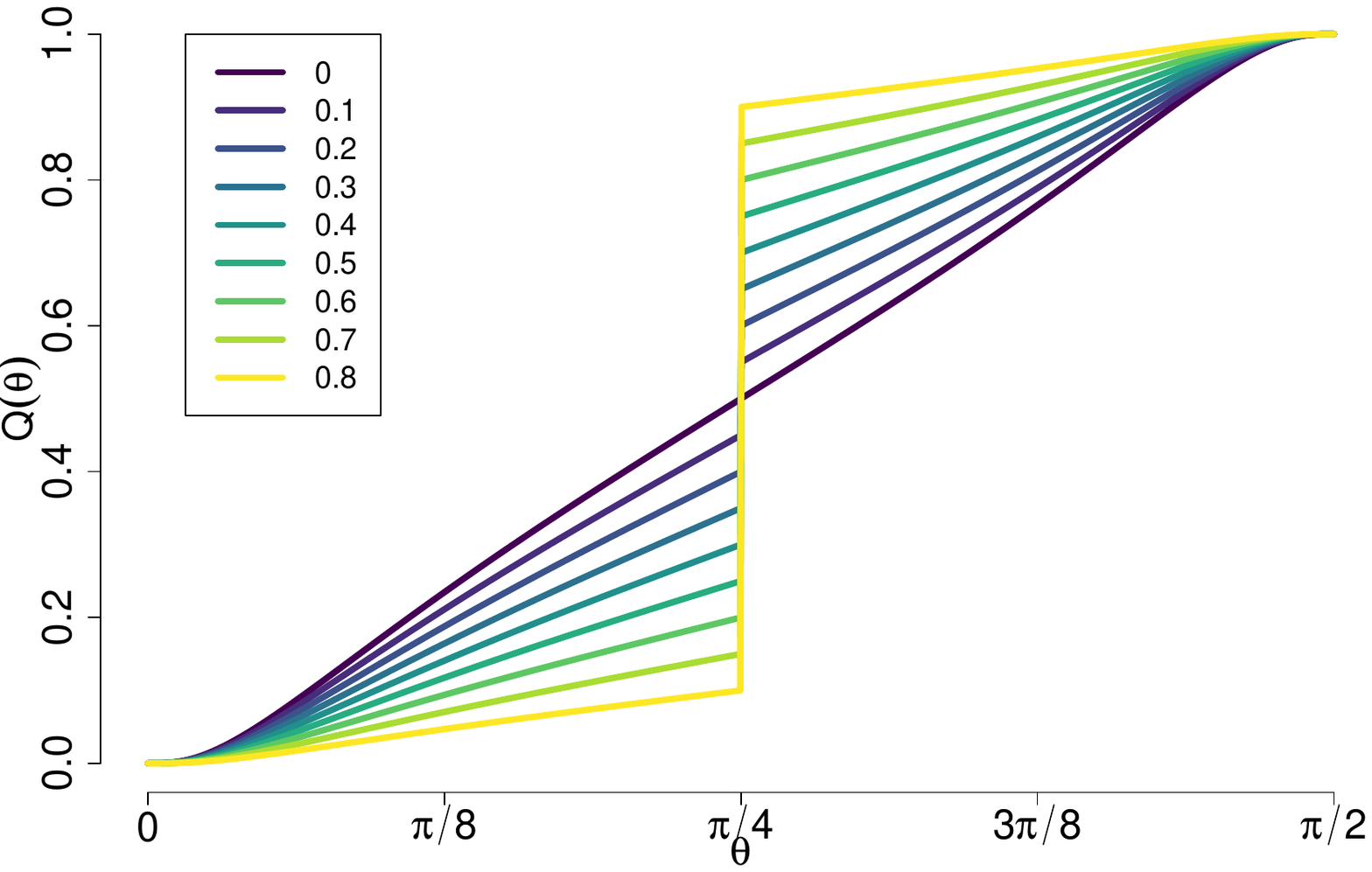}
    \caption{Scenario 1, $C_a = C_1$}
    \label{fig:sc1_HR}
\end{subfigure}
\hfill
\begin{subfigure}{0.49\textwidth}
    \includegraphics[width=\textwidth]{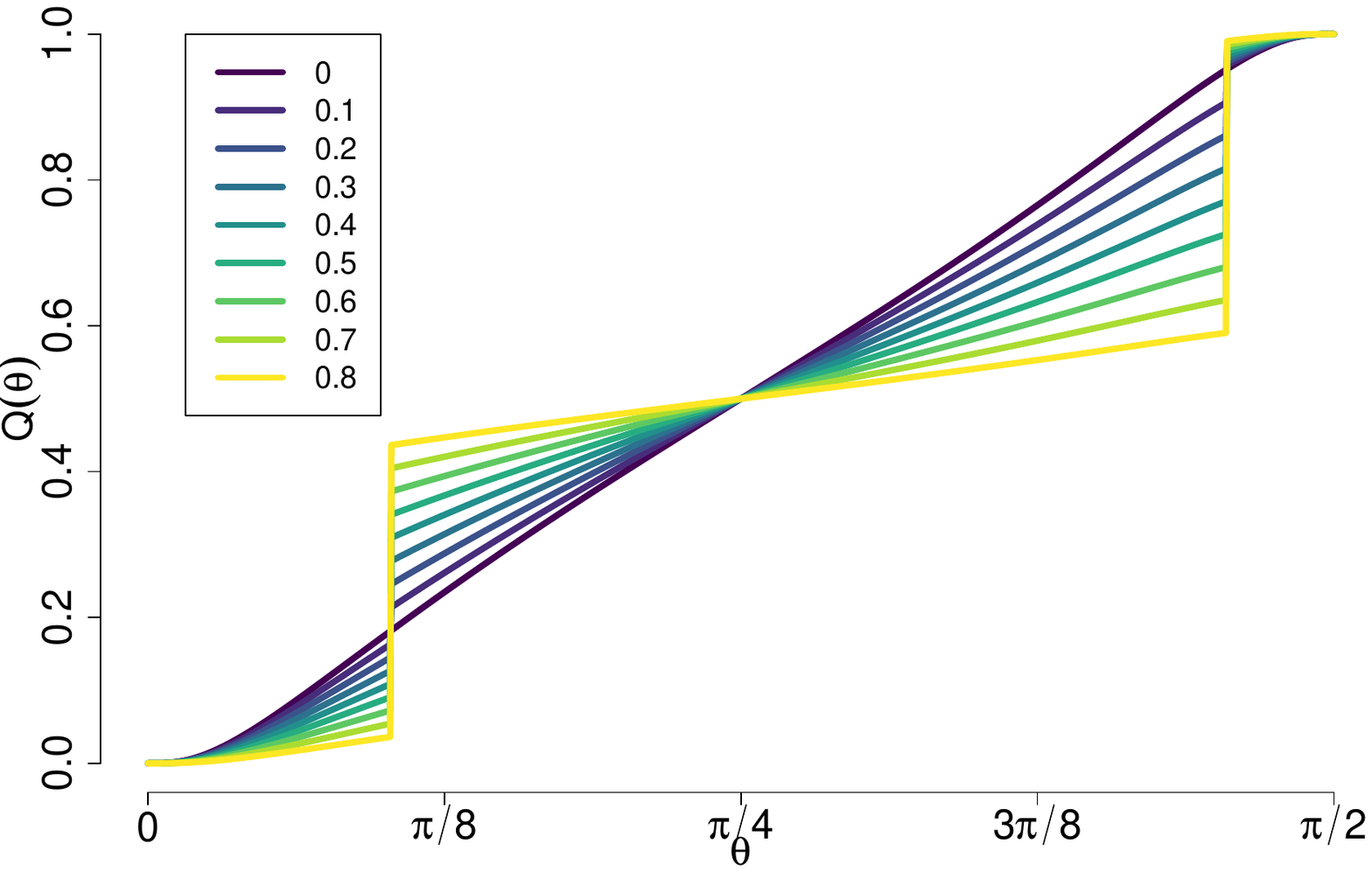}
    \caption{Scenario 2, $C_a = C_2$}
    \label{fig:sc2_HR}
\end{subfigure}
\caption{Angular distribution functions $\theta \mapsto Q_p(\theta)$ for $p=2$ of the copula mixture model~\eqref{eq:copmixmod} with Hüsler--Reiss copula $C_0$ in \eqref{eq:simuC0HR} and under two scenarios for $C_a$.}
\label{fig:sc_HR}
\end{figure}

The null hypothesis of the test~\eqref{eq:THEtest} is satisfied for the Hüsler--Reiss model with parameter $r_0 = 1$ when $\lambda = 0$ and Assumption~\ref{ass:bias} holds true if we choose $\alpha = 1$, while for $0 < \lambda < 1$, the null hypothesis is not satisfied anymore. For the two scenarios, the angular distribution functions $\theta \mapsto Q_p(\theta)$ at $p=2$ of the copula mixture model are shown in Figure~\ref{fig:sc_HR}.

\paragraph*{Results.}

The results for scenarios~1 and~2 are shown in Figures~\ref{fig:results_HR_sc1} and~\ref{fig:results_HR_sc2}, respectively. The settings are the same as those for the logistic model in Section~\ref{sec:simu-logistic}. The results are similar, even though for $\lambda = 0$ and $n = 500$, the test rejects the null hypothesis too often. In scenario~2, the estimated power remains higher than in scenario~1, despite the fact that the distribution function of the angular measure 
for the mixing component $C_2$ resembles that of the Hüsler–Reiss model with $r_0 = 1$ more.

\begin{figure}[h]
    \centering
    \begin{subfigure}{0.45\textwidth}
        \centering
   		\includegraphics[scale=0.2]{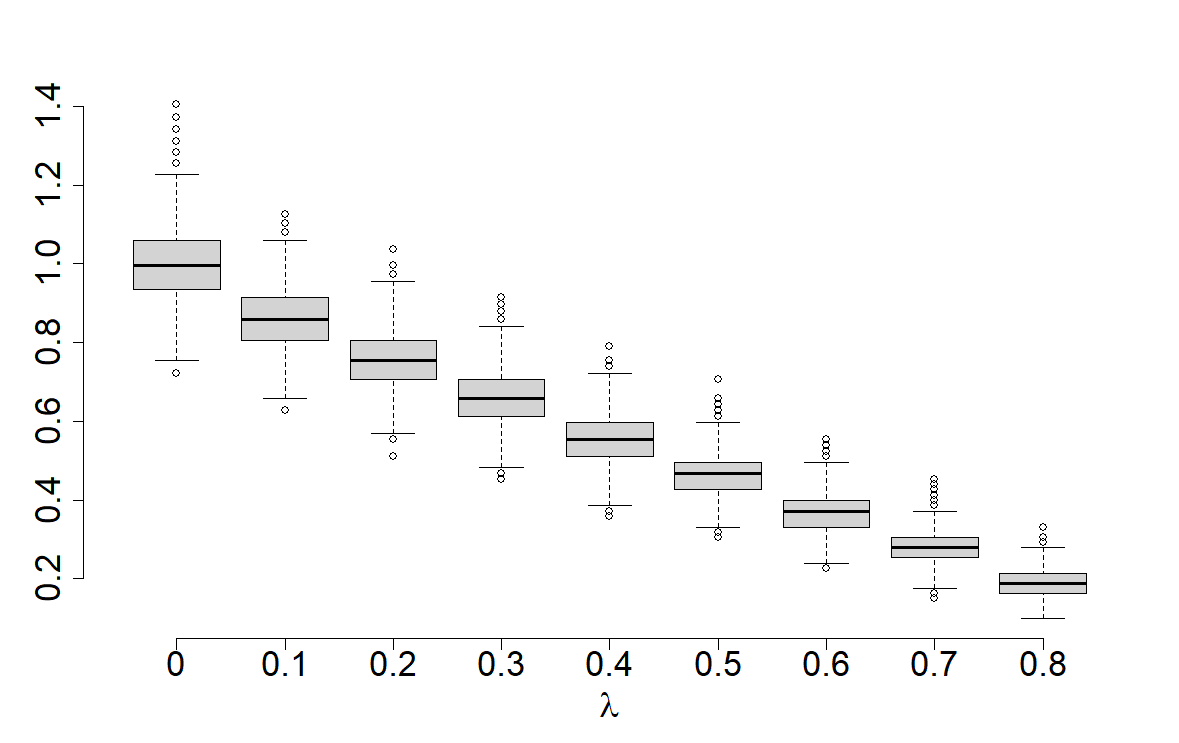}
   		\caption{Parameter estimates, $n=10\,000$ and $k=100$}
    \end{subfigure}
    \hspace{0.05\textwidth}
    \begin{subfigure}{0.45\textwidth}
        \centering
        \includegraphics[scale=0.2]{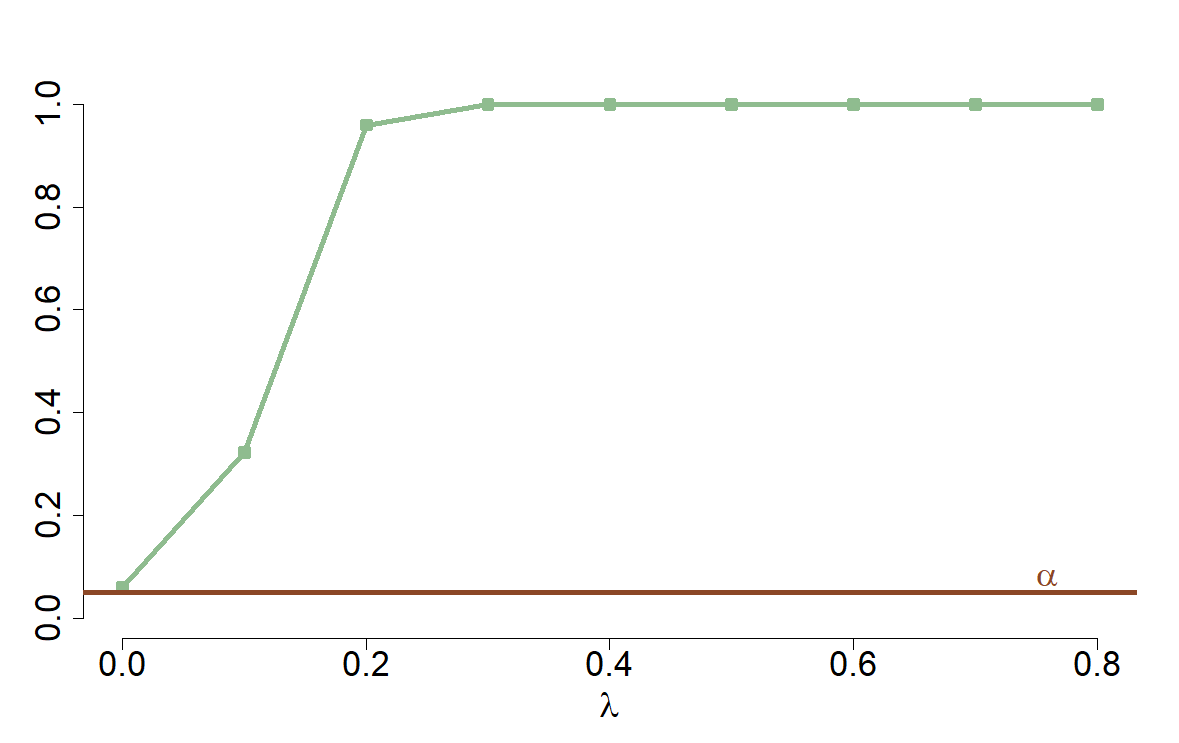}
   		\caption{Power curve, $n=10\,000$ and $k=100$}
    \end{subfigure}
    \vspace{0.5cm}
    \begin{subfigure}{0.45\textwidth}
        \centering
  			\includegraphics[scale=0.2]{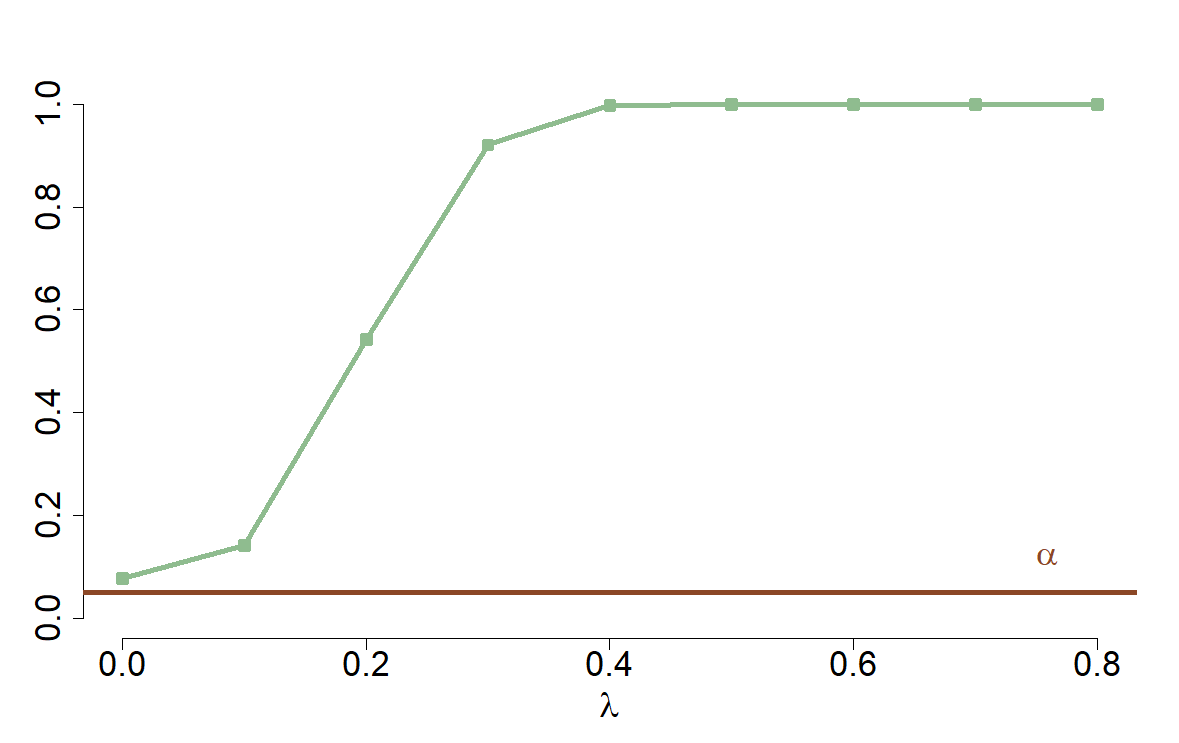}
   		\caption{Power curve at $n=3\,000$ and $k=50$}
    \end{subfigure}
    \hspace{0.05\textwidth}
    \begin{subfigure}{0.45\textwidth}
        \centering
   		\includegraphics[scale=0.2]{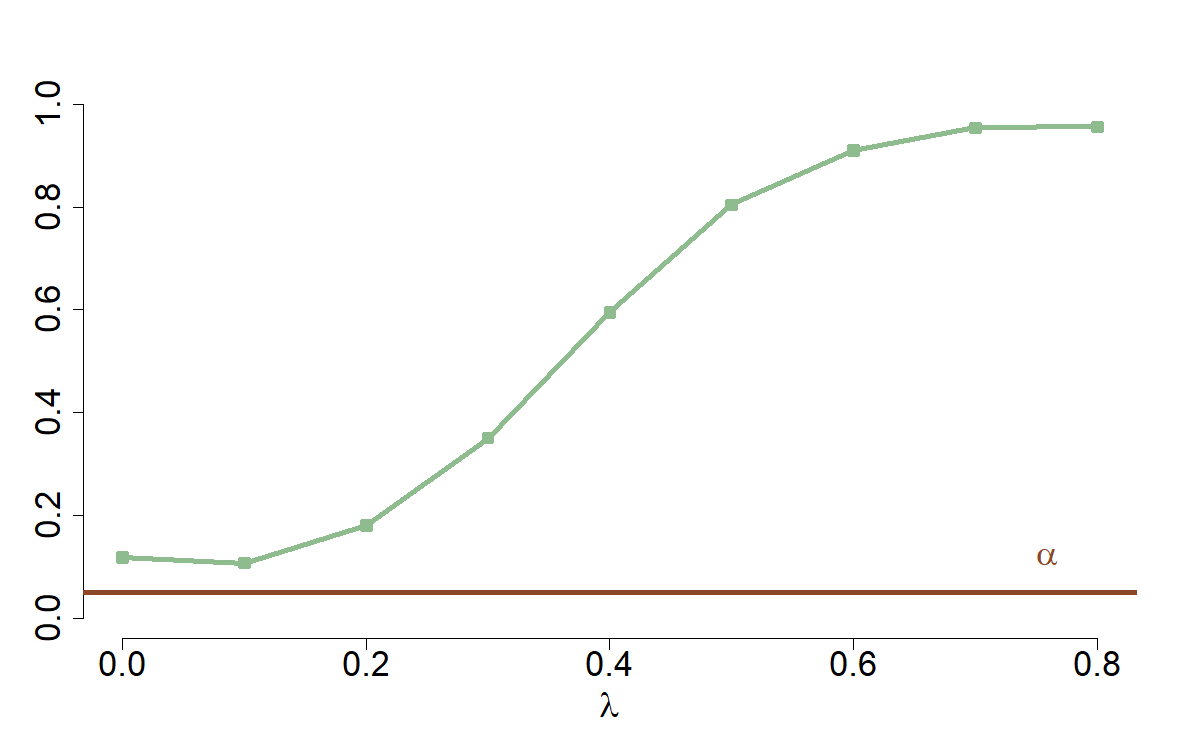}
   		\caption{Power curve at $n=500$ and $k=25$}
    \end{subfigure}
\caption{Hüsler--Reiss model: parameter estimates (a) and empirical power curves (b--d) of weighted $L_1$-Wasserstein goodness-of-fit test in scenario~1 of copula mixture model~\eqref{eq:copmixmod}}
\label{fig:results_HR_sc1}
\end{figure}
\begin{figure}[h]
    \centering
    \begin{subfigure}{0.45\textwidth}
        \centering
   		\includegraphics[scale=0.2]{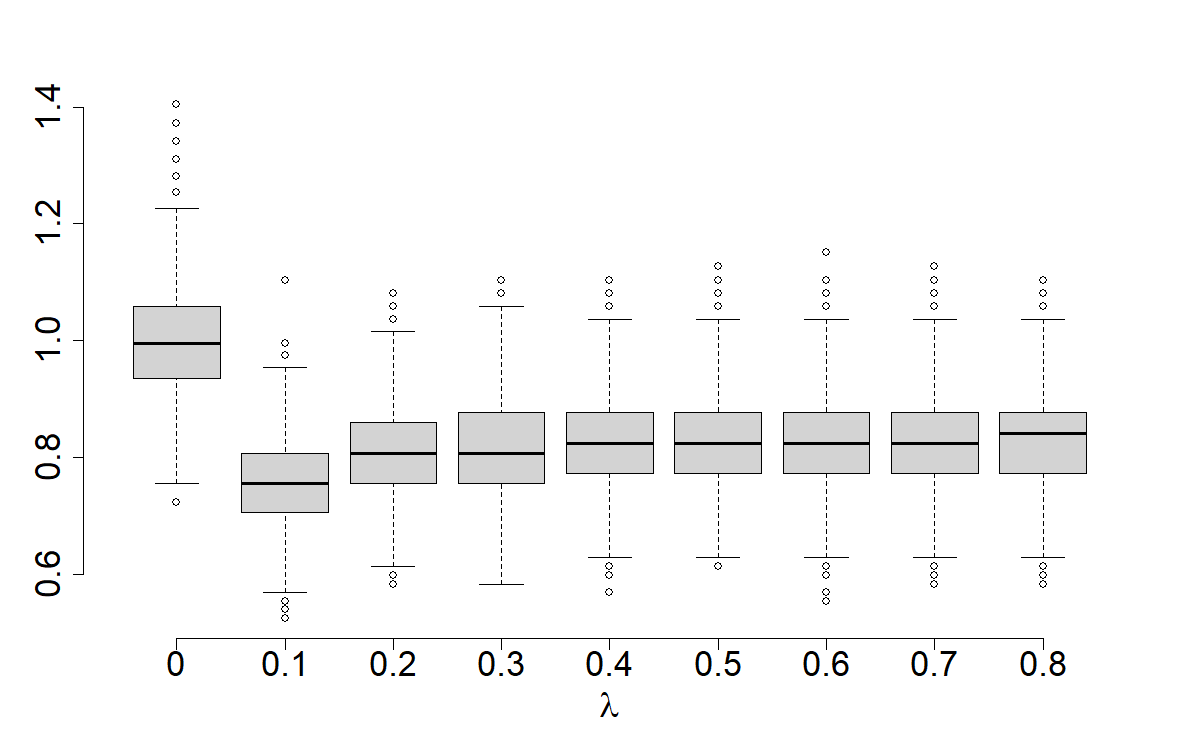}
   		\caption{Parameter estimates, $n=10\,000$ and $k=100$}
    \end{subfigure}
    \hspace{0.05\textwidth}
    \begin{subfigure}{0.45\textwidth}
        \centering
        \includegraphics[scale=0.2]{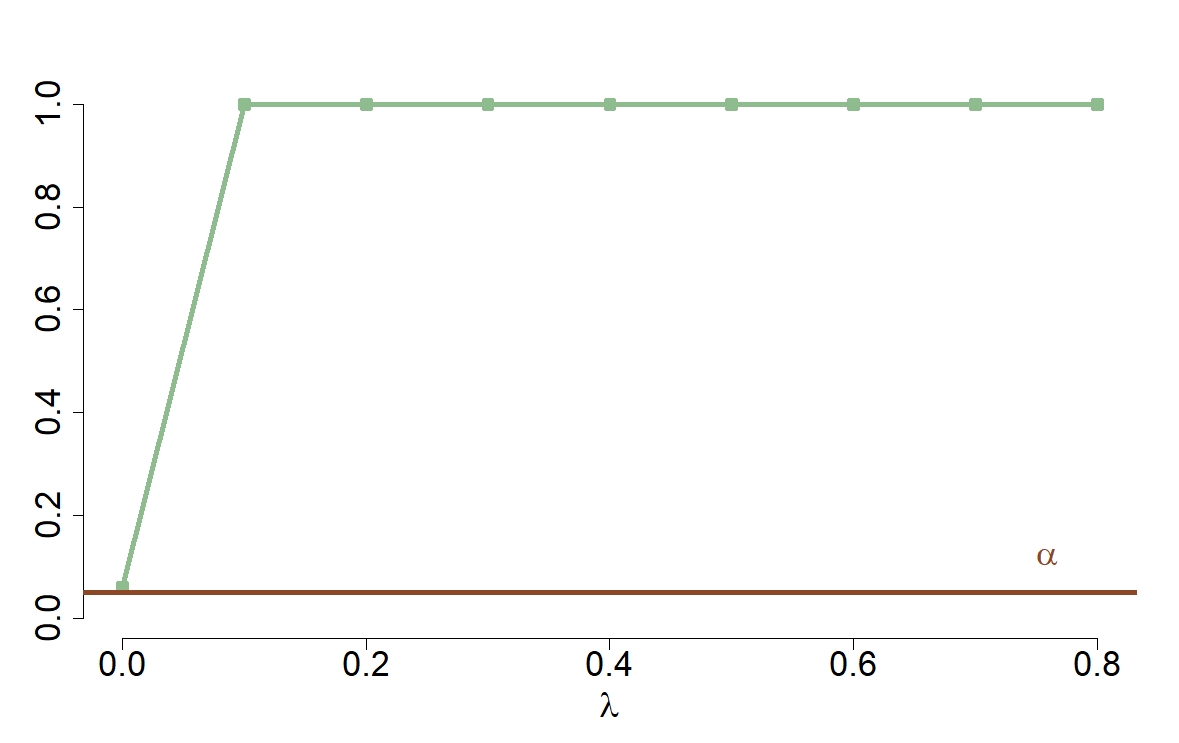}
   		\caption{Power curve, $n=10\,000$ and $k=100$}
    \end{subfigure}
    \vspace{0.5cm}
    \begin{subfigure}{0.45\textwidth}
        \centering
  			\includegraphics[scale=0.2]{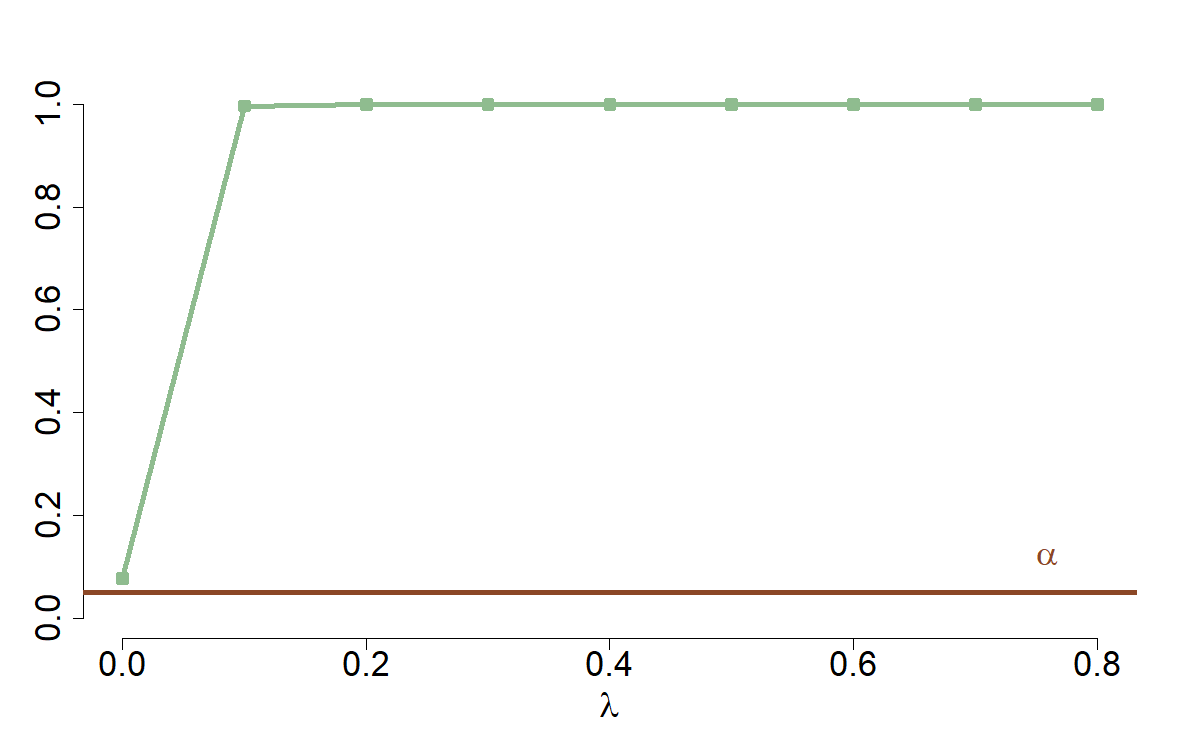}
   		\caption{Power curve, $n=3000$ and $k=50$}
    \end{subfigure}
    \hspace{0.05\textwidth}
    \begin{subfigure}{0.45\textwidth}
        \centering
   		\includegraphics[scale=0.2]{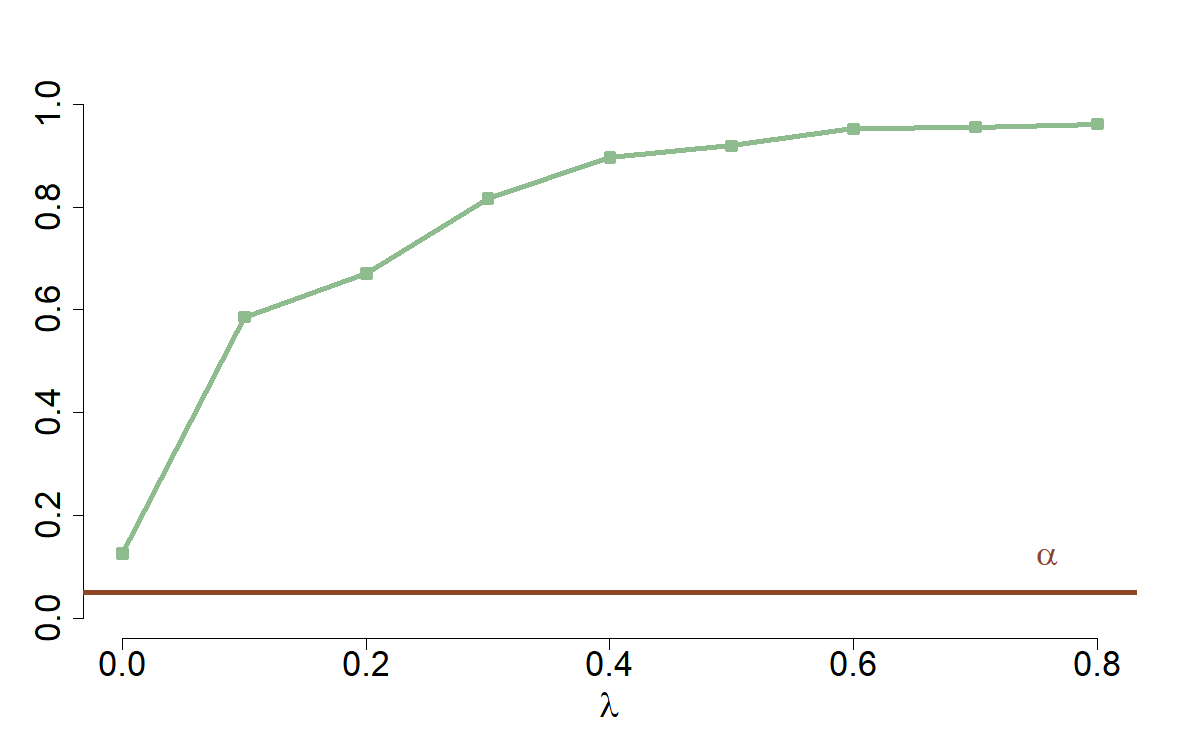}
   		\caption{Power curve, $n=500$ and $k=25$}
    \end{subfigure}
\caption{Hüsler--Reiss model: parameter estimates (a) and empirical power curves (b--d) of weighted $L_1$-Wasserstein goodness-of-fit test in scenario~1 of copula mixture model~\eqref{eq:copmixmod}}
\label{fig:results_HR_sc2}
\end{figure}

\section{Testing the Hüsler--Reiss model for river discharge data}
\label{sec:application}

We consider the \texttt{danube} dataset from the \texttt{R} package \texttt{graphicalExtremes} \cite{graphicalExtremes}. 
Average daily discharges are recorded at 31 gauging stations in the upper Danube basin covering parts of Germany, Austria and Switzerland which are often affected by flooding. To not have to deal with seasonality, we only consider the daily discharges in the summer months (June, July and August), and we use the declustered time series from the \texttt{data\_clustered} dataset. The same data have already been analyzed by several authors, e.g.~\cite{hu2024}, who estimate bivariate dependence structures using the Hüsler--Reiss model for connected pairs of stations after estimating a Markov tree on the full dataset. The topographic map of the upper Danube basin which can be found in Figure~1 of~\cite{asadi2015} is displayed in Figure~\ref{fig:danube} together with the associated flow chart of its river network.

Following the approach of~\cite{hu2024}, our goal is to assess the goodness-of-fit of the Hüsler--Reiss model to all $m = 30$ adjacent pairs in the flow chart of Figure~\ref{fig:danube_fc}. Naturally, this leads to multiple testing challenges due to the many dependent tests involved. The p-values can be adjusted using either the classical Bonferroni correction or the False Discovery Rate (FDR) paradigm of Benjamini and Hochberg~\cite{benjamini1995}, adapted for dependent tests in~\cite{benjamini2001}. The sample size is $n = 428$, and we set $k = \sqrt{n}$ as we observed in Section~\ref{sec:simu-HR} that this typically leads to better results for the nonparametric estimation of $Q_p$, while still providing a reasonable estimate of the nuisance parameter $r$, with respect to which the quantile values $\quantile_{r}(1-\alpha)$ do not vary too quickly. Ideally, we would like to use two different values of $k$ for the nonparametric and parametric estimators in the test statistic $T_n$, as a larger $k$ often improves the accuracy of the parameter estimators. This, however, is a topic for future research.
We estimate the p-values $P_j$ corresponding to each of the $j \in \{1,\ldots,m\}$ tests by
\[
	\widehat{P}_j = \frac{1}{B} \sum_{b=1}^{B} \1 \lacc T_n > L_{\widehat{r}_n}^{(b)} \racc
\]
for $B=4\,000$, where $L_{\widehat{r}_n}^{(b)}$ is drawn from the distribution of $L_r$ in Theorem~\ref{thm:asymptoticTS} at $r=\widehat{r}_n$, computed in the same way as in Section~\ref{sec:simu-HR}. Although alternative parameter estimators could be used, this would result in more complex expressions for $I_{g,\sigma,r}$ in the definition of $T_n$, and we aimed to keep the methodology simple. Results are shown in Table~\ref{tab:p-values}.

\begin{figure}[H]
	\begin{subfigure}{0.4\textwidth}
		\includegraphics[width=\textwidth]{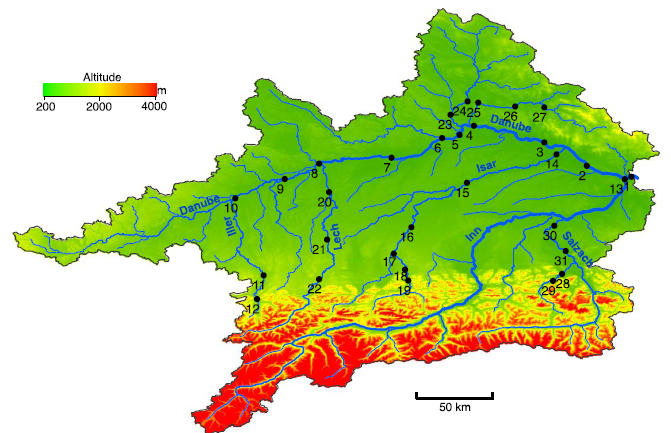}
		\caption{Topographic map from~\cite{asadi2015}}
		\label{fig:danube_topography}
	\end{subfigure}
	\hfill
	\begin{subfigure}{0.4\textwidth}
		\includegraphics[width=\textwidth]{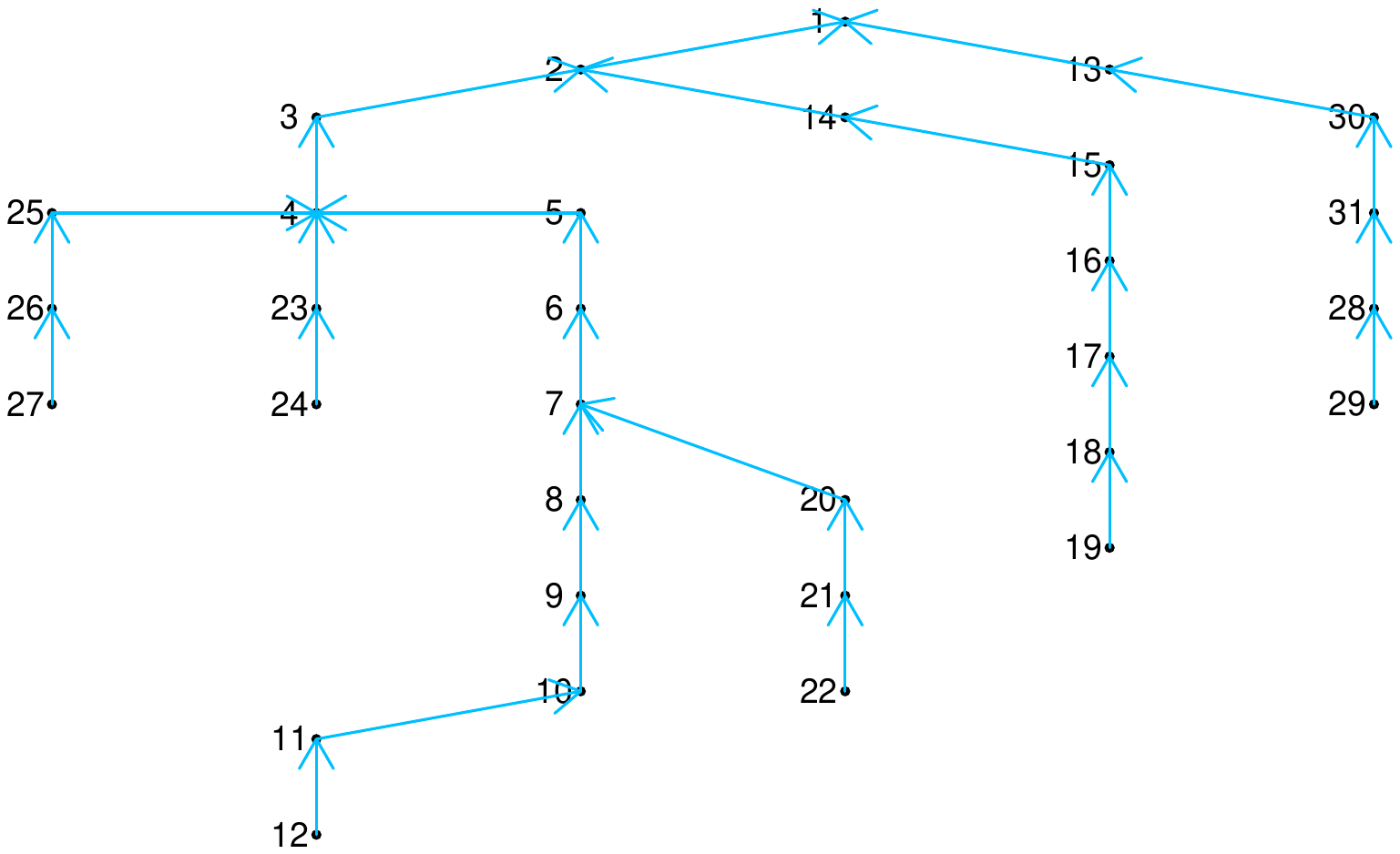}
		\caption{Flow chart}
		\label{fig:danube_fc}
	\end{subfigure}
	\caption{The upper Danube basin}
	\label{fig:danube}
\end{figure}

%
%

\begin{table}[h]
\centering
\begin{minipage}{.5\linewidth}
\centering
\begin{tabular}{r@{--}lr}
  \toprule
  \multicolumn{2}{l}{edge} & p-value \\ 
  \midrule
  2 & 1 & 0.02500 \\ 
  13 & 1 & 0.60825 \\ 
  3 & 2 & 0.75750 \\ 
  14 & 2 & 0.00000 \\ 
  4 & 3 & 0.91175 \\ 
  5 & 4 & 0.34800 \\ 
  23 & 4 & 0.00150 \\ 
  25 & 4 & 0.43725 \\ 
  6 & 5 & 0.67950 \\ 
  7 & 6 & 0.64925 \\ 
  8 & 7 & 0.40375 \\ 
  20 & 7 & 0.76375 \\ 
  9 & 8 & 0.51550 \\ 
  10 & 9 & 0.85875 \\ 
  11 & 10 & 0.16925 \\ 
  \hline
\end{tabular}
\end{minipage}%
\begin{minipage}{.5\linewidth}
\centering
\begin{tabular}{r@{--}lr}
  \toprule
	\multicolumn{2}{l}{edge} & p-value \\ 
	\midrule
 12 & 11 & 0.05425 \\ 
  30 & 13 & 0.07450 \\ 
  15 & 14 & 0.49675 \\ 
  16 & 15 & 0.84725 \\ 
  17 & 16 & 0.84875 \\ 
  18 & 17 & 0.25175 \\ 
  19 & 18 & 0.15525 \\ 
  21 & 20 & 0.17775 \\ 
  22 & 21 & 0.15525 \\ 
  24 & 23 & 0.61300 \\ 
  26 & 25 & 0.72550 \\ 
  27 & 26 & 0.78075 \\ 
  29 & 28 & 0.19350 \\ 
  31 & 28 & 0.57775 \\ 
  31 & 30 & 0.33550 \\ 
   \hline
\end{tabular}
\end{minipage}
\caption{p-values of the weighted $L_1$-Wasserstein goodness-of-fit test of the bivariate Hüsler--Reiss model for the 30 pairs of adjacent gauging stations in the Danube data}
\label{tab:p-values}
\end{table}

With the Bonferroni correction, the Hüsler--Reiss model would only be rejected for edges 14--2 and 23--4, using the standard significance level of $\alpha = 5\%$. In fact, even at the individual test level, a large proportion of these p-values are significantly higher than the rejection threshold $\alpha$. In Figure~\ref{fig:2-14}, we compare the nonparametric estimate of the angular distribution function $\theta \in [0, \pi/2] \mapsto Q_2(\theta)$ with its estimated counterpart under the null hypothesis of the Hüsler--Reiss model for edges 14--2 and 23--4. The asymmetry in the angular distribution function observed in the data for these edges cannot be captured by the assumed Hüsler--Reiss model. This leads to the rejection of the null hypothesis and suggests that this parametric form may not be well-suited for these pairs of stations.

\begin{figure}[H]
	\begin{subfigure}{0.49\textwidth}
		\includegraphics[width=\textwidth]{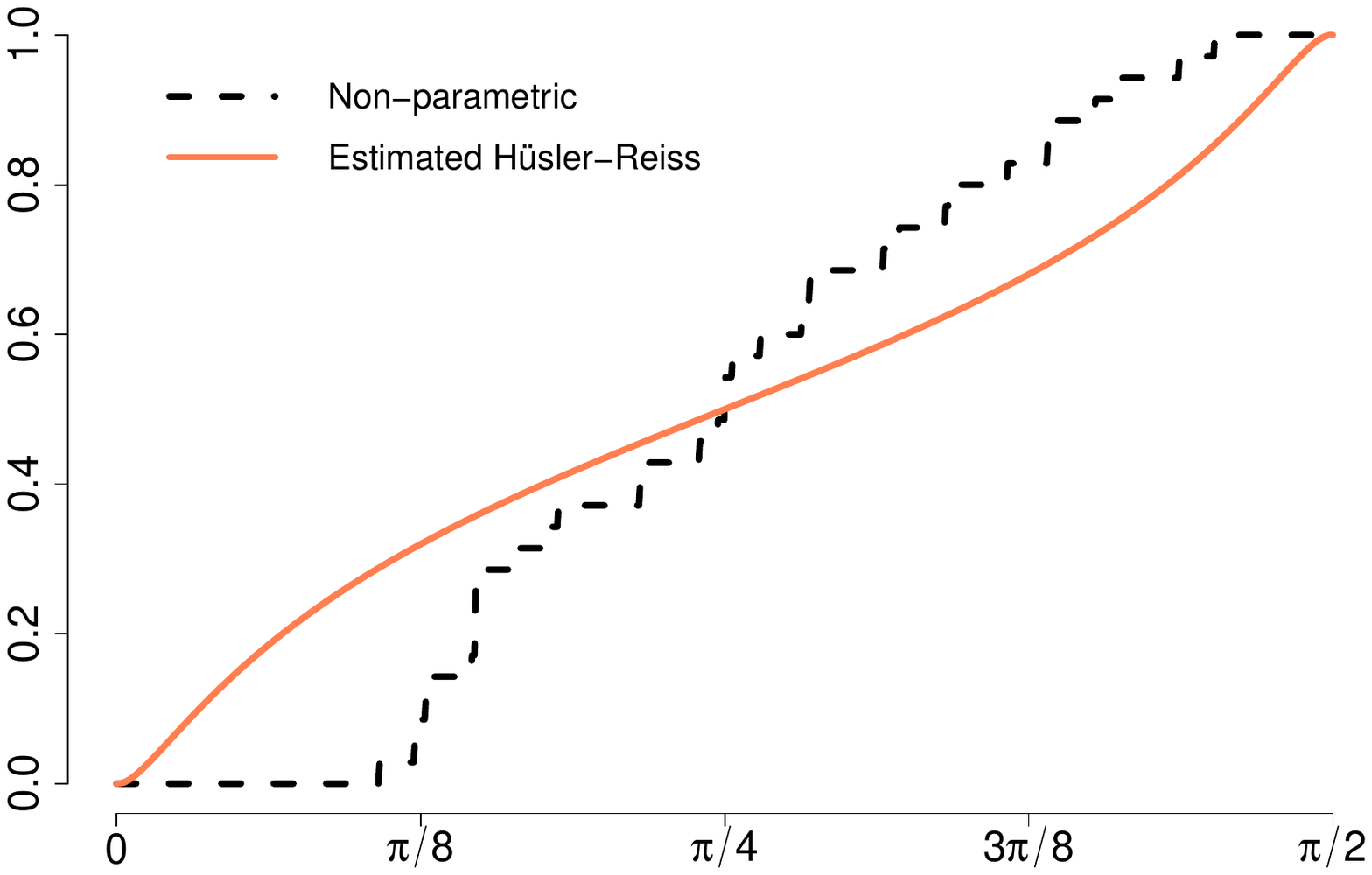}
	\end{subfigure}
	\hfill
	\begin{subfigure}{0.49\textwidth}
		\includegraphics[width=\textwidth]{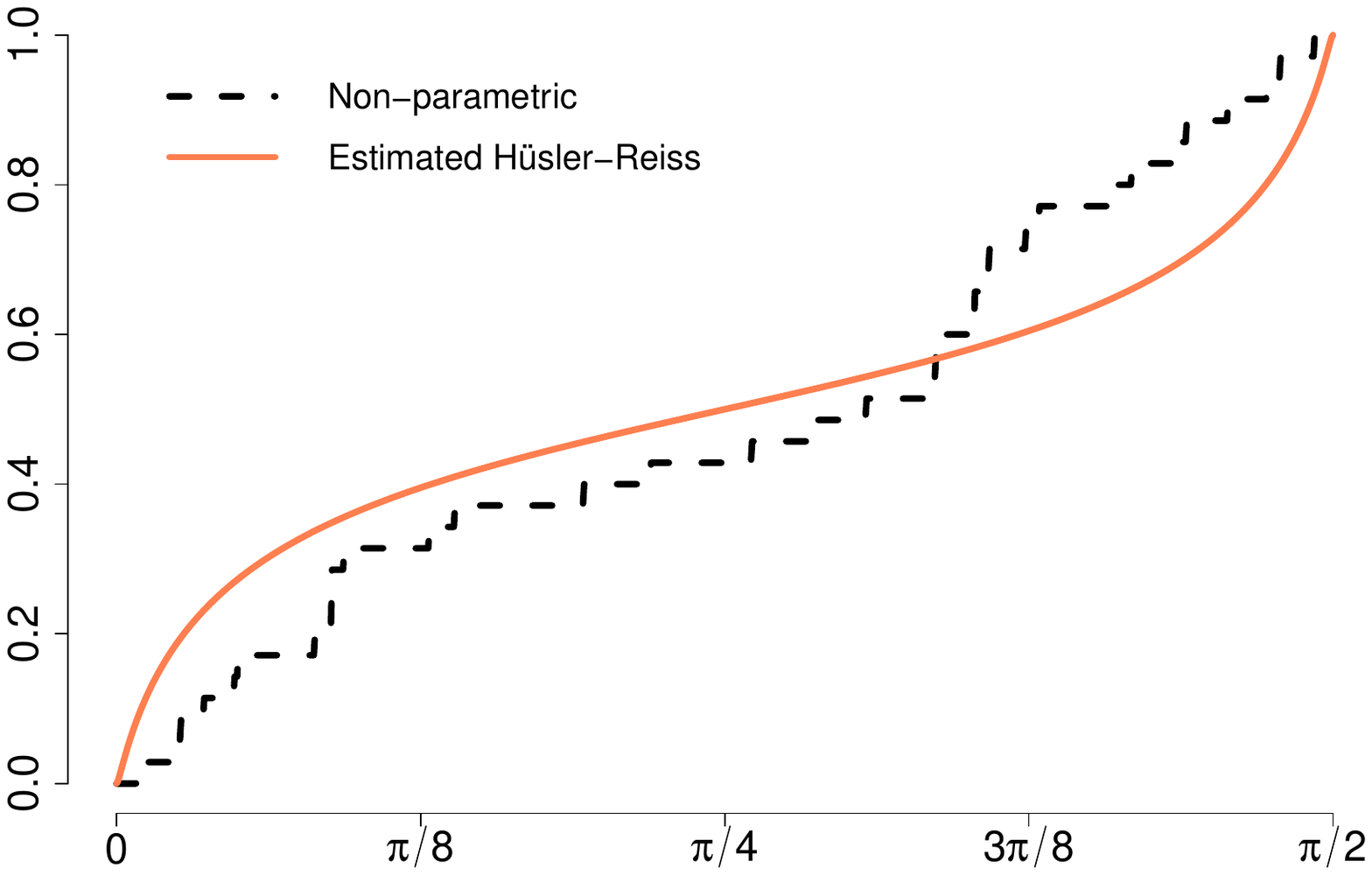}
	\end{subfigure}
	\caption{Danube data: estimated angular distributions functions $\theta \in [0,\pi/2] \mapsto Q_2(\theta)$ for edges 14--2 and 23--4 of gauging stations}
	\label{fig:2-14}
\end{figure}

\section{Conclusion}
\label{sec:conclusion}

We propose a new inference method based on the angular measure to evaluate whether the asymptotic dependence structure of a bivariate dataset aligns with a given parametric model. Unlike previous approaches, our method involves sampling directly from the asymptotic distribution of the test statistic under the estimated parameter value. Leveraging continuity of the asymptotic quantile function, this leads to a detailed proof of the method's consistency. The finite-sample performance is also compelling when compared to other techniques, and the ease of computing the test statistic makes it a strong candidate for practical applications. Although the method is currently limited to bivariate data, many multivariate extreme value methods rely on multiple bivariate analyses, and we demonstrate how this can be done on a real dataset. 

Extending the procedure to the general multivariate case is certainly of interest but requires a deeper understanding of the asymptotic behavior of the empirical angular measure, about which little is currently known. A promising approach would be to focus on discrete angular measures, for which more concrete results might be achievable. Using the Wasserstein distance as a test statistic to quantify the discrepancy between $\tQ_p$ and $Q_{p,\widehat{r}_n}$ could be effective, leveraging results from \cite{sommerfeld2018}, where the Wasserstein distance is shown to be Hadamard differentiable with respect to its arguments for measures with discrete supports.

\section*{Declarations}

\subsection*{Funding}

The research of Stéphane Lhaut was supported by the Fonds National de la Recherche Scientifique (FNRS, Belgium) within the framework of a FRIA grant (number 1.E114.23).

\subsection*{Competing interests}
The authors have no competing interests to declare that are relevant to the content of this article.

\subsection*{Authors contribution}
The two authors contributed to the analysis of the results and to the writing of the manuscript.

\bibliographystyle{acm}
\bibliography{references.bib}

\newpage
\appendix

\section{Proof of Corollary~\ref{cor:asymptoticsMEL}}
\label{sec:app:expansionMEL}

The second part is just an application of the functional delta method, i.e., Theorem~20.8 in~\cite{VVV1998}. The only thing we need to show is Hadamard differentiability of the map $\Psi$, that is, let $h \in C_0([0,\pi/2])$ and $(h_t)_{t>0}$ be any sequence such that $\sup_{\theta \in [0,\pi/2]} |h_t(\theta) - h(\theta)| \ra 0$ as $t \ra 0$, then
\begin{equation}
\label{eq:hadamard}
	\sup_{\theta \in [0,\pi/2]} \labs \frac{(\Psi(Q_p + th_t))(\theta) - (\Psi(Q_p))(\theta)}{t} - (\Psi'_{Q_p}[h])(\theta) \rabs \ra 0, \qquad t \ra 0.
\end{equation}
Noting the useful relation
\[
	\int_0^\theta f(\psi) \d F(\psi) = F(\theta)f(\theta) - \int_0^\theta F(\varphi) f'(\varphi) \d\varphi,
\]
the candidate for the derivative is obtained by computing
\[
	(\Psi'_{Q_p}[h])(\theta) = \frac{\d}{\d t} (\Psi(Q_p + th))(\theta) \Big|_{t = 0},
\]
which leads to the indicated formula. Showing that~\eqref{eq:hadamard} holds with this candidate is an easy exercise and completes the proof.

\section{Proof of Theorem~\ref{thm:expansionSTDF}}
\label{sec:app:expansionSTDF}

To facilitate the notation later in the proof, we introduce the following functions: for $j \in \{1,\ldots, d\}$,
\begin{enumerate}[label=\textbullet]
	\item $Q_{nj}(u_j) := U_{\lceil nu_j \rceil:n,j}, \qquad u \in [0,1]^d$ ;
	\item $S_{nj}(x_j) := \frac{n}{k} Q_{nj} \lp \frac{k}{n} x_j \rp, \quad S_n(x) := (S_{n1}(x_1),\ldots, S_{nd}(x_d)), \quad x \in [0,\infty)^d$ ;
\end{enumerate}
where $U_{1:n,j} \leq U_{2:n,j} \leq \ldots \leq U_{n:n,j}$ are the order statistics of the marginal sample $U_{1j},\ldots,U_{nj}$ and $\lceil a \rceil$ is the smallest integer not smaller than $a \in \R$. We also write for $x \in [0,\infty)^d$,
\begin{align*}
	V_n(x) &:= \frac{n}{k} \PP \lc U_1 \leq \frac{k}{n} x_1 \text{ or } \ldots \text{ or } U_d \leq \frac{k}{n} x_d \rc ; \\
	T_n(x) &:= \frac{n}{k} \frac{1}{n} \sum_{i=1}^n \1 \lacc U_{i1} \leq \frac{k}{n} x_1 \text{ or } \ldots \text{ or } U_{id} \leq \frac{k}{n} x_d \racc ; \\
	\hat{L}_n(x) &:= \frac{n}{k} \frac{1}{n} \sum_{i=1}^n \1 \lacc U_{i1} \leq \frac{k}{n} S_{n1}(x_1) \text{ or } \ldots \text{ or } U_{id} \leq \frac{k}{n} S_{nd}(x_d) \racc \\
	&= \frac{1}{k} \sum_{i=1}^n \1 \lacc R_{i1} > n + 1 - kx_1 \text{ or } \ldots \text{ or } R_{id} > n + 1 - kx_d \racc.
\end{align*}
With this notation, we note that $v_n(x) = \sqrt{k} \{T_n(x) - V_n(x)\}$ and $\hat{L}_n(x) = T_n(S_n(x))$. As claimed in~\cite[Equation~(7.1)]{einmahl2012}, it is easily seen that the asymptotic properties of $\ell_n$ and $\hat{L}_n$ are the same so that we will consider the process $\sqrt{k}(\hat{L}_n - \ell)$ instead.

We consider the following decomposition: for $x \in [0,\infty)^d$,
\begin{align*}
	\sqrt{k}(\hat{L}_n(x) - \ell(x)) &= \sqrt{k} \lp T_n(S_n(x)) - V_n(S_n(x)) \rp \\
	&\qquad + \sqrt{k} \lp V_n(S_n(x)) - \ell(S_n(x)) \rp \\
	&\qquad + \sqrt{k} \lp \ell(S_n(x)) - \ell(x) \rp, 
\end{align*}
where we call the first term of the decomposition the \emph{stochastic term} (and denote it by $D_1(x)$), the second term will be called the \emph{bias term} (and will be denoted by $D_2(x)$) and the last term will be called the \emph{random set term} (and will be denoted by $D_3(x)$).

\paragraph*{Skorokhod construction.}
As in~\cite{einmahl2012}, we will work on a Skorokhod construction, without changing the notation to keep the proof readable.

For fixed $T > 0$, Theorem~3.1 in~\cite{einmahl1997} ensures that on a Skorokhod construction (depending on $T$), there exists a sequence of processes $(v_n)$ and a Wiener process $W_l(x) = W_\Lambda(A_x)$, where $W_\Lambda$ is the $d$-dimensional analog to the one appearing in Theorem~\ref{thm:asymptoticsEAM}, such that as $n \ra \infty$,
\begin{equation}
\label{eq:skorokhod_stdf}
	\sup_{x \in [0,2T]^d} \labs v_n(x) - W_l(x) \rabs \xrightarrow{a.s.} 0.
\end{equation}
In particular, this implies the marginal convergences
\begin{equation}
	\label{eq:marginal_processes}
	\sup_{x_j \in [0,2T]} \labs v_{nj}(x_j) - W_j(x_j) \rabs \xrightarrow{a.s.} 0, \qquad j \in \{1,\ldots,d\},
\end{equation}
where $W_j(x_j) = W_l(0,\ldots,0,x_j,0,\ldots,0)$. By Vervaat's Lemma~\cite[Lemma~A.0.2]{dHF2006}, we also have the convergence of the marginal quantile processes:
\begin{equation}
	\label{eq:marginal_quantile_processes}
	\sup_{x_j \in [0,2T]} \labs \sqrt{k} \lp S_{nj}(x_j) - x_j \rp + W_j(x_j) \rabs \xrightarrow{a.s.} 0, \qquad j \in \{1,\ldots,d\}.
\end{equation}

\paragraph*{Asymptotic expansions of the bias term.}
We show that $D_2$ is asymptotically negligible for our purposes, that is, $\sup_{x \in [0,T]^d} |D_2(x)| = \oh_{\PP}(1)$ as $n \ra \infty$. It follows from~\eqref{eq:marginal_quantile_processes} that $S_{nj}(x_j) = x_j + \Oh_{\PP}(1/\sqrt{k})$ where the $\Oh_{\PP}$ term is uniform in $x_j \in [0,T]$ so that, with probability tending to one, we have
\begin{align*}
	\sup_{x \in [0,T]^d} |D_2(x)| &\leq \sup_{y \in [0,2T]^d} \sqrt{k} \labs V_n(y) - \ell(y) \rabs \\
	&\leq 2T \sqrt{k} \sup_{w \in [0,2T]^d, \|w\|_1 = 1} \labs V_n(w) - \ell(w) \rabs.
\end{align*}
By points 1 and 2 of Assumption~\ref{ass:stdf_expansion}, the latter expression is of the order
\[
	\sqrt{k} \Oh \lp (k/n)^\alpha \rp = \Oh \lp \lp \frac{k}{n^{2\alpha/(1+2\alpha)}} \rp^{1/2 + \alpha} \rp = \oh(1)
\]
as $n \ra \infty$.

\paragraph*{Asymptotic expansions of the stochastic term.}
Note that for any $x \in [0,T]^d$, $D_1(x) = v_n(S_n(x))$. Our aim is to show that
\[
\sup_{x \in [0,T]^d} \labs v_n(S_n(x)) - v_n(x) \rabs = \oh_{\PP}(1),
\]
as $n \ra \infty$. By the triangle inequality, we have
\begin{align*}
	\sup_{x \in [0,T]^d} \labs v_n(S_n(x)) - v_n(x) \rabs &\leq \sup_{x \in [0,T]^d} \labs v_n(S_n(x)) - W_l(S_n(x)) \rabs \\
	&\quad + \sup_{x \in [0,T]^d} \labs W_l(S_n(x)) - W_l(x) \rabs \\
	&\quad + \sup_{x \in [0,T]^d} \labs W_l(x) - v_n(x) \rabs.
\end{align*}
Arguing as for the bias term, by~\eqref{eq:marginal_quantile_processes}, we see that the first term in this decomposition is bounded with probability tending to one by $\sup_{y \in [0,2T]^d} |v_n(y) - W_l(y)|$ which is $\oh_{\PP}(1)$ as $n \ra \infty$, by~\eqref{eq:skorokhod_stdf}. The second term of this decomposition is $\oh_{\PP}(1)$ as $n \ra \infty$ by uniform continuity of the limiting process $W_l$. Finally, the last term of this decomposition is immediately $\oh_{\PP}(1)$ as $n \ra \infty$ by~\eqref{eq:skorokhod_stdf}.

\paragraph*{Asymptotic expansions of the random set term.}
It follows from point 3 in Assumption~\ref{ass:stdf_expansion} and the mean value theorem that for $x \in [0,T]^d$,
\[
	D_3(x) = \sqrt{k} \lp \ell(S_n(x)) - \ell(x) \rp = \sum_{j=1}^d \sqrt{k} \lp S_{nj}(x_j) - x_j \rp \dot{\ell}_j(\xi_n),
\]
for some $\xi_n$ such that $\xi_{nj} \in (x_j, S_{nj}(x_j))$ for any $j=1,\ldots,d$.
Hence, by the triangle inequality, we are done proving our expansion if we show
\begin{align*}
	\sup_{x \in [0,T]^d} \sum_{j=1}^d \labs \sqrt{k} \lp S_{nj}(x_j) - x_j \rp \dot{\ell}_j(\xi_n) + v_{nj}(x_j) \dot{\ell}_j(x) \rabs = \oh_{\PP}(1).
\end{align*}
Every term in this sum can be dealt with in the exact same way so that we only consider $j=1$.

Fix $0 < \delta < T$. By the triangle inequality and~\eqref{eq:marginal_quantile_processes},
\begin{align*}
	&\sup_{x \in [0,T]^d} \labs \sqrt{k} \lp S_{n1}(x_1) - x_1 \rp \dot{\ell}_1(\xi_n) + v_{n1}(x_1) \dot{\ell}_1(x) \rabs \\
	&\quad \leq \sup_{y \in [0,2T]^d} \labs \dot{\ell}_1(y) \rabs \cdot \sup_{x_1 \in [0,T]} \labs \sqrt{k} \lp S_{n1}(x_1) - x_1 \rp + v_{n1}(x_1) \rabs \\
	&\qquad + \sup_{x_1 \in [0,T]} \labs v_{n1}(x_1) \rabs \cdot \sup_{x \in [\delta,T] \times [0,T]^{d-1}} \labs \dot{\ell}_1(\xi_n) - \dot{\ell}_1(x) \rabs \\
	&\qquad + 2 \sup_{y \in [0,2T]^d} \labs \dot{\ell}_1(y) \rabs \cdot \sup_{x_1 \in [0,\delta]} \labs v_{n1}(x_1) \rabs.
\end{align*}
In each term of the latter upper bound, it is easy to see that the first factor in the product is bounded (in probability) and that the second factor converges to zero (in probability). 
For the first term, we simply use that $0 \leq \dot{\ell}_1 \leq 1$ and we combine~\eqref{eq:marginal_processes} and~\eqref{eq:marginal_quantile_processes}. 
For the second term, we use~\eqref{eq:marginal_processes} which shows that $v_{n1}$ is uniformly bounded in probability on $[0,T]^d$ by Prohorov's theorem and the fact that $\dot{\ell}_1$ is continuous on the compact set $[\delta,T] \times [0,T]^{d-1}$, hence, uniformly continuous on the same set.
For the third term, we again use that $0 \leq \dot{\ell}_1 \leq 1$ and~\eqref{eq:marginal_processes} which ensures that by picking $n$ large enough, $\sup_{x_1 \in [0,\delta]} \labs v_{n1}(x_1) \rabs$ is arbitrary close to $\sup_{x_1 \in [0,\delta]} \labs W_1(x_1) \rabs$ in probability and, by picking $\delta$ small enough, the latter supremum will be arbitrary close to zero.
This completes the proof of the first part of the theorem. 

\bgroup
\color{black}
For the second part, note that
\begin{multline*}
	\sup_{\substack{x_1,x_2 > 0 \\ x_1 + x_2 = 1}} \labs t^{-1} \PP \lc U_1 \leq tx_1, U_2 \leq tx_2 \rc - \ell(x_1,x_2) \rabs \\
	\leq \sup_{\substack{x_1,x_2 > 0 \\ x_1 + x_2 = 1}} \iint_{A_{(x_1,x_2)}} \labs t c(tu_1,tu_2) - \lambda(u_1,u_2) \rabs \d u_1\d u_2,
\end{multline*}
where we recall $A_{(x_1,x_2)} = \{(u_1,u_2) \in [0,\infty)^2: u_1 \leq x_1 \text{ or } u_2 \leq x_2\}$. One observes that the latter supremum can be bounded, independently of $(x_1,x_2)$ by
\[
	\mathcal{D}_{1/t}(t) + \iint_{\substack{u_1 \wedge u_2 \leq 1 \\ u_1 \vee u_2 > 1/t}} \labs t c(tu_1,tu_2) - \lambda(u_1,u_2) \rabs \d u_1\d u_2.
\]
Under Assumption~\ref{ass:bias}, the first term satisfies $\mathcal{D}_{1/t}(t) = \Oh(t^{\alpha})$ as $t \ra \infty$. For the second term, assuming $t < 1$, it equals
\[
	\iint_{\substack{u_1 \wedge u_2 \leq 1 \\ u_1 \vee u_2 > 1/t}} \lambda(u_1,u_2) \d u_1\d u_2,
\]
since the copula density $c \equiv 0$ outside of the unit square $[0,1]^2$. Moving to polar coordinates, one shows that the latter integral is bounded by
\[
	\Phi_p(t) + \lacc \Phi_p \lp \tfrac{\pi}{2} \rp  - \Phi_p \lp \tfrac{\pi}{2} - t \rp \racc = O(t^\alpha),
\]
where the equality also follows from Assumption~\ref{ass:bias}. This shows that condition~1 of Assumption~\ref{ass:stdf_expansion} is verified in this setting. Condition~2 is immediate from Assumption~\ref{ass:bias}. Condition~3 follows from Assumption~\ref{ass:smoothness} since for $x= (x_1,x_2) \in [0,\infty)^2$, using the fact that $\Lambda$ has Lebesgue margins,
\begin{align*}
	\ell(x) = \Lambda(A_x) 
	&= \Lambda \lp [0,x_1] \times [0,\infty) \rp + \Lambda \lp [0,\infty) \times [0,x_2] \rp - \Lambda \lp [0,x_1] \times [0,x_2] \rp \\
	&= x_1 + x_2 - \Lambda \lp [0,x_1] \times [0,x_2] \rp,
\end{align*}
which is continuously differentiable under Assumption~\ref{ass:smoothness}.
\egroup

\section{Proof of Theorem~\ref{thm:asymptoticTS}}
\label{sec:app:asymptoticTS}

We assume that we are working under $H_0$ and that $r_0 \in \param$ denotes the true parameter underlying the data.

First observe that
\[
	\ts = F \lc \sqrt{k}(\tQ_p-Q_{p,\est}) \rc,
\]
where $F : \ell^\infty([0,\pi/2]) \ra (0,\infty)$ is defined for any $z \in \ell^\infty([0,\pi/2])$ by
\[
	F(z) = \int_0^{\pi/2} |z(\theta)| \weight(\theta) \d\theta.
\]
The map $F$ is continuous since if $(z_n)_{n \in \N} \subset \ell^\infty([0,\pi/2])$ is such that $z_n \ra z \in \ell^\infty([0,\pi/2])$,
\[
	|F(z_n) - F(z)| \leq \int_0^{\pi/2} \labs |z_n(\theta)| - |z(\theta)| \rabs \weight(\theta) \d\theta \leq \|z_n-z\|_{\ell^\infty([0,\pi/2])} \int_0^{\pi/2} \weight(\theta) \d\theta \ra 0,
\] 
as $n \ra \infty$ since $\weight$ is assumed to be integrable. Consequently, in view of the continuous mapping theorem, it suffices to determine the asymptotic distribution of the process
\[
	\lacc \sqrt{k} \lp \tQ_p(\theta)-Q_{p,\est}(\theta) \rp : \theta \in [0,\pi/2] \racc,
\]
in $\ell^\infty([0,\pi/2])$ to determine the one of our test statistic.

For any $\theta \in [0,\pi/2]$, we have
\[
	\sqrt{k} \lp \tQ_p(\theta)-Q_{p,\est}(\theta) \rp
	= \sqrt{k} \lp \tQ_p(\theta)-Q_{p,r_0}(\theta) \rp - \sqrt{k} \lp Q_{p,\est}(\theta) -  Q_{p,r_0}(\theta) \rp.
\]
In view of Corollary~\ref{cor:asymptoticsMEL}, the first term in the right hand side satisfies
\[
	\sqrt{k} \lp \tQ_p(\theta)-Q_{p,r_0}(\theta) \rp
	= E_{n,p}^{Q_{p,r_0}}(\theta) + \frac{\int_0^{\pi/2} E_{n,p}^{Q_{p,r_0}}(\psi) f'(\psi) \d\psi }{\sigma_{Q_{p,r_0}}^2(f)} \int_0^\theta f(\psi) \d Q_{p,r_0}(\psi) + \oh_{\PP^*}(1),
\]
as $n \ra \infty$, where the reminder is uniform in $\theta \in [0,\pi/2]$. In view of Assumption~\ref{ass:modelAndEstimator} and Theorem~\ref{thm:expansionSTDF}, the second term satisfies,
\begin{align*}
	&\sqrt{k} \lp Q_{p,\est}(\theta) -  Q_{p,r_0}(\theta) \rp = \sqrt{k} \lp \paramap(\est) - \paramap(r_0) \rp \\
	&\quad =  \int_0^\theta \nabla_r \varrho_{p,r_0}(x) \d x \dotp \sqrt{k} (\est - r_0) + \oh_{\PP}(1) \\
	&\quad =  \int_0^\theta \nabla_r \varrho_{p,r_0}(x) \d x \dotp \int_{[0,1]^2} \sqrt{k} \lp \widehat{\ell}_n(x_1,x_2) - \ell_{r_0}(x_1,x_2) \rp g(x_1,x_2) \, \d\sigma(x_1,x_2) + \oh_{\PP}(1) \\
	&\quad = \int_0^\theta \nabla_r \varrho_{p,r_0}(x) \d x \dotp \int_{[0,1]^2}  \lp v_n(x_1,x_2) - \sum_{j=1}^{2} \dot{\ell}_{r_0,j}(x_1,x_2) v_{nj}(x_j) \rp g(x_1,x_2) \, \d\sigma(x_1,x_2) \\
	&\qquad \qquad + \oh_{\PP}(1),
\end{align*}
where the reminder is uniform in $\theta \in [0,\pi/2]$.

Summarizing, we have shown that, up to asymptotically negligible terms, all the terms appearing in the decomposition of the process $\sqrt{k}(\tQ_p-Q_{p,\est})$ are asymptotically equivalent to continuous functionals of the tail empirical process
\[
	W_{n,k}(C) = \sqrt{k} \lacc \frac{n}{k} P_n \lp \frac{k}{n} C \rp - \frac{n}{k} P \lp \frac{k}{n} C \rp \racc, 
\]
evaluated on sets belonging to a collection $\mathcal{C}$ which is of finite VC dimension. Asymptotics of this process and weighted marginal versions of it are described in Proposition~1 and Proposition~2 of~\cite{lhaut2023asymptotic}, respectively. 
Combining those results with a suitable application of the continuous mapping theorem~\cite[Theorem~1.3.6]{VVV1996} permits to conclude the proof.

\section{Proof of Theorem~\ref{thm:cont_quantiles}}
\label{sec:app:cont_quantiles}

The proof of Theorem~\ref{thm:cont_quantiles} relies on intermediate results that we formulate and prove below.

Consider again the Gaussian process $W_\Lambda$ indexed by Borel sets of $\spacerv$ appearing in Theorem~\ref{thm:asymptoticsEAM}. Given a simple function
\[
	f(x) = \sum_{i=1}^{N} c_i \1_{C_i}(x),
\]
with constants $c_1,\ldots,c_N \in \R$ and Borel sets $C_1, \ldots, C_N$ in $\spacerv$, it will be convenient to introduce the notation
\[
	W_\Lambda(f) = \sum_{i=1}^{N} c_i W_\Lambda(C_i).
\]
It is easy to see that for any such functions $f_1,\ldots,f_k$, the random vector $\big(W_\Lambda(f_1), \ldots, W_\Lambda(f_k) \big)$ remains normally distributed with zero mean and the covariance between any pair equals
\[
	\EE[W_\Lambda(f_i) W_\Lambda(f_j)] = \Lambda(f_i f_j) = \int_{\spacerv} f_i(x)f_j(x) \d\Lambda(x).
\]
Consequently, we may view $W_\Lambda$ as a Gaussian process indexed by simple functions on $\spacerv$.

\bgroup
\color{black}
\begin{lemma}
\label{lem:cont_gaussian_process}

Consider the class of functions
\begin{align*}
	\cF 
	&\egdef \lacc \1_{C_{p,\theta}} : \theta \in [0,\pi/2] \racc \cup \lacc \1_{[0,x] \times [0,\infty]}w(x) : x \in [0,\infty) \racc \\
	&\qquad \cup \lacc \1_{[0,\infty] \times [0,y]}w(y) : y \in [0,\infty) \racc \cup \lacc \1_{A_{(x,y)}} : (x,y) \in [0,1]^2 \racc \\
	&\egdef \cF_1 \cup \cF_2 \cup \cF_3 \cup \cF_4,
\end{align*}
where
\[
	w(x) \egdef
	\begin{cases}
		\frac{1}{x^\eta \vee x^\gamma}, \qquad &\text{si } x > 0 \\
		0,  \qquad &\text{si } x = 0,
	\end{cases}
\]
with $0 \leq \eta < 1/2 < \gamma < 1$.
Let $r_0 \in \param$. If $(r_n)_{n \in \N} \subset \param$ is such that $r_n \rightarrow r_0$ as $n \ra \infty$,  we have in $\ell^\infty(\cF)$ the weak convergence
\[
	\lacc W_{\Lambda_{r_n}}(f) : f \in \cF \racc \wc \lacc W_{\Lambda_{r_0}}(f) : f \in \cF \racc, \qquad n \ra \infty,
\]
provided the maps $\theta \in [0,\pi/2] \mapsto \Phi_{p,r_0}(\theta)$ and $r \in \param \mapsto \Phi_{p,r}(\theta)$ (for fixed $\theta \in [0,\pi/2]$) are continuous, an assumption that is verified under Assumption~\ref{ass:smoothness} and point 1 of Assumption~\ref{ass:modelAndEstimator}.
\end{lemma}
\egroup

\begin{proof}
By Theorem~1.5.4 in~\cite{VVV1996}, the desired weak convergence holds provided
\begin{enumerate}
	\item For any $k \in \N$ and any $f_1,\ldots,f_k \in \cF$,
	\[
		\lp 
			W_{\Lambda_{r_n}}(f_1), \ldots,
			W_{\Lambda_{r_n}}(f_k) 
		\rp 
		\wc
		\lp 
			W_{\Lambda_{r_0}}(f_1), \ldots, 
			W_{\Lambda_{r_0}}(f_k) 
		\rp
	\] 
	as random vectors on $\R^k$.
	\item The sequence
	\[
		\lacc W_{\Lambda_{r_n}}(f) : f \in \cF \racc_{n \in \N}
	\]
	is asymptotically tight. 
\end{enumerate}
Point~1 follows from the fact that each random vector $\lp W_{\Lambda_{r_n}}(f_1), \ldots, W_{\Lambda_{r_n}}(f_k) \rp$ is Gaussian with zero mean and covariance matrix $\Sigma_k(r_n)$ having elements
\[
	\lim_{n \ra \infty} \Sigma_{k,ij}(r_n) = \Sigma_{k,ij}(r_0), \qquad i,j \in \{1,\ldots,k\},
\]
where $\Sigma_k(r_0)$ is the covariance matrix of $\lp W_{\Lambda_{r_0}}(f_1), \ldots, W_{\Lambda_{r_0}}(f_k) \rp$. \textcolor{black}{Indeed, each function $f \in \cF$ is an (possibly weighted) indicator function of a certain set in $\spacerv$ so that element-wise convergence of the covariance matrix $\Sigma_k(r_n)$ to $\Sigma_k(r_0)$ follows from the continuity of the map $r \in \param \mapsto \Phi_{p,r}(\theta)$ for fixed $\theta \in [0,\pi/2]$. This is in particular guaranteed by point 1 of Assumption~\ref{ass:modelAndEstimator}.}


Point~2 will follow from establishing the asymptotic tightness of each sequence
\[
	\lacc W_{\Lambda_{r_n}}(f) : f \in \cF_j \racc_{n \in \N}, \qquad j \in \{1,\ldots,4\}
\]
separately. By Theorem~1.5.7 in~\cite{VVV1996}, it is equivalent to show that for each $j = 1,\ldots,4$, there exists a semi-metric $\rho_j$ on $\cF_j$ such that $(\cF_j,\rho_j)$ is totally bounded and for any $\epsilon, \eta > 0$, there exists $\delta = \delta(\epsilon,\eta) > 0$ and $n_0 = n_0(\delta) \in \N$ such that
\[
	\PP \lc \sup_{f,g \in \cF_j, \rho_j(f,g) < \delta} \labs W_{\Lambda_{r_n}}(f) - W_{\Lambda_{r_n}}(g) \rabs > \epsilon \rc < \eta, \qquad n \geq n_0. 
\]

\paragraph*{Case $\boldsymbol{j=1}$.} \quad
Each function $f \in \cF_1$ is in bijection with a certain $\theta \in [0,\pi/2]$. Hence, we may view the sequence $\lacc W_{\Lambda_{r_n}}(f) : f \in \cF_1 \racc_{n \in \N}$ as a sequence of processes indexed by $T_1 \egdef [0,\pi/2]$. Trivially, each element of $\lacc W_{\Lambda_{r_n}}(f) : f \in \cF_1 \racc_{n \in \N}$ is a Gaussian process and, hence, sub-Gaussian as defined on page~101 of~\cite{VVV1996} with respect to the standard deviation semi-metric
\[
	\rho_{1,n}(\theta_1,\theta_2) \egdef \sqrt{\labs \Phi_{p,r_n}(\theta_1) - \Phi_{p,r_n}(\theta_2) \rabs} \egdef \sqrt{d_{1,n}(\theta_1,\theta_2)}.
\]
By Corollary~2.2.8 and page~98 of~\cite{VVV1996}, it follows that there exists a universal constant $K > 0$ such that for any $\delta >0$ and any $n \in \N$, we have
\begin{equation}
\label{eq:expec_j1}
	\EE \lc \sup_{\rho_{1,n}(\theta_1,\theta_2) < \delta} \labs W_{\Lambda_{r_n}}(\theta_1) - W_{\Lambda_{r_n}}(\theta_2) \rabs \rc 
	\leq K \int_0^\delta \sqrt{\log N(\tfrac{t}{2}, T_1 ,\rho_{1,n})} \d t,
\end{equation}
where $N(t,T_1,\rho_{1,n})$ is the minimal number of $\rho_{1,n}$-balls of radius $t > 0$ needed to cover $T_1$.
First observe that for any $t>0$, we have
\[
	 N(t,T_1,\rho_{1,n}) =  N(t^2,T_1,d_{1,n}).
\]
Furthermore, the map
\[
	\theta \in (T_1, d_{1,n}) \mapsto \Phi_{p,r_n}(\theta) \in ([0,\Phi_{p,r_n}(\pi/2)], e(x,y) \egdef |x-y|)
\]
is an isometry so that we can compute the covering number in the latter space. Since $[0,\Phi_{p,r_n}(\pi/2)] \subseteq I \egdef [0,2]$ for any $n \in \N$, we have 
\[
	N(t,T_1,d_{1,n}) \leq N(t,I,e), \qquad \forall t > 0.
\]
Note that for any $x \in I$,
\[
	B(x,t,e) \egdef \lacc y \in I : e(x,y) \leq t \racc = [x-t,x+t] \cap I,
\]
so that if $t \geq 1$, we have $N(t,I,e) = 1$ and for $0 < t < 1$,
\[
	N(t,I,e) \leq t^{-1}.
\]
We deduce that for $t \geq 1$, we have $N(t,T_1,\rho_{1,n}) = 1$, while for $0 < t < 1$,
\[
	N(t,T_1,\rho_{1,n}) =  N(t^2,T_1,d_{1,n}) \leq N(t^2,I,e) \leq t^{-2}.
\]
Since this holds for any $n \in \N$, note in particular that $N(t,T_1,\rho_{1})$ where
\[
	\rho_1(\theta_1,\theta_2) \egdef \lim_{n \ra \infty} \rho_{1,n}(\theta_1,\theta_2), \qquad (\theta_1,\theta_2) \in [0,\pi/2]^2
\]
satisfies the same bound and the space $(T_1,\rho_{1})$ is totally bounded. Now let $\delta \in (0,2)$. Then for any $t \in [0,\delta]$, we have by previous computations,
\[
	\sqrt{\log N(\tfrac{t}{2}, T_1 ,\rho_{1,n})} \leq \sqrt{\log(4/t^2)} = \sqrt{2 \log(2/t)}.
\]
thus, using a standard bound on the complementary error function $\operatorname{erfc}$,
\begin{align*}
	\int_0^\delta \sqrt{\log N(\tfrac{t}{2}, T_1 ,\rho_{1,n})} \d t
	&\leq \sqrt{2} \lacc 2 (\delta/2) \sqrt{\log(2/\delta)}  + \sqrt{\pi} \operatorname{erfc} \lp \sqrt{\log(2/\delta)} \rp \racc \\
	&\leq \sqrt{2} \delta \lacc \sqrt{\log(2/\delta)}  + 1/(2\sqrt{\log(2/\delta)} ) \racc.
\end{align*}

It follows from~\eqref{eq:expec_j1} and Markov's inequality that for any $\epsilon, \delta > 0$ and $n \in \N$,
\[
	\PP \lc \sup_{\rho_{1,n}(\theta_1,\theta_2) < \delta} \labs W_{\Lambda_{r_n}}(\theta_1) - W_{\Lambda_{r_n}}(\theta_2) \rabs > \epsilon \rc \leq K\sqrt{2} \eps^{-1}   \delta \lacc \sqrt{\log(2/\delta)}  + 1/(2\sqrt{\log(2/\delta)} ) \racc.
\]
Now fix $\epsilon, \eta > 0$. Pick $\delta = \delta(\epsilon,\eta) \in (0,2)$ such that the inequality
\[
	K2^{3/2} \eps^{-1} \delta \lacc \sqrt{\log(1/\delta)}  + 1/(2\sqrt{\log(1/\delta)} ) \racc < \eta
\]
holds true and pick $n_0 = n_0(\delta) \in \N$ such that for $n \geq n_0$,
\[
	\lacc (\theta_1,\theta_2) \in [0,\pi/2]^2 : \rho_{1}(\theta_1,\theta_2) < \delta \racc
	\subset \lacc (\theta_1,\theta_2) \in [0,\pi/2]^2 : \rho_{1,n}(\theta_1,\theta_2) < 2\delta \racc
\]
(such $n_0$ exists by uniform convergence of $\rho_{1,n}$ to $\rho_{1}$ which itself follows from the uniform convergence $\Phi_{p,r_n} \to \Phi_{p,r_0}$ which is a consequence of the pointwise convergence and the continuity of $\Phi_{p,r_0}$) then, for $n \geq n_0$ we have
\begin{multline*}
	\PP \lc \sup_{\rho_{1}(\theta_1,\theta_2) < \delta} \labs W_{\Lambda_{r_n}}(\theta_1) - W_{\Lambda_{r_n}}(\theta_2) \rabs > \epsilon \rc \\
	\leq \PP \lc \sup_{\rho_{1,n}(\theta_1,\theta_2) < 2\delta} \labs W_{\Lambda_{r_n}}(\theta_1) - W_{\Lambda_{r_n}}(\theta_2) \rabs > \epsilon \rc 
	< \eta,
\end{multline*}
and the asymptotic tightness condition is satisfied for $j=1$. 

\paragraph*{Case $\boldsymbol{j=2}$ and $\boldsymbol{j=3}$.} \quad
We only consider $j=2$ since the other situation is perfectly symmetric and can be handled with exactly the same tools. Here, each function is in bijection with a certain $x \in [0,\infty)$ and we will view the sequence $\lacc W_{\Lambda_{r_n}}(f) : f \in \cF_2 \racc_{n \in \N}$ as a sequence of processes indexed by $T_2 \egdef [0,\infty)$. For any $n \in \N$, the standard deviation semi-metric of the Gaussian process $\lacc W_{\Lambda_{r_n}}(x) : x \in [0,\infty) \racc$ equals
\[
	\rho_{2}(x_1,x_2) \egdef \sqrt{w(x_1)^2x_1 + w(x_2)^2x_2 - 2 w(x_1)w(x_2)\min(x_1,x_2)} \egdef \sqrt{d_2(x_1,x_2)}.
\]
Consequently, there exists a universal constant $K > 0$ such that for any $n \in \N$ and any $\delta > 0$,
\begin{equation}
\label{eq:expect_j2}
	\EE \lc \sup_{\rho_{2}(x_1,x_2) < \delta} \labs W_{\Lambda_{r_n}}(x_1) - W_{\Lambda_{r_n}}(x_2) \rabs \rc 
	\leq K \int_0^\delta \sqrt{\log N(\tfrac{t}{2},T_2,\rho_{2})} \d t,
\end{equation}
with $N(t,T_2,\rho_{2})$ the minimal number of $\rho_2$-balls of radius $t > 0$ needed to cover $T_2$.


Observe that for any $t>0$,
\begin{equation}
\label{eq:covering_2}
	N(t,T_2,\rho_{2}) = N(t^2,T_2,d_2) \leq N(t^2,T_{2,1},d_2) + N(t^2,T_{2,2},d_2),
\end{equation}
where $T_{2,1} \egdef [0,1]$ and $T_{2,2} \egdef [1,\infty]$. We start by dealing with the quantity $N(t,T_{2,1},d_2)$. Note that for any $x_1,x_2 \in T_{2,1}$, we have
\begin{equation}
\label{eq:d2_01}
	d_2(x_1,x_2) = x_1^{1-2\eta} + x_2^{1-2\eta} - 2(x_1 \wedge x_2) (x_1x_2)^{-\eta}.
\end{equation}
In particular, for any $x \in T_{2,1}$,
\[
	d_2(0,x) = x^{1-2\eta}.
\]
This shows that $N(t,T_{2,1},d_2) = 1$ for $t \geq 1$. Allowing $d_2$ to be defined on $[0,\infty)$ in the exact same way as in~\eqref{eq:d2_01}, we observe that for any $t > 0$ and $n \in \N$,
\[
	d_2(nt, (n+1)t) = t^{1-2\eta} d_2(n,n+1) = d_2(0,t) d_2(n,n+1).
\]
Our aim is to show that $d_2(n,n+1) \leq 1$ for any $n \in \N$. To this end, take any $n \in \N$ and note that
\[
	d_2(n,n+1) \leq 1 \iff g(\eta) \egdef d_2(n,n+1) = n^{1-2\eta} + (n+1)^{1-2\eta} - 2n^{1-\eta}(n+1)^{-\eta} \leq g(0) = 1,
\]
so that it suffices to show that the map $\eta \in (0,1/2) \mapsto g(\eta)$ is decreasing. This follows from writing
\[
	g(\eta) = \lp (n+1)^{\tfrac{1}{2} - \eta} - n^{\tfrac{1}{2}-\eta} \rp^2 + 2 \sqrt{n} [n(n+1)]^{-\eta} \lc \sqrt{n+1} - \sqrt{n} \rc,
\]
which shows that $g(\eta)$ is a sum of two terms decreasing in $\eta \in (0,1/2)$. This is obvious for the second term while for the first it follows from the fact that the application $\alpha \in (0,1/2) \mapsto (n+1)^\alpha - n^\alpha$ is increasing as shown by the direct computation
\[
	\dfrac{\d}{\d \alpha} \lacc (n+1)^\alpha - n^\alpha \racc = \log(n+1) (n+1)^\alpha - \log(n) n^{\alpha} \geq 0.
\]
For any $0 < t < 1$, the previous arguments guarantee that $d_2(t^{\tfrac{1}{1-2\eta}}, 2t^{\tfrac{1}{1-2\eta}}) \leq d_2(0, t^{\tfrac{1}{1-2\eta}}) = t$, and, similarly, using the notation $B(x,t,d_2) \egdef \{y \in T_{2,1} : d_2(y,x) \leq t\}$, we see that
\begin{align*}
	B(t^{\tfrac{1}{1-2\eta}}, t, d_2) &\subseteq [0,2t^{\tfrac{1}{1-2\eta}}] \\
	B(3t^{\tfrac{1}{1-2\eta}}, t, d_2) &\subseteq [2t^{\tfrac{1}{1-2\eta}}, 4t^{\tfrac{1}{1-2\eta}}] \\
	&\hspace{2mm} \vdots
\end{align*}
hence showing that
\[
	T_{2,1} \subseteq \bigcup_{n=}^{N(t)} B((n+1)t^{\tfrac{1}{1-2\eta}}, t, d_2),
\]
where
\[
	N(t,T_{2,1},d_2) \leq N(t) \leq \frac{1}{2} t^{- \tfrac{1}{1-2\eta}}.
\]
Regarding $N(t,T_{2,2},d_2)$, simply observe that the map
\[
	x \in (T_2,d_2) \mapsto x^{-1} \in (T_2,d_2)
\]
is an isometry for $\eta = 1 - \gamma$. We deduce that $N(t,T_{2,2},d_2) = 1$ for any $t \geq 1$ and
\[
	N(t,T_{2,2},d_2) \leq \frac{1}{2} t^{- \tfrac{1}{2\gamma-1}}, \qquad 0 < t < 1,
\]
so that in view of equation~\eqref{eq:covering_2}, for any $t \in (0,1)$,
\[
	N(t,T_2,\rho_{2}) \leq \frac{1}{2} \lp t^{- \tfrac{2}{1-2\eta}} + t^{-\tfrac{2}{2\gamma-1}} \rp.
\]
In particular, the space $(T_2,\rho_{2})$ is totally bounded.
Pick $\delta \in (0,2)$. Then by previous computations, we have that for any $t \in [0,\delta]$,
\[
	\sqrt{\log N(\tfrac{t}{2},T_2,\rho_{2})} \leq \sqrt{\frac{2}{1-2\alpha}} \sqrt{\log(2/t)},
\]
where $\alpha \egdef \min\{\eta,1-\gamma\} \in [0,1/2)$. Computations similar to the case $j=1$ then leads to
\begin{align*}
	\int_0^\delta \sqrt{\log N(\tfrac{t}{2}, T_2 ,\rho_{2})} \d t
	&\leq \sqrt{\frac{2}{1-2\alpha}} \lacc 2 (\delta/2) \sqrt{\log(2/\delta)}  + \sqrt{\pi} \operatorname{erfc} \lp \sqrt{\log(2/\delta)} \rp \racc \\
	&\leq \sqrt{\frac{2}{1-2\alpha}} \delta \lacc \sqrt{\log(2/\delta)}  + 1/(2\sqrt{\log(2/\delta)} ) \racc.
\end{align*}

Hence, equation~\eqref{eq:expect_j2} and Markov's inequality imply that for any $\epsilon,\delta > 0$,
\[
	\PP \lc \sup_{\rho_{2}(x_1,x_2) < \delta} \labs W_{\Lambda_{r_n}}(x_1) - W_{\Lambda_{r_n}}(x_2) \rabs > \eps \rc
	\leq \sqrt{\frac{2}{1-2\alpha}} K \eps^{-1} \delta \lacc \sqrt{\log(2/\delta)}  + 1/(2\sqrt{\log(2/\delta)} ) \racc.
\]
Fixing $\epsilon, \eta > 0$ and picking $\delta = \delta(\epsilon,\eta) \in (0,2)$ sufficiently small such that
\[
	\sqrt{\frac{2}{1-2\alpha}} K \eps^{-1} \delta \lacc \sqrt{\log(2/\delta)}  + 1/(2\sqrt{\log(2/\delta)} ) \racc < \eta,
\]
show the asymptotic tightness condition for $j=2$.

\paragraph*{Case $\boldsymbol{j=4}$.} \quad
As for the previous cases, each function $f \in \cF_4$ is in a one-to-one relation with a certain $(x,y) \in [0,1]^2$ and we will view the sequence $\lacc W_{\Lambda_{r_n}}(f) : f \in \cF_4 \racc_{n \in \N}$ as a sequence of processes indexed by $T_4 \egdef [0,1]^2$. Observe that for any $n \in \N$, the standard deviation semi-metric associated with the Gaussian process $\lacc W_{\Lambda_{r_n}}(x,y) : (x,y) \in [0,1]^2 \racc$ satisfies
\begin{align*}
	\rho_{4,n} \lp (x_1,y_1), (x_2,y_2) \rp 
	&\egdef \sqrt{\Lambda_{r_n} \lp A_{(x_1,y_1)} \triangle  A_{(x_2,y_2)} \rp} \\
	&\leq \sqrt{|x_1-x_2| + |y_1 - y_2|}  \egdef \rho_{4} \lp (x_1,y_1), (x_2,y_2) \rp.
\end{align*}
Therefore, the process $\lacc W_{\Lambda_{r_n}}(x,y) : (x,y) \in [0,1]^2 \racc$ is sub-Gaussian with respect to $\rho_{4}$ for any $n \in \N$ and satisfies for any $\delta > 0$,
\begin{equation}
\label{eq:expec_j4}
	\EE \lc \sup_{\rho_{4} \lp (x_1,y_1), (x_2,y_2) \rp < \delta} \labs W_{\Lambda_{r_n}}(\theta_1) - W_{\Lambda_{r_n}}(\theta_2) \rabs \rc 
	\leq K \int_0^\delta \sqrt{\log N(\tfrac{t}{2},T_4,\rho_{4})} \d t,
\end{equation}
for some universal constant $K > 0$, where $N(t,T_4,\rho_{4})$ is the minimal number of $\rho_{4}$-balls of radius $t>0$ needed to cover $T_4$. Defining $d_4 \egdef \rho_{4}^2$, we see that for any $t>0$,
\[
	N(t,T_4,\rho_{4}) = N(t^2,T_4,d_4).
\]
Define for any $t>0$ and $(x,y) \in T_4$,
\[
	B((x,y),t,d_4) \egdef \lacc (u,v) \in T_4 : |x-u| + |y-v| \leq t \racc.
\]
Pythagoras' theorem shows that for $t \geq 1$,
\[
	T_4 \subseteq B((0.5,0.5),t,d_4),
\]
so that $N(t,T_4,d_4) = 1$ for such values of $t$. For $0<t<1$, note that
\[
	B((x,y),t,d_4) \supseteq \lacc (u,v) \in T_4 : \max\{|u-x|,|v-y|\} \leq t/2 \racc.
\]
The latter set corresponds to a square centered at $(x,y)$ aligned with the axes with area $t^2$. We need can cover $T_4$ with $t^{-2}$ such squares so that for $0<t<1$,
\[
	N(t,T_4,\rho_{4}) = N(t^2,T_4,d_4) \leq t^{-4}.
\]
In particular that the space $(T_4,\rho_{4})$ is totally bounded.
Pick $\delta \in (0,2)$. Then by previous computations, we have that for any $t \in [0,\delta]$,
\[
	\sqrt{\log N(\tfrac{t}{2},T_4,\rho_{4})} \leq 2 \sqrt{\log(2/t)}.
\]
Computations similar to the previous cases leads to
\begin{align*}
	\int_0^\delta \sqrt{\log N(\tfrac{t}{2}, T_4 ,\rho_{4})} \d t
	&\leq 2 \lacc 2 (\delta/2) \sqrt{\log(2/\delta)}  + \sqrt{\pi} \operatorname{erfc} \lp \sqrt{\log(2/\delta)} \rp \racc \\
	&\leq 2 \delta \lacc \sqrt{\log(2/\delta)}  + 1/(2\sqrt{\log(2/\delta)} ) \racc.
\end{align*}
Equation~\eqref{eq:expec_j4} and Markov's inequality imply that for any $\eps, \delta > 0$,
\begin{multline*}
	\PP \lc \sup_{\rho_4((x_1,y_1), (x_2,y_2)) < \delta} \labs W_{\Lambda_{r_n}}(x_1,y_1) -  W_{\Lambda_{r_n}}(x_1,y_1) \rabs > \eps \rc \\
	\leq 2K\eps^{-1} \delta \lacc \sqrt{\log(2/\delta)}  + 1/(2\sqrt{\log(2/\delta)} ) \racc.
\end{multline*}
Given $\epsilon, \eta > 0$, by picking $\delta = \delta(\epsilon,\eta) > 0$ small enough such that
\[
	2K\eps^{-1} \delta \lacc \sqrt{\log(2/\delta)}  + 1/(2\sqrt{\log(2/\delta)} ) \racc < \eta,
\]
we verify the asymptotic tightness condition for this case.
\end{proof}

For any $r \in \param$, let $\P_r$ denote the law of $L_r$ on $\R$.
\begin{lemma}
\label{lem:cont_limit_law}
\bgroup
\color{black}
Under the same assumptions as Lemma~\ref{lem:cont_gaussian_process} and Assumption~\ref{ass:density_reg}, if $(r_n)_{n \in \N} \subset \param$ is such that $r_n \rightarrow r_0$ as $n \ra \infty$ for an interior point $r_0 \in \param$, we have $\P_{r_n} \wc \P_{r_0}$ on $\R$.
\egroup
\end{lemma}

\begin{proof}
Let $r_0 \in \param$ be an interior point.
Recall that
\[
	L_{r_0} = \int_0^{\pi/2} \labs X_{r_0}(\theta) \rabs \weight(\theta) \d\theta,
\]
where for $\theta \in [0,\pi/2]$,
\[
	X_{r_0}(\theta) \egdef \gamma_{p,r_0}(\theta) - \int_0^\theta \nabla_r \varrho_{p,r_0}(x) \d x \dotp I_{g,\sigma,r_0}.
\]
The idea is to write
\begin{equation}
\label{eq:T_r}
	X_{r_0} = T_{r_0}(W_{\Lambda_{r_0}}),
\end{equation}
for some linear operator $T_{r_0} : \ell^\infty(\cF) \mapsto \ell^\infty([0,\pi/2])$ satisfying the property that if $(r_n)_{n \in \N} \subset \param$ is such that $r_n \rightarrow r_0$ in $\R^m$ and $(z_n)_{n \in \N} \subset \ell^\infty(\cF)$ is such that $z_n \ra z$ in $\ell^\infty(\cF)$, then also
\begin{equation}
\label{eq:continuous_convergence}
	T_{r_n}(z_n) \ra T_{r_0}(z) \qquad \text{in }  \ell^\infty([0,\pi/2]).
\end{equation}
The result would then follow from Lemma~\ref{lem:cont_gaussian_process} combined with the extended continuous mapping theorem~\cite[Theorem~1.11.1]{VVV1996}.

A sufficient condition to ensure~\eqref{eq:continuous_convergence} is to show that the operator $T_{r_0}$ is bounded, i.e.,
\begin{equation*}
	\|T_{r_0}\| \egdef \sup \lacc \|T_{r_0}(z)\|_{\ell^\infty([0,\pi/2])} : \|z\|_{\ell^\infty(\cF)} \leq 1 \racc < \infty
\end{equation*}
together with
\begin{equation}
\label{eq:operator_convergence}
	\lim_{n \ra \infty} \|T_{r_n} - T_{r_0} \| = 0.
\end{equation}

In practice, if we are able to decompose $T_{r_0}$ as $T_r = \sum_{i=1}^{k} T_{r_0,i}$ for bounded linear operators $T_{r_0,i} : \ell^\infty(\cF) \mapsto \ell^\infty([0,\pi/2])$ satisfying~\eqref{eq:operator_convergence}, then we are done in view of the triangle inequality. The same reasoning also applies to composition of such operators: if $T_n,T: E \ra G$ are bounded linear operators between two normed spaces such that
\[
	T_n = V_n \circ U_n \qquad \text{and} \qquad
	T = V \circ U,
\] 
for some bounded linear operators $U_n,U : E \ra F$ and $V_n,V : F \ra G$ between normed spaces satisfying
\[
	\lim_{n \ra \infty} \| U_n - U \|_{E,F} = 0 \qquad \text{and} \qquad
	\lim_{n \ra \infty} \| V_n - V \|_{F,G} = 0,
\]
then also
\[
	\lim_{n \ra \infty} \|T_n - T \|_{E,G} = 0.
\]
Indeed, for any $e \in E$, we have
\begin{align*}
	\|T_n(e) - T(e)\|_G
	&= \|V_n(U_n(e)) - V(U(e)) \|_G \\
	&= \|V_n(U_n(e)) - V(U_n(e)) + V(U_n(e)) - V(U(e)) \|_G \\
	&\leq \|V_n - V\|_{F,G} \|U_n\|_{E,F} \|e\|_E + \|V\|_{F,G} \|U_n - U\|_{E,F} \|e\|_E.
\end{align*}
Since $\limsup_{n \ra \infty} \|U_n\|_{E,F} < \infty$, we can conclude.

\paragraph*{Process $\boldsymbol{\alpha_{p,r_0}}$.}
Let us first consider the linear operator $T_{\alpha,r_0} : \ell^\infty(\cF) \mapsto \ell^\infty([0,\pi/2])$ which is such that
\begin{equation}
\label{eq:alpha_op}
	T_{\alpha,r_0}(W_{\Lambda_{r_0}}) = \alpha_{p,r_0}.
\end{equation}
Our aim is to show that $T_{\alpha,r_0}$ is bounded and satisfies~\eqref{eq:operator_convergence}. We will assume $p < \infty$ since this corresponds to the most involved case. Note that $T_{\alpha,r_0}$ can be trivially decomposed as
\[
	T_{\alpha,r_0} = \sum_{i=1}^3 T_{\alpha,r_0,i},
\]
where for $z \in \ell^\infty(\cF)$ and $\theta \in [0,\pi/2]$ we defined
\begin{equation}
\label{eq:alpha_op_1}
	\lp T_{\alpha,r_0,1}(z) \rp (\theta) \egdef z(\1_{C_{p,\theta}}),
\end{equation}
and
\begin{multline}
\label{eq:alpha_op_2}
	\lp T_{\alpha,r_0,2}(z) \rp (\theta)
	\egdef \lambda_{r_0}(1,\tan\theta) \int_0^{x_p(\theta)} x^{-1} \Bigg[ z \lp \1_{[0,x] \times [0,\infty]} w(x) \rp w(x)^{-1} \tan\theta \\ - z \lp \1_{[0,\infty] \times [0,x\tan\theta]} w(x\tan\theta) \rp w(x\tan\theta)^{-1} \Bigg] \d x,
\end{multline}
and
\begin{multline}
\label{eq:alpha_op_3}
	\lp T_{\alpha,r_0,3}(z) \rp (\theta)
	\egdef \int_{x_p(\theta)}^{\infty} x^{-1} \lambda_{r_0}(1,(x^p-1)^{-1/p}) \Bigg[ z \lp \1_{[0,x] \times [0,\infty]} w(x) \rp w(x)^{-1} y_p'(x) \\
	- z \lp \1_{[0,\infty] \times [0,y_p(x)]} w(y_p(x)) \rp w(y_p(x))^{-1} \Bigg] \d x.
\end{multline}
In view of our previous comments, we can analyze each of those operators separately.

Let $z \in \ell^\infty(\cF)$ be such that $\|z\|_{\ell^\infty(\cF)} \leq 1$. Then we immediately see that
\[
	\|T_{\alpha,r_0,1}(z)\|_{\ell^\infty([0,\pi/2])} \leq 1.
\]
Therefore, $T_{\alpha,r_0,1}$ is bounded. Furthermore, if $(r_n)_{n \in \N} \subset \param$ is a sequence such that $r_n \ra r_0$ in $\R^m$ as $n \ra \infty$, condition~\eqref{eq:operator_convergence} is trivially satisfied since
\[
	\| T_{\alpha,r_n,1} - T_{\alpha,r_0,1} \| = 0 \qquad \text{for any } n \in \N.
\]

Let us now turn to the second operator. Let $z \in \ell^\infty(\cF)$ be such that $\|z\|_{\ell^\infty(\cF)} \leq 1$. Then we see that
\[
	\|T_{\alpha,r_0,2}(z)\|_{\ell^\infty([0,\pi/2])}
	\leq \sup_{\theta \in [0,\pi/2]} \lacc \lambda_{r_0}(1,\tan\theta) \int_0^{x_p(\theta)} x^{-1} \lc w(x)^{-1} \tan\theta + w(x\tan\theta)^{-1} \rc \d x \racc,
\]
so that if we show the latter supremum is finite, the operator will be bounded.
First observe that by~\eqref{eq:densities_relation}, we have for any $\theta \in [0,\pi/2]$,
\[
	\lambda_{r_0}(1,\tan\theta) \leq \min \lacc \tan\theta, (\tan\theta)^{-2} \racc \varphi_{p,r_0}(\theta).
\]
Furthermore, one may compute that for any $\theta \in [0,\pi/2]$,
\[
	\int_0^{x_p(\theta)} x^{-1} w(x)^{-1} \tan\theta \d x
	= \tan\theta \lacc \eta^{-1} + \gamma^{-1} \lp x_p(\theta)^\gamma -1 \rp \racc,
\]
and
\[
	\int_0^{x_p(\theta)} x^{-1}  w(x\tan\theta)^{-1} \d x
	= \eta^{-1} + (\tan\theta)^\gamma(x_p(\theta)^\gamma - (\tan\theta)^{-\gamma})/\gamma.
\]
One may show that for any $\theta \in [0,\pi/4]$,
\[
	\int_0^{x_p(\theta)} x^{-1} \lc w(x)^{-1} \tan\theta + w(x\tan\theta)^{-1} \rc \d x
	\leq 2 \eta^{-1} + \frac{2^{\gamma/p}}{\gamma},
\]
while for $\theta \in [\pi/4,\pi/2]$,
\[
	\frac{1}{\tan\theta} \int_0^{x_p(\theta)} x^{-1} \lc w(x)^{-1} \tan\theta + w(x\tan\theta)^{-1} \rc \d x
	\leq 2 \eta^{-1} + \frac{2^{\gamma/p}}{\gamma}.
\]
Consequently, observing that $(\tan\theta)^{-1} = \tan(\tfrac{\pi}{2} - \theta)$, we find the bound
\begin{multline*}
	\sup_{\theta \in [0,\pi/2]} \lacc \lambda_{r_0}(1,\tan\theta) \int_0^{x_p(\theta)} x^{-1} \lc w(x)^{-1} \tan\theta + w(x\tan\theta)^{-1} \rc \d x \racc \\
	\leq \lp 2 \eta^{-1} + \frac{2^{\gamma/p}}{\gamma} \rp \sup_{\theta \in [0,\pi/2]} \lacc \min \lacc \tan\theta, \tan(\tfrac{\pi}{2} - \theta) \racc \varphi_{p,r_0}(\theta) \racc.
\end{multline*}
The latter quantity is bounded in view of point 1 in Assumption~\ref{ass:density_reg}.
We deduce that the operator $T_{\alpha,r_0,2}$ is bounded. 
Now, let $(r_n)_{n \in \N} \subset \param$ be a sequence such that $r_n \ra r_0$ in $\R^m$ as $n \ra \infty$. Similar computations as before show that for any $z \in \ell^\infty(\cF)$ with $\|z\|_{\ell^\infty(\cF)} \leq 1$ and any $\theta \in [0,\pi/2]$, we have
\begin{align*}
	&\labs \lp T_{\alpha,r_n,2}(z) \rp (\theta) - \lp T_{\alpha,r_0,2}(z) \rp (\theta) \rabs \\
	&\qquad \leq \min \{\tan\theta,\tan(\tfrac{\pi}{2} - \theta) \} \labs  \varphi_{p,r_n}(\theta) -  \varphi_{p,r_0}(\theta) \rabs \lp 2\eta^{-1} + \frac{2^{\gamma/p}}{\gamma} \rp.
\end{align*}
The latter quantity converges to zero in view of point 2 in Assumption~\ref{ass:density_reg} so that condition~\eqref{eq:operator_convergence} is verified for $T_{\alpha,r_0,2}$.

Let us finally turn to the third operator $T_{\alpha,r_0,3}$ defined by~\eqref{eq:alpha_op_3}. We start by showing that it is bounded. Let $z \in \ell^\infty(\cF)$ be such that $\|z\|_{\ell^\infty(\cF)} \leq 1$. One computes that
\[
	\| T_{\alpha,r_0,3}(z) \|_{\ell^\infty([0,\pi/2])}
	\leq \int_1^\infty x^{-1} \lambda_{r_0}(1,(x^p-1)^{-1/p}) \lc x^\gamma |y_p'(x)| + y_p(x)^\gamma \rc \d x,
\]
so that we just need to show that the latter integral is finite. Formula~\eqref{eq:densities_relation} permits to show that
\[
	\lambda_{r_0}(1,(x^p-1)^{-1/p}) = \frac{1}{x} \cdot \frac{1}{1 + (x^p-1)^{-2/p}} \cdot \varphi_{p,r_0} \lp \arctan((x^p-1)^{-1/p}) \rp.
\]
Note also that $|y_p'(x)| = (x^p-1)^{-(1+1/p)}$ and recall that $y_p(x) = x/(x^p-1)^{1/p}$.
Considering the change of variable $u = \arctan((x^p-1)^{-1/p})$ which has inverse $x = x_p(u) = \|(1,\cot u)\|_p$ and differential $\d x = - x_p(u)^{1-p} (\tan u)^{-(1+p)} (\cos u)^{-2} \d u$ and noting that $(x^p-1)^{-1/p} = \tan u$, tedious computations show that the latter integral equals
\[
	\int_0^{\pi/2} \varphi_{p,r_0}(u) \lp \|(1,\cot u)\|_p^{\gamma-(1+p)} + \|(1,\tan u)\|_p^{\gamma-(1+p)} \rp \d u.
\]
The term in parenthesis is always smaller than $1$ so the integrand is bounded by $\varphi_{p,r_0}$, which is integrable on $[0,\pi/2]$ and the operator $T_{\alpha,r_0,3}$ is bounded.
Now, let $(r_n)_{n \in \N} \subset \param$ be a sequence such that $r_n \ra r_0$ in $\R^m$ as $n \ra \infty$.
Similar computations as before show that for any $z \in \ell^\infty(\cF)$ with $\|z\|_{\ell^\infty(\cF)} \leq 1$,
\begin{multline*}
	\| T_{\alpha,r_n,3}(z) - T_{\alpha,r_0,3}(z) \|_{\ell^\infty([0,\pi/2])} \\
	\leq \int_0^{\pi/2} \labs \varphi_{p,r_n}(u) - \varphi_{p,r_0}(u) \rabs \lp \|(1,\cot u)\|_p^{\gamma-(1+p)} + \|(1,\tan u)\|_p^{\gamma-(1+p)} \rp \d u
\end{multline*}
By point 2 of Assumption~\ref{ass:density_reg} and integrability of the map $\theta \in [0,\pi/2] \mapsto (\eta(\theta))^{-\nu}$ for any $\nu \in (0,1)$, the latter quantity converges to zero, showing that condition~\eqref{eq:operator_convergence} is verified for $T_{\alpha,r_0,3}$.

\paragraph*{Process $\boldsymbol{\beta_{p,r_0}}$.}
Let us now consider the linear operator $T_{\beta,r_0} : \ell^\infty([0,\pi/2]) \mapsto \ell^\infty([0,\pi/2])$ which is such that
\begin{equation}
\label{eq:beta_op}
	T_{\beta,r_0}(\alpha_{p,r_0}) = \beta_{p,r_0}.
\end{equation}
Our aim is to show that $T_{\beta,r_0}$ is bounded and satisfies~\eqref{eq:operator_convergence}.

Let $z \in \ell^\infty([0,\pi/2])$. The operator $T_{\beta,r_0}$ is defined for any $\theta \in [0,\pi/2]$ by
\begin{equation}
\label{eq:beta_op_point}
	\lp T_{\beta,r_0}(z) \rp (\theta)
	\egdef \frac{z(\theta) \Phi_{p,r_0}(\tfrac{\pi}{2}) - \Phi_{p,r_0}(\theta) z(\tfrac{\pi}{2})}{\Phi_{p,r_0}(\tfrac{\pi}{2})^2}.
\end{equation}
It is clear that if $\|z\|_{\ell^\infty([0,\pi/2])} \leq 1$, since the angular distribution function satisfies $1 \leq \Phi_{p,r_0}(\pi/2) \leq 2$ for any $r \in \param$,
\[
	\|T_{\beta,r_0}(z)\|_{\ell^\infty([0,\pi/2])} \leq 2,
\]
so that $T_{\beta,r_0}$ is bounded.
Now, let $(r_n)_{n \in \N} \subset \param$ be a sequence such that $r_n \ra r_0$ in $\R^m$ as $n \ra \infty$.
Similar computations show that if $\|z\|_{\ell^\infty([0,\pi/2])} \leq 1$,
\[
	\|T_{\beta,r_n}(z) - T_{\beta,r_0}(z) \|_{\ell^\infty([0,\pi/2])}
	\leq 6 \|\Phi_{p,r_n} - \Phi_{p,r_0}\|_{\ell^\infty([0,\pi/2])}.
\]
The latter quantity converges to zero as $n \ra \infty$ since $\Phi_{p,r_n} \ra \Phi_{p,r_0}$ in $\ell^\infty([0,\pi/2])$ if $r_n \ra r$ in $\R^m$, as explained in the proof of Lemma~\ref{lem:cont_gaussian_process}.
We deduce that $T_{\beta,r_0}$ satisfies~\eqref{eq:operator_convergence}.

\paragraph*{Process $\boldsymbol{\gamma_{p,r_0}}$.}
\bgroup
\color{black}
Let us finally consider the linear operator $T_{\gamma,r_0} : \ell^\infty([0,\pi/2]) \mapsto \ell^\infty([0,\pi/2])$ which is such that
\begin{equation}
\label{eq:gamma_op}
	T_{\gamma,r_0}(\beta_{p,r_0}) = \gamma_{p,r_0}.
\end{equation}
Our aim is to show that $T_{\gamma,r_0}$ is bounded and satisfies~\eqref{eq:operator_convergence}.

Let $z \in \ell^\infty([0,\pi/2])$. The operator $T_{\gamma,r_0}$ is defined for any $\theta \in [0,\pi/2]$ by
\begin{equation}
\label{eq:gamma_op_point}
	\lp T_{\gamma,r_0}(z) \rp (\theta)
	\egdef z(\theta) + \frac{\int_0^{\pi/2} z(\psi) f'(\psi) \d\psi }{\sigma_{Q_{p,r_0}}^2(f)} \int_0^\theta f(\psi) \d Q_{p,r_0}(\psi),
\end{equation}
where we recall that the functions $f$ and $f'$ are defined by
\begin{align*}
	f(\psi) &= \frac{\sin\psi - \cos\psi}{\|(\sin\psi, \cos\psi)\|_p} \qquad \text{and} \\
	f'(\psi) &= \frac{(\sin\psi)^{p-1} + (\cos\psi)^{p-1}}{\|(\sin\psi, \cos\psi)\|_p^{1+p}}
\end{align*}
satisfy the bounds $|f| \leq 1$ and $0 \leq f' \leq 2$.
It readily follows that if $\|z\|_{\ell^\infty([0,\pi/2])} \leq 1$, then
\[
	\|T_{\gamma,r_0}(z)\|_{\ell^\infty([0,\pi/2])} \leq 1 + \frac{\pi}{\sigma_{Q_{p,r_0}}^2(f)} < \infty,
\]
so $T_{\gamma,r_0}$ is a bounded operator. 
Now, let $(r_n)_{n \in \N} \subset \param$ be a sequence such that $r_n \ra r_0$ in $\R^m$ as $n \ra \infty$.
For $\|z\|_{\ell^\infty([0,\pi/2])} \leq 1$, we compute
\begin{align*}
\lefteqn{\|T_{\gamma,r_n}(z) - T_{\gamma,r_0}(z) \|_{\ell^\infty([0,\pi/2])}} \\
	&\leq \pi \left\| \frac{\int_0^\theta f(\psi) \d Q_{p,r_n}(\psi)}{\sigma_{Q_{p,r_n}}^2(f)} - \frac{\int_0^\theta f(\psi) \d Q_{p,r_0}(\psi)}{\sigma_{Q_{p,r_0}}^2(f)} \right\|_{\ell^\infty([0,\pi/2])}.
\end{align*}
Fubini's theorem leads to
\[
	\left\| \int_0^\theta f(\psi) \d Q_{p,r_n}(\psi) - \int_0^\theta f(\psi) \d Q_{p,r_0}(\psi) \right\|_{\ell^\infty([0,\pi/2])}
	\leq (1+\pi) \|Q_{p,r_n} - Q_{p,r_0} \|_{\ell^\infty([0,\pi/2])},
\]
and the latter quantity converges to zero as $n \ra 0$ since $Q_{p,r_n} \ra Q_{p,r_0}$ in $\ell^\infty([0,\pi/2])$. The same argument implies
\[
	\sigma_{Q_{p,r_n}}^2(f) \to \sigma_{Q_{p,r_0}}^2(f) > 0, \qquad \text{as } n \to \infty,
\]
and this suffices to verify~\eqref{eq:operator_convergence} for the operator $T_{\gamma,r_0}$.
\egroup 

\paragraph*{Process associated with the derivative.}
The last operator to be considered is the linear operator $T_{\nabla,r_0} : \ell^\infty(\cF) \ra \ell^\infty([0,\pi/2])$ which is such that for $\theta \in [0,\pi/2]$,
\begin{equation}
\label{eq:nabla_op}
	\lp T_{\nabla,r_0} \lp W_{\Lambda,r_0} \rp \rp (\theta) = \int_0^\theta \nabla_r \varrho_{p,r_0}(x) \d x \dotp I_{g,\sigma,r_0},
\end{equation}
where we recall that
\[
	I_{g,\sigma,r_0} = \int_{[0,1]^2} \Big( W_{\Lambda_{r_0}}(A_{(x_1,x_2)}) - \sum_{j=1}^2 \dot{\ell}_{r,j}(x_1,x_2) W_j(x_j) \Big) g(x_1,x_2) \d\sigma(x_1,x_2),
\]
for some finite Borel measure $\sigma$ on $[0,1]^2$ and a function $g : [0,1]^2 \ra \R^m$ in $\lp L_2([0,1]^2,\sigma) \rp^m$. The latter is naturally decomposed as $T_{\nabla,r_0} = \sum_{k=1}^{m} T_{\nabla,k,r_0}$ where the operators $T_{\nabla,k,r_0} : \ell^\infty(\cF) \ra \ell^\infty([0,\pi/2])$ are defined by
\[
	\lp T_{\nabla,K,r_0} \lp W_{\Lambda,r_0} \rp \rp (\theta) \egdef \lp I_{g,\sigma,r_0} \rp_k \int_0^\theta \partial_{r_k} \varrho_{p,r_0}(x) \d x, \qquad k \in \{1,\ldots,m\}.
\]
Fix $k \in \{1,\ldots,m\}$ and let $z \in \ell^\infty(\cF)$ be such that $\|z\|_{\ell^\infty(\cF)} \leq 1$. Using the fact that $\dot{\ell}_{r_0,j} \leq 1$ for any $j \in \{1,2\}$, we readily see that
\[
	\|  T_{\nabla,K,r_0}(z) \|_{\ell^\infty([0,\pi/2])}
	\leq 3 \int_{[0,1]^2} g_k(x_1,x_2) \d\sigma(x_1,x_2) \int_0^{\pi/2} \labs \partial_{r_k} \varrho_{p,r_0}(x) \rabs \d x < \infty,
\]
so that $T_{\nabla,K,r_0}$ is bounded since the partial derivative $\partial_{r_k} \varrho_{p,r_0}$ is assumed to be integrable on $[0,\pi/2]$ by point 1 of Assumption~\ref{ass:modelAndEstimator}. Let us now consider a sequence $(r_n)_{n \in \N} \subset \param$ which is such that $r_n \ra r_0$ in $\R^m$ as $n \ra \infty$. We compute
\begin{align*}
\lefteqn{\|T_{\nabla,k,r_n}(z) - T_{\nabla,k,r_0}(z)\|_{\ell^\infty([0,\pi/2])}} \\
	&\leq \int_0^{\pi/2} \labs \partial_{r_k} \varrho_{p,r_n}(x) \rabs \d x \sum_{j=1}^2 \int_{[0,1]^2} \labs \dot{\ell}_{r_n,j}(x_1,x_2) - \dot{\ell}_{r_0,j}(x_1,x_2) \rabs g(x_1,x_2) \d\sigma(x_1,x_2) \\
	&\qquad + 3\int_{[0,1]^2} g_k(x_1,x_2) \d\sigma(x_1,x_2) \int_0^{\pi/2} \labs \partial_{r_k} \varrho_{p,r_n}(x) - \partial_{r_k} \varrho_{p,r_0}(x) \rabs \d x.
\end{align*}
\bgroup
\color{black}
Using the fact that for any $r \in \param$
\begin{align*}
	\ell_r(x_1,x_2) &= x_1 + x_2 - \Lambda_r([0,x_1] \times [0,x_2]) \\
	&= x_1 + x_2 - \iint_{[0,x_1] \times [0,x_2]} \lambda_r(x,y) \d x \d y \\
	&=  x_1 + x_2 - \iint_{[0,x_1] \times [0,x_2]} \frac{xy}{x^2+y^2} \|(x,y)\|_p^{-1} \varphi_{p,r}(\arctan(y/x)) \d x \d y,
\end{align*}
where the last equality makes use of~\eqref{eq:densities_relation}. A change of variable then leads to
\begin{align*}
	\dot{\ell}_{r,1}(x_1,x_2)
	&= 1 - \int_0^{\arctan(\tfrac{x_2}{x_1})} \frac{\tan\theta}{(1+(\tan\theta)^p)^{1/p}} \varphi_{p,r}(\theta) \d\theta \quad \text{and} \\
	\dot{\ell}_{r,2}(x_1,x_2)
	&= 1 - \int_{\arctan(\tfrac{x_2}{x_1})}^{\pi/2} \frac{\tan(\tfrac{\pi}{2}-\theta)}{(1+(\tan(\tfrac{\pi}{2}-\theta)^p)^{1/p}} \varphi_{p,r}(\theta) \d\theta.
\end{align*}
Point 2 of Assumption~\ref{ass:density_reg} then leads to
\[
	\dot{\ell}_{r_n,j}(x_1,x_2) \to \dot{\ell}_{r_0,j}(x_1,x_2), \qquad \forall (x_1,x_2) \in [0,1]^2, \ j \in \{1,2\}
\]
as $n \to \infty$. The dominated convergence theorem permits to conclude that the first term in the upper bound converges to zero.
\egroup
Thanks to point 1 of Assumption~\ref{ass:modelAndEstimator} 
The second term also converges to zero by point 2 of Assumption~\ref{ass:density_reg}, showing that condition~\eqref{eq:operator_convergence} is satisfied for $T_{\nabla,r_0}$.

Combining all the above arguments permit to conclude the proof of Lemma~\ref{lem:cont_limit_law}.
\end{proof}

If $(r_n)_{n \in \N} \subset \param$ is a sequence such that $r_n \ra r_0$ in $\R^m$ as $n \ra \infty$ for some interior point $r_0 \in \param$, Lemma~\ref{lem:cont_limit_law} implies pointwise convergence of the cumulative distribution associated to $L_{r_n}$ to the one of $L_{r_0}$ which is continuous. Since it is also strictly increasing, we get the pointwise convergence $\quantile_{r_n}(1-\alpha) \ra \quantile_{r_0}(1-\alpha)$ for any $\alpha \in (0,1)$. The continuous mapping then permits to conclude that
\[
	\quantile_{\est}(1-\alpha) \xrightarrow{\PP} \quantile_{r_0}(1-\alpha), \qquad \alpha \in (0,1),
\]
for any consistent estimator $\est$ of $r_0$. The local uniformity follows from the continuity and monotonicity of the map $\alpha \in (0,1) \mapsto \quantile_{r_0}(1-\alpha)$.

\section{Sampling from the asymptotic distribution of the test statistic}
\label{sec:app:simus}

In this section, we briefly describe how to generate samples from the asymptotic distribution of $T_n$ in Theorem~\ref{thm:asymptoticTS}, as it could be of independent interest. The underlying principle is not new and has been used, e.g., in~\cite{einmahl2001} to draw confidence bounds for the empirical angular measure based on its asymptotic distribution. However, no detailed procedure could be found in the literature and our aim is to fill in this gap.

As in the proof of Lemma~\ref{lem:cont_limit_law} in Appendix~\ref{sec:app:cont_quantiles}, the key observation is that the random variable $L_{r_0}$ appearing as the limit in distribution of $T_n$ can be written as
\[
	L_{r_0} = \int_0^{\pi/2} \labs X_{r_0}(\theta) \rabs \, q(\theta) \, \d\theta,
\]
for some process $X_{r_0} \in \ell^\infty([0,\pi/2])$ which is a linear transformation of the Gaussian process $W_{\Lambda_{r_0}}$ introduced before Theorem~\ref{thm:asymptoticsEAM}. From now on, we will omit the subscript $r_0$ to alleviate the notation. The random variable $L$ will be approximated by the Riemann sum
\[
	L \approx \frac{1}{N} \sum_{i=1}^N |X(\theta_i)| \, q(\theta_i) 
\] 
for a large equidistant grid $0 < \theta_1 < \ldots < \theta_N < \pi/2$. Typically, we choose $N = 1\, 000$. Hence, to obtain a large number $B$ of simulated values for $L$, it suffices to generate $B$ trajectories of the process $X$ on the chosen grid. In Section~\ref{sec:simulations}, we choose $B=2\, 000$ while the $p$-value computations in Section~\ref{sec:application} are based on $B=4\, 000$ replications. We give an overview of the procedure below.

The first step is to obtain a realization of the underlying process $W_\Lambda$ by considering a sufficiently large rectangular grid of points $\{(x_i,y_j) : i,j = 1,\ldots,M\}$ on $\spacerv$ of the form
\begin{equation*}
	x_i = (i-1) h \qquad \text{and} \qquad y_j = (j-1) h,
\end{equation*}
for a small step size $h > 0$. Typically, we choose $h = 0.05$ and $M = 1000$. These points are used to define the cells
\begin{equation*}
	C_{ij} \egdef \lacc (x,y) \in \spacerv : x \in [x_i, x_{i+1}) \text{ and } y \in [y_j, y_{j+1}) \racc, \qquad i,j \in \{1, \ldots, M-1\},
\end{equation*}
which are considered as an approximate disjoint covering of the space $\spacerv$. The fact that the cells do not cover points $(x,y) \in \spacerv$ for which $x \vee y > (M-1) h$ is not supposed to affect the results too much as the variance of $W_\Lambda$ is negligible on this part of the space.

Associated with those cells, we generate a matrix $W = (W_{ij})_{i,j = 1, \ldots, M-1}$ of independent centered Gaussian random variables with $\Var(W_{ij}) = \Lambda(C_{ij})$. The property $\EE[W_\Lambda(C) W_\Lambda(C')] = 0$ for disjoint Borel sets $C,C' \subset \spacerv$ justifies the approximation
\begin{equation}
\label{eq:process_disc}
	\tW_\Lambda(C) \egdef \sum_{i=1}^{M-1} \sum_{j=1}^{M-1} W_{ij} \, \1\{(x_i,y_j) \in C\}
\end{equation}
for any Borel set $C \subset \spacerv$ to obtain a realization of the process $W_\Lambda$. 

Of particular interest are sets of the form $C = C_{p,\theta}$ for some $1 \leq p \leq \infty$ and $\theta \in [0,\pi/2]$, for which formula~\eqref{eq:process_disc} simplifies to
\begin{align*}
	\tW_\Lambda(C_{p,\theta}) 
	&= \sum_{i=1}^{M-1} \sum_{j=1}^{M-1} W_{ij} \, \1 \{y_j \leq \min(y_p(x_i), x_i\tan\theta)\} \\
	&= \sum_{i=1}^{M-1} \sum_{j=1}^{ \lc \tfrac{y_p(x_i)}{h} + 1 \rc \wedge (M-1)} W_{ij} \, \1 \lacc (j-1) \leq (i-1) \tan \theta \racc,
\end{align*}
where $[x]$ denotes the integer part of $x \in \R$. Similar approximations based on~\eqref{eq:process_disc} can be found for the marginal processes $W_1$ and $W_2$, leading to
\[
	\tW_1(x) = \sum_{i=1}^{\lc \frac{x}{h} + 1 \rc \wedge (M-1)} W_{i \bullet} \quad \text{and} \quad
	\tW_2(y) = \sum_{j=1}^{\lc \frac{y}{h} + 1 \rc \wedge (M-1)} W_{\bullet j},
\]
where $W_{i \bullet} \egdef \sum_{j=1}^{M-1} W_{ij}$ and $W_{\bullet j} \egdef \sum_{i=1}^{M-1} W_{ij}$ for $(i,j) \in \{1,\ldots, M-1\}^2$. For $C = A_{(x,y)}$ as introduced in Assumption~\ref{ass:stdf_expansion} and appearing in the formula for $X(\theta)$, we get
\[
	\tW_\Lambda \lp A_{(x,y)} \rp = \tW_1(x) + \tW_2(y) - \sum_{i=1}^{\lc \frac{x}{h} + 1 \rc \wedge (M-1)} \sum_{j=1}^{\lc \frac{y}{h} + 1 \rc \wedge (M-1)} W_{ij}.
\] 

As the random variable $X(\theta)$ consists of a linear transformation of the process $W_\Lambda$ evaluated on sets of the above forms, these formulas will lead to a convenient approximation for the random vector $\lp X(\theta_1), \ldots, X(\theta_N) \rp$ as scalar products between $W$ and weights, possibly depending on $\theta$, which can be computed independently of the generation of $W$, before the evaluation of the scalar products, which can be done in an efficient way using parallel computing. More details can be found on the GitHub repository of the first author \footnote{E-mail \href{mailto:stephane.lhaut@uclouvain.be}{stephane.lhaut@uclouvain.be} if the repository is not yet available at the time of reading.}.

\end{document}